\documentclass[11pt]{amsart}
\usepackage[english]{babel} 
\usepackage{url}
\usepackage{hyperref}
\usepackage[mathscr]{eucal}
\usepackage{amssymb,amsmath,amsthm}
\usepackage{xypic,comment,braket}
\usepackage{xy}
\usepackage{frcursive}
\newtheorem{theorem}{Theorem}[section]

\newtheorem{proposition}[theorem]{Proposition}
\newtheorem{definition}[theorem]{Definition}
\newtheorem{conjecture}[theorem]{Conjecture}
\newtheorem{hypothesis}[theorem]{Hypothesis}
\newtheorem{remark}[theorem]{Remark}
\newtheorem{remarks}[theorem]{Remarks}
\def\Q{\mathbb{Q}}
\def\Z{\mathbb{Z}}
\def\F{\mathbb{F}}
\let\ds=\displaystyle

\def\Cl{{\mathcal C}\hskip-2pt{\ell}}
\def\order{\raise1.5pt \hbox{${\scriptscriptstyle \#}$}}
\def\prd{\displaystyle\mathop{\raise 2.0pt \hbox{$\prod$}}\limits}
\def\sm{\displaystyle\mathop{\raise 2.0pt \hbox{$\sum$}}\limits}

\begin{document}

\markboth{Georges Gras}{}

\title[Successive maxima of class numbers]
{Successive maxima of the non-genus part \\ of class numbers}


\author{Georges Gras}
\address{Villa la Gardette, 4 Chemin Ch\^ateau Gagni\`ere,
F-38520 Le Bourg d'Oisans}
\email{g.mn.gras@wanadoo.fr}
\date{November 28, 2019}

\begin{abstract} Some PARI programs have bringed out a 
property for the non-genus part of the class number 
of the imaginary quadratic fields, with respect to 
$(\sqrt D\,)^{\varepsilon}$, where $D$ is the absolute value of the 
discriminant and $\varepsilon \in ]0, 1[$, in relation with the 
$\varepsilon$-conjecture. The general Conjecture 
\ref{mainconj}, restricted to quadratic fields, states that, for  
$\varepsilon \in ]0, 1[$, the successive maxima, as $D$ increases,
of $\frac{H}{2^{N-1} \cdot (\sqrt D\,)^{\varepsilon}}$, where $H$ is the 
class number and $N$ the number of ramified primes, occur 
only for prime discriminants (i.e., $H$ odd); we perform computations 
giving some obviousness in the selected intervals. For degree 
$p>2$ cyclic fields, we define a ``mean value'' of the non-genus parts of 
the class numbers of the fields having the same conductor and obtain 
an analogous property on the successive maxima. 
In Theorem \ref{thm} we prove, under an assumption (true for $p=2, 3$), 
that the sequence of successive 
maxima of $\frac{H}{p^{N-1} \cdot (\sqrt D\,)^{\varepsilon}}$ is infinite.
Finally we consider cyclic or non-cyclic abelian 
fields of degrees $4, 8, 6, 9, 10, 30$ to test the Conjecture \ref{mainconj}. 
The successive maxima of $\frac{H}{(\sqrt D\,)^{\varepsilon}}$ are also analyzed.
\end{abstract}

\maketitle

\vspace{-1.0cm}
\tableofcontents

\vspace{-1.0cm}
\section{Introduction and main conjecture} 
Let ${\mathcal K}_G^s$ be the set of Galois number fields of 
{\it abelian} Galois group $G$ and signature $s$; for $K \in {\mathcal K}_G^s$, 
of discriminant $D_K$, let $H_K$ be the class number of $K$ and $g_{K/\Q}^{}$ 
the ``genus part'' of $H_K$ (formula \eqref{genusformula}).
The purpose of this paper is to suggest that for any $\varepsilon \in ]0, 1[$,
the successive maxima of a suitable mean of the $\frac{H_K}
{\raise2.5pt \hbox{$g_{K/\Q}^{}$} \cdot (\sqrt{\vert D_K \vert}\,)^{\varepsilon}}$
over the set ${\mathcal F}_{G,f}^s$ of fields $K \in {\mathcal K}_G^s$ of 
conductor~$f$, have some 
smooth behavior, regarding the number $N:= \omega(f)$ of ramified primes, 
these sets of fields being ordered by increasing conductors $f$; of course,
for $G\simeq \Z/2\Z$ and $s$ fixed, a quadratic field of conductor $f$ is unique.

\smallskip
By ``successive maxima'' of a sequence $(X_i)_{i \geq 1}$ of positive 
numbers, we mean the sequence $(Y_n)_{n \geq 1}$ inductively defined by:
$Y_1=X_1$, $Y_{n+1}=X_{i_{n+1}}$, for the smallest index $i_{n+1}$ such that
$X_{i_{n+1}} > Y_n$; thus $Y_{n+1} > Y_n$ for all $n$,
the sequence $(Y_n)$ being finite or infinite. 

\smallskip
Such a $Y_n$
will also be called a {\it local maximum}.

\smallskip
For instance, the case of imaginary quadratic fields will give the following 
details of computation with $\varepsilon =0.05$ (see \S\,\ref{P1}). The 
sequence $X := \frac{H}{2^{N-1} \cdot (\sqrt{D}\,)^{\varepsilon}}$
is indexed by the absolute value of the discriminants and the 
successive maxima $Y$ are written with {\footnotesize $<< \cdots >>$}
(all are with $N=1$):

\smallskip
\footnotesize
\begin{verbatim}
<<D_K=-3    Y=0.972908>>   N=1    H=1
  D_K=-4    X=0.965936     N=1    H=1
  D_K=-7    X=0.952516     N=1    H=1
  D_K=-8    X=0.949342     N=1    H=1
  D_K=-11   X=0.941814     N=1    H=1
  D_K=-15   X=0.934539     N=2    H=2
  D_K=-19   X=0.929033     N=1    H=1
  D_K=-20   X=0.927842     N=2    H=2
<<D_K=-23   Y=2.773818>>   N=1    H=3
  D_K=-24   X=0.923622     N=2    H=2
  D_K=-31   X=2.753196     N=1    H=3
  D_K=-35   X=0.914951     N=2    H=2
  D_K=-39   X=1.824960     N=2    H=4
  D_K=-40   X=0.911902     N=2    H=2
  D_K=-43   X=0.910255     N=1    H=1
<<D_K=-47   Y=4.541167>>   N=1    H=5
  D_K=-51   X=0.906380     N=2    H=2
  D_K=-52   X=0.905941     N=2    H=2
  D_K=-55   X=1.809343     N=2    H=4
  D_K=-56   X=1.808528     N=2    H=4
  D_K=-59   X=2.709255     N=1    H=3
  D_K=-67   X=0.900218     N=1    H=1
  D_K=-68   X=1.799771     N=2    H=4
<<D_K=-71   Y=6.292403>>   N=1    H=7
  D_K=-79   X=4.482593     N=1    H=5
  D_K=-83   X=2.686236     N=1    H=3
  (...)
  D_K=-163  X=0.880430     N=1    H=1
  D_K=-164  X=3.521185     N=2    H=8
<<D_K=-167  Y=9.678872>>   N=1    H=11
  D_K=-168  X=0.879766     N=3    H=4
  (...)
\end{verbatim}

\normalsize
For short, one can say that the expression under consideration
is locally maximum only when $N_K$ ($K \in{\mathcal F}_{G,f}^s$)
is minimum, which is natural in some sense since, as we know,
$N_K \gg 0$ gives a large genus number dividing $H_K$
(and a large discriminant $D_K$). The study performed in \cite{Gr3}, 
for the degree $p$ cyclic fields, shows that the genus part, and more 
precisely the $p$-part of $H_K$, may be an obstruction regarding the 
$\varepsilon$-conjectures on the behavior of 
$\frac{H_K}{(\sqrt{\vert D_K \vert}\,)^{\varepsilon}}$ as 
$\vert D_K \vert \to \infty$, at least for sparse families 
of fields when considering the Koymans--Pagano density 
results \cite{KP} (see Section \ref{para2}).

\begin{remark}{\rm The genus formula \eqref{genusformula}, recalled below, yields, 
in the abelian case, $g_{K/\Q}^{} = \frac{1}{[K : \Q]} \cdot \prod_\ell \, e_\ell$, where
$e_\ell$ is the ramification index of $\ell$; it depends simply on the number $N_K$ 
of ramified primes only in the case $G \simeq \Z/p\Z$ where $g_{K/\Q}^{}=p^{N_K-1}$.
Otherwise, many ``subcases'' may occur, which make more difficult some statements.
Indeed, it is clear that the structure of $G$ has a strong influence on 
$N_K= \omega(f_K)$, thus on $g_{K/\Q}^{}$; but the non-genus part
of $H_K$ is exactly the ``non-predictible'' part of the class number.

\smallskip
The difficulty is to take into account these remarks and to give precise 
and relevant definitions then a right analysis of the growth of these functions 
of the class numbers regarding $(\sqrt{\vert D_K \vert}\,)^{\varepsilon}$; 
so we shall restrict ourselves to some arbitrary contexts, hoping that some 
general properties will emerge from these experiments and will strengthen the 
Conjecture \ref{genconj}.}
\end{remark}

\subsection{Genus field and genus number $g_{K/\Q}^{}$}
For simplicity, to recall genus theory, we restrict ourselves 
to the base field $k=\Q$.
The genus field, according to an extension $K/\Q$,
is the maximal subextension ${\mathcal H}_{K/\Q}$ 
of the Hilbert class field ${\mathcal H}_K$ of $K$ (in the restricted sense) 
equal to the compositum of $K$ with an abelian extension $A$ of $\Q$.
Clearly we can take $A={\mathcal H}_K^{\rm ab}$ (the maximal abelian 
subextension of ${\mathcal H}_K$):
\unitlength=0.80cm
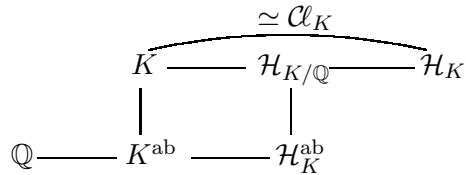
\medskip
$$\vbox{\hbox{\hspace{-1.0cm} \begin{picture}(11.5,2.8)
\put(7.2,2.0){\line(1,0){1.4}}
\put(4.5,2.0){\line(1,0){1.3}}
\put(4.9,0.5){\line(1,0){1.3}}
\put(2.35,0.5){\line(1,0){1.3}}
\put(4.05,0.9){\line(0,1){0.75}}
\put(6.55,0.9){\line(0,1){0.75}}
\bezier{350}(4.2,2.3)(6.5,2.8)(8.8,2.3)
\put(6.0,2.7){$\simeq \Cl_K$}
\put(8.7,1.9){${\mathcal H}_K$}
\put(6.0,1.9){${\mathcal H}_{K/\Q}$}
\put(3.9,1.9){$K$}
\put(6.3,0.4){${\mathcal H}_K^{\rm ab}$}
\put(1.9,0.4){$\Q$}
\put(3.8,0.4){$K^{{\rm ab}}$}
\end{picture} }} $$
\unitlength=1.0cm
\noindent
The genus number is $g_{K/\Q}^{} := [{\mathcal H}_{K/\Q} :K]$.
In the general Galois case, we have (in the restricted sense for class groups):
\begin{equation}\label{genusformula}
g_{K/\Q}^{} = \ds \frac{\prod_\ell \, e^{\rm ab}_\ell}
{[K^{\rm ab} : \Q]},
\end{equation}

\noindent
where $e^{\rm ab}_\ell$ is the ramification index of the prime 
$\ell$ in $K^{\rm ab}/\Q$. A general non-Galois formula does exist 
(e.g., \cite[Theorem IV.4.2 \& Corollaries]{Gr1}).

\smallskip
We shall call $[{\mathcal H}_K : {\mathcal H}_{K/\Q}]$ the non-genus part of the
class number $H_K := \order \Cl_K = [{\mathcal H}_K : K]$. So
we shall put, in the abelian case:
$$h_K = \frac{H_K}{g_{K/\Q}^{}} = \frac{H_K}{[{\mathcal H}_{K/\Q} :K]} = 
\frac{[K : \Q] \cdot H_K}{\prod_\ell \, e_\ell}. $$

 When $K/\Q$ is cyclic, the genus number 
$g_{K/\Q}^{}$ is equal to the number of invariant classes by 
$G:={\rm Gal}(K/\Q)$ (Chevalley's formula \cite{Che}).

\subsection{Practical framework -- Main conjecture} 
Since the classification of Galois number fields, of given non-abelian 
Galois group, by increasing conductors or discriminants, is too complex, it
is assumed in all the sequel that the extensions $K/\Q$ are abelian of 
Galois group $G$ (and of signature $s$); which allows a 
classification using the conductors to list the suitable fields, even if a 
conductor may correspond to a finite number of fields. 

\smallskip
For prime degree $p$ cyclic fields, the classifications via increasing discriminants 
is equivalent to that using the conductors $f_K$ since $D_K=f_K^{p-1}$.
 
\begin{definition}\label{def}
Denote by ${\mathcal K}_G^s$ the set of Galois number fields $K$, of 
given abelian Galois group $G$ and signature $s=\pm 1$ (i.e., real or imaginary 
fields). 

\smallskip
The conductor of the field $K$ is denoted $f_K$ and its discriminant
$D_K$; then $\omega(f_K)= \omega(D_K) =: N_K$ denotes the number of
ramified primes. We then define ${\mathcal F}_{G,f}^s :=
\{K \in {\mathcal K}_G^s, \ \, f_K=f\}$.

\smallskip
If $G$ is cyclic of order $d$ and if $s$ is fixed, we shall put 
${\mathcal K}_G^s = {\mathcal K}_d^-$ or ${\mathcal K}_d^+$
and in the same way for ${\mathcal F}_{G,f}^s =: {\mathcal F}_{d,f}^s$,
$s = \pm1$.

\smallskip
We shall in general remove the index $K$ for the above objects,
giving the notations $f$, $D$, $N$, $H$, $h$.
\end{definition}

If $G \simeq \Z/p\Z$, the genus number only depends on the number 
$N$ of prime divisors of the conductor $f$ and one gets:
$$h=\frac{H}{p^{N-1}}. $$

It is clear that the structure of abelian group of $G$ leads to some 
constraints on ${\rm Gal}(\Q(\mu_f)/\Q) \simeq (\Z/f\,\Z)^\times$, then 
on $N$ (e.g., if for some prime number $p$, ${\rm dim}_{\F_p}(G/G^p) \geq 2$,
the conductor $f$ of $K$ cannot be a prime number $q\ne p$).
Moreover, the wild ramification gives some trouble as the case
of abelian $2$-extensions with even conductors since the $2$-rank of
$(\Z/2^n\Z)^\times$ ($n\geq 2$) may be $1$ or $2$; the case of $p>2$ is
easier since $(\Z/p^n\Z)^\times \otimes \Z_p$ ($n\geq 2$) is cyclic.
The mixed case where $G$ is a product of $p$-groups for at least
two distinct primes $p$ depends on several sub-cases.

\smallskip
If two fields, in the same family, have the same discriminant, they 
have the same conductor, except possibly for the wild parts; but in our
examples, we will have equality.

\smallskip
Of course, a list by increasing discriminants may be very different from 
the same list by increasing conductors; for instance, the 
degree $4$ cyclic fields $K_1$ and $K_2$ such that
$D_{K_1}=5^3 \,17^2=36125$ and $D_{K_2}=17^3 \,3^2=44217 > D_{K_1}$
are of conductors $f_{K_1}=85$ and $f_{K_2}=51$, respectively.

\smallskip
So, it makes sense to consider, instead, a classification by increasing 
conductors, which seems to be a better classification with regard to
Galois structure and ramification. But, considering the 
$\varepsilon$-conjectures, the class numbers are usually 
compared with suitable functions of the discriminant, as 
$(\sqrt{\vert D_K \vert}\,)^{\varepsilon}$.

\smallskip
Whence the following conjecture, strengthened by the various 
computations that we shall perform in the forthcoming sections:

\begin{conjecture}\label{genconj}
Let $G$ be an abelian group and consider the family ${\mathcal K}_G^s$
(see Definition \ref{def}).
Let  ${\mathcal F}_{G, f}^s$ (resp. ${\mathcal F}'^s_{G, f}$) be the set of abelian fields 
$K \in {\mathcal K}_G^s$ of conductor $f$ (resp. of any conductor $f' \mid f$).
For any $\varepsilon \in ]0, 1[$, let:
\begin{equation*}
\begin{aligned}
C_{\varepsilon}({\mathcal F}_{G, f}^s) & := \Bigg[ \prod_{K \in {\mathcal F}_{G, f}^s}\ 
\frac{\order \Cl_K}{\raise2.5pt \hbox{$g_{K/\Q}^{}$} \cdot 
(\sqrt{ \vert D_K \vert}\,)^\varepsilon} \Bigg]^{\frac{1}{\order {\mathcal F}_{G, f}^s}}, \\
C'_{\varepsilon}({\mathcal F}'^s_{G, f}) & := \Bigg[ \prod_{K \in {\mathcal F}'^s_{G, f}}\ 
\frac{\order \Cl_K}{\raise2.5pt \hbox{$g_{K/\Q}^{}$} \cdot 
(\sqrt{ \vert D_K \vert}\,)^\varepsilon} \Bigg]^{\frac{1}{\order {\mathcal F}'^s_{G, f}}},\\
C_{\varepsilon}(K) & :=
\ds  \frac{\order \Cl_K} {\raise2.5pt \hbox{$g_{K/\Q}^{}$} \cdot 
(\sqrt{ \vert D_K \vert}\,)^\varepsilon}, \ \, \hbox{for any $K \in  {\mathcal K}_G^s$}. 
\end{aligned}
\end{equation*}

Let $r(G)$ be the minimal number of generators of $G$ and let
$\order G = \prod_{i=1}^t p_i^{n_i}$ be the prime decomposition of the 
order of $G$, then put $R(G) := \sum_{i=1}^t  n_i$.

\smallskip
For instance $r(\Z/4\Z)=r(\Z/6\Z)=1$ and $R(\Z/4\Z)=R(\Z/6\Z)=2$,\,\ldots

\smallskip
We then have conjecturaly\,\footnote{Recall that when $\order G$ is odd, 
the prime $2$ is unramified in the abelian extension $K/\Q$. When $\order G$ 
is even and $2$ ramifies with $8 \!\mid\! f$, one must preferably replace $\omega(f)$
by  $\omega(f) +1$; for simplicity, we shall neglect this exception in the statements
and comments.}:

\smallskip
(i) The successive maxima of $C_{\varepsilon}({\mathcal F}_{G, f}^s)$,
$C'_{\varepsilon}({\mathcal F}'^s_{G, f})$, $C_{\varepsilon}(K)$, as $f$ 
increases from its minimum, occur for conductors $f$ such that 
$r(G) \leq \omega(f)  \leq R(G)$, with probability $1$.\,\footnote{This 
means the exactness except possibly for
a set of conductors of zero density. However, the probabilistic 
behaviour of $\omega(f)$ in $[r(G),R(G)]$ is unclear.} This property
does not depend on the choice of $\varepsilon \in ]0,1[$.

\smallskip
(ii) For any $\varepsilon \in ]0, 1[$, the sequences of successive 
maxima are infinite.
\end{conjecture}

\begin{remarks}{\rm 
(i) The functions $C_{\varepsilon}$ and $C'_{\varepsilon}$ are 
multiplicative means over the fields of same conductor $f$
(resp. of any conductor $f' \mid f$). 
The computation of the successive maxima of the
function $C_{\varepsilon}$ over ${\mathcal K}_G^s$
does not consider any mean and is equivalent to take the maximum
of $\frac{\order \Cl_K} {\raise2.5pt \hbox{$g_{K/\Q}^{}$} \cdot 
(\sqrt{ \vert D_K \vert}\,)^\varepsilon}$ over ${\mathcal F}^s_{G, f}$,
for each $f$, then the successive maxima of these numbers.

\smallskip
One has $\order {\mathcal F}_{G, f}^s=1$, for all $f$, only for the family 
of quadratic fields, in which case, $C_{\varepsilon}({\mathcal F}_{G, f}^s) 
= C_{\varepsilon}(K)$ for the unique $K \in {\mathcal F}_{G, f}^s$.

\smallskip
(ii) The sequences are not necessarily the same as $\varepsilon$ varies, 
but the practice shows that they differ very little.

\smallskip
(iii) It is possible that ``except for a number of conductors of density $0$'' 
may be replaced by ``except a finite number of conductors'', but the numerical
experiments are not sufficient to be more precise.}
\end{remarks}

These kind of functions are related to some less known results
concerning minorations of class numbers and to the $\varepsilon$-conjectures 
that we recall now.

\section{Classical results and $\varepsilon$-conjectures} \label{para2}

\subsection{Upper and lower bounds for class numbers}\label{strong}
The order of magni\-tude of the class number $\order \Cl_K$ 
of a number field $K$ is very irregular but is in part governed by the absolute value 
of the discriminant $D_K$ and the signature of the field $K$; of course, 
this is to be precised and a method is to compare the class number with 
$(\sqrt {\vert D_K \vert}\,)^\varepsilon$, for any $\varepsilon \in ]0, 1[$ fixed.

\smallskip
Using the family ${\mathcal K}_p^s$ of degree $p$ cyclic fields 
for simplicity ($p \geq 2$ prime), let's recall that there exist
an absolute constant ${\mathscr C}$ and, for all $\eta > 0$, 
${\mathscr C}_{\eta}$ such that (Siegel, Landau;
see, e.g., \cite[Chapter 4, \S\,1, Theorem 4.4]{N}):
\begin{equation}\label{0}
\begin{aligned}
(i) \hspace{2.5cm} & \order \Cl_K \leq {\mathscr C} \cdot \sqrt{\vert D_K \vert\,} 
\cdot {\rm log} (\sqrt{\vert D_K \vert}\,),\hspace{2.5cm} \\
(ii) \hspace{2.5cm} & \order \Cl_K \leq {\mathscr C}_{\eta} 
\cdot (\sqrt{\vert D_K \vert}\,)^{1+ \eta}.\hspace{2.5cm}
\end{aligned}
\end{equation}

Under GRH  (Ankeny--Brauer--Chowla; see more comments in 
\cite{Da,DaKM,Du}), we have the following result for quadratic fields 
which will be of some significance for successive maxima:
\begin{proposition}\label{00}
For any $\eta >0$, there exist infinitely many $K \in {\mathcal K}_2^s$ 
such that $\order \Cl_K \geq (\sqrt{\vert D_K \vert}\,)^{1- \eta}$.
\end{proposition}

Moreover, Littlewood proves the existence of constant 
${\mathscr C}'$, ${\mathscr C}''$, such that:
\begin{equation}\label{wood}
\begin{aligned}
(i)\  \order \Cl_K & \leq {\mathscr C}' \cdot \sqrt{\vert D_K \vert} 
\cdot {\rm log}_2 (\sqrt{\vert D_K \vert}\,), \,\ 
\hbox{for all $K \in {\mathcal K}_2^-$,} \\
(ii)\  \order \Cl_K&  \geq {\mathscr C}'' \cdot \sqrt{\vert D_K \vert} 
\cdot {\rm log}_2 (\sqrt{\vert D_K \vert}\,), \ 
\hbox{for infinitely many  $K \in {\mathcal K}_2^-$.} 
\end{aligned}
\end{equation}
 
For more information and generalization of such phenomena, including
real fields, see \cite{Da, DaKM,Du,GS,Lam,L0,L1,L2,MW,R}.

\subsection{The $\varepsilon$-conjectures}
The concept of $\varepsilon$-conjecture comes from the work of 
Ellenberg--Venkatesh \cite{EV}, in close relation with the heuristics and 
conjectures of Cohen--Lenstra--Martinet. Many developments have followed 
as \cite{AM,DJ,EPW,FW,Gr3,KP,Ma1,Ma2,P-TB-W,W}, to give an order 
of magnitude of various invariants attached to the class group of a number 
field $K$, according to the function $(\sqrt{\vert D_K \vert}\,)^\varepsilon$ 
of its discriminant, for any $\varepsilon > 0$.

\smallskip
We shall now emphasize some of these $\varepsilon$-conjectures and 
precise our purpose, although this paper is not concerned by the $q$-class
groups $\Cl_K \otimes \Z_q$, for each prime $q$, except that genus theory 
involves the $q$-parts of the class group when $q \mid d$, 
as for the degree $p$ cyclic extensions for which  the $p$-Sylow 
subgroup $\Cl_K \otimes \Z_p$ may be very large:

\smallskip
(i) The $p$-rank $\varepsilon$-conjecture for number fields 
claims that for all $\varepsilon > 0$:
$$\order (\Cl_K \otimes \F_p) \ll_{d, p, \varepsilon} \cdot 
(\sqrt{\vert D_K \vert}\,)^\varepsilon, \ \ \hbox{ for all $K \in {\mathcal K}^s_d$}. $$

We have proved the $p$-rank $\varepsilon$-conjecture for 
the family of degree $p$ cyclic extensions \cite[Theorem 2.5]{Gr3}.

\smallskip
Many authors study and prove for ${\mathcal K}^s_d$ inequalities of the form:
\begin{equation}\label{c}
\order (\Cl_K \otimes \F_ p) \ll_{d, p, \eta} 
\cdot (\sqrt{\vert D_K \vert}\,)^{\,c+ \eta}, \ \, \eta>0,
\end{equation} 

with a constant $c \in ]0, 1[$ as small as possible 
\big(e.g., $c=1-\frac{1}{p\,(d-1)}$\big) since the 
equality does exist for $c=1$ (relation \eqref{0}\,(ii)).

\smallskip
(ii) The strong $(p, \varepsilon)$-conjecture that we have considered 
in \cite{Gr3} is:
\begin{equation}\label{sc}
\order (\Cl_K \otimes \Z_p) \ll_{d, p, \varepsilon} \cdot 
(\sqrt{\vert D_K \vert}\,)^\varepsilon, \ \,\hbox{for all $K \in {\mathcal K}^s_d$}. 
\end{equation} 

The results of density of \cite[Theorem 1.1]{KP} prove, under GRH, that the strong 
$(p, \varepsilon)$-conjecture for degree $p$ cyclic fields is true, {\it except possibly 
for sparse subfamilies of fields of zero density}, probably in relation with 
that we bringed out in \S\,\ref{strong}.

\subsection{Non-genus number in the degree $p$ cyclic case}
We have shown in \cite{Gr3}, using the family of degree $p$ 
cyclic extensions $K/\Q$, that the genus subgroup of $\Cl_K$,
may be an obstruction or at least a difficulty, to prove 
inequality \eqref{sc} and one may for instance replace the whole 
class group by its non-genus part in the above 
$(p, \varepsilon)$-conjecture.  

\smallskip
Consider a family ${\mathcal K}^s_p$ of prime degree 
$p \geq 2$ cyclic fields $K$.

\smallskip
Let $G = {\rm Gal}(K/\Q) =: \langle \sigma \rangle$ and let 
$H$, be the class number of $K$ in the restricted sense when
this makes sense (e.g., case of real quadratic fields).

\smallskip
We consider the quotient $\Cl_K/\Cl_K^G \simeq \Cl_K^{1-\sigma}$ 
(of order $h := \frac{H}{g_{K/\Q}^{}}$), of $\Cl_K$ by the subgroup 
of invariant (or ambiguous) classes; this group measures the set of
``exceptional $p$-classes'' and that of classes of prime to $p$ orders.

\smallskip
So we may consider the non-genus $(p,\varepsilon)$-conjecture for this family:
\begin{equation}\label{genus}
\order (\Cl_K \otimes \Z_p)^{1-\sigma} \ll_{p,\varepsilon} \cdot 
(\sqrt {\vert D_K \vert}\,)^\varepsilon,
\end{equation}

except for subfamilies of zero density.

\subsection{Successive maxima -- Main Theorem for degree $p$ cyclic fields}
We have (to our knowledge) no information on the successive maxima of
$\ds \frac{\order \Cl_K}{(\sqrt {\vert D_K \vert}\,)^\varepsilon}$
or $\ds \frac{\order \Cl_K}{\raise 2pt \hbox{$g_{K/\Q}^{}$}
 \cdot (\sqrt {\vert D_K \vert}\,)^\varepsilon}$,
as $K$ varies in some defined family ordered by increasing discriminants
and considering that several fields may have the same discriminant, which
forces to give the more precise definitions used in the Conjecture \ref{genconj}.

\begin{remark}{\rm 
A main problem is then that Proposition \ref{00} (or the lower bound of \eqref{wood})
will be in contradiction with an unconditional global abelian ``strong $\varepsilon$-conjecture'',
this strong $\varepsilon$-conjecture being that for all $\varepsilon \in ]0, 1[$:
\begin{equation}\label{000}
\order \Cl_K \ll_{G, \varepsilon} \cdot 
(\sqrt {\vert D_K \vert}\,)^\varepsilon, 
\end{equation}

for the family ${\mathcal K}_G^s$ of field of abelian Galois group $G$
and signature $s$. 

\smallskip
To overcome this difficulty, we must conjecture that this 
inequality \eqref{000} holds {\it except for a subfamily of 
zero density}. This caution shall be necessary whatever the
context as shown by the work \cite{KP} of Koymans--Pagano on densities 
in the family of degree $p$ cyclic fields for which the results are true
with ``probability~$1$'', in the framework of the above description on the 
possible existence of sparse ``bad subfamilies''.}
\end{remark}

To take into account these phenomena, we base our study on the 
following which holds for $G \simeq \Z/2\Z$ and $G \simeq \Z/3\Z$:

\begin{hypothesis} \label{hypo}
We assume that Proposition \ref{00} does exist in ${\mathcal K}_G^s$ 
for all non-trivial abelian group $G$, in other words that for all $\eta \in ]0, 1[$ there 
exist infinitely many $K \in {\mathcal K}_G^s$ such that 
$\order \Cl_K \geq (\sqrt {\vert D_K \vert}\,)^{1-\eta}$.
\end{hypothesis}

\begin{remark}{\rm This statement is equivalent to: ``for  all $\eta \in ]0, 1[$ there 
exist ${\mathcal C}_\eta > 0$ and infinitely many $K \in {\mathcal K}_G^s$ such 
that $\order \Cl_K \geq {\mathcal C}_\eta \!\cdot\! (\sqrt {\vert D_K \vert}\,)^{1-\eta}$\,''.
Indeed, let $\eta \in ]0, 1[$ be given and let $\theta > 0$ be such that
$\eta_0 := \eta-\theta \in ]0, 1[$; we obtain, with $\eta_0$ and the 
existence of ${\mathcal C}_{\eta_0}$, an infinite 
sequence of fields $K_i \in {\mathcal K}_G^s$ such that:
\begin{equation*}
\begin{aligned}
{\rm log}( \Cl_{K_i}) & \geq {\rm log}({\mathcal C}_{\eta_0}) +
(1 - (\eta - \theta)) \cdot {\rm log}(\sqrt{D_{K_i}}\,) \\
& \geq {\rm log}({\mathcal C}_{\eta_0}) +
\theta \cdot {\rm log}(\sqrt{D_{K_i}}\,)  + (1 - \eta) \cdot {\rm log}(\sqrt{D_{K_i}}\,) \\
& \geq (1 - \eta) \cdot {\rm log}(\sqrt{D_{K_i}}\,)\ \,\hbox{for all $i \geq i_0$, $i_0$ suitable}.
\end{aligned}
\end{equation*}}
\end{remark}

A consequence of this is that for any $\varepsilon \in ]0, 1[$,
the sequence of successive maxima of $\ds \frac{\order \Cl_K}
{(\sqrt{ \vert D_K \vert}\,)^\varepsilon}$ is infinite since 
$\ds \frac{\order \Cl_K}{(\sqrt{ \vert D_K \vert}\,)^\varepsilon} \geq 
(\sqrt{ \vert D_K \vert}\,)^{1 - \eta - \varepsilon}$ will be unbounded,
due to infinitely many fields $K_i \in {\mathcal K}_G^s$, as soon as 
$\eta$ is chosen small enough to get $1 - \eta - \varepsilon >0$.

\smallskip
For the family of degree $p$ cyclic fields (of given signature for $p=2$), 
we will have some information (Theorem \ref{thm}) on the successive maxima of:
$$C_{\varepsilon}(K) := \ds\frac{\order\big( \Cl_K^{1-\sigma} \big)}
{(\sqrt {\vert D_K \vert}\,)^\varepsilon} =:
\ds\frac{H}{p^{N-1} \cdot (\sqrt {D}\,)^\varepsilon},
\ \, N = \omega(D), $$

for which we have proposed the Conjecture \ref{genconj}.
This is related to \eqref{genus} in the following manner for $G \simeq \Z/p\Z$:

\begin{remark}{\rm
The strong $\varepsilon$-conjecture states that,
``for almost all $K$'', the sequence 
$\frac{\order \Cl_{K}}{(\sqrt {\vert D_K \vert}\,)^\varepsilon}$ is 
bounded by a constant ${\mathscr C}_ \varepsilon$ (which 
increases as $\varepsilon$ decreases). For $\varepsilon \in ]0,1[$, this 
constant is enormous (and unbounded as $\varepsilon \to 0$).
Thus the sequence obtained by some discriminants $D_{K_i}$ giving the 
successive maxima, is limited by ${\mathscr C}_ \varepsilon$, up to 
some huge unreachable discriminant; but if the sequence of the
successive maxima is infinite, the ``exceptional cases'' of the
$\varepsilon$-conjecture may be some local maximua, but 
cannot be computed in practice when $\varepsilon$ is too small.}
\end{remark}

Using these lower bounds and our results in \cite{Gr3} proving (unconditionally) 
the $p$-rank $\varepsilon$-conjecture, we can state the following main result 
for the family ${\mathcal K}_p^s$ of prime degree $p$ cyclic fields; 
${\mathcal K}_2^s$ may be ${\mathcal K}_2^-$ or ${\mathcal K}_2^+$
then of course ${\mathcal K}_p^s = {\mathcal K}_p^+$ for $p>2$, and where
$\order {\mathcal F}_{p, f}^s = (p-1)^{N-1}$ for $N := \omega (f)$:

\begin{theorem}\label{thm}
Let $\varepsilon \in  ]0, 1 [$ be given. Under Hypothesis \ref{hypo}
in the family ${\mathcal K}_p^s$ of degree $p$ cyclic number fields 
ordered by increasing conductors $f$ (or discriminants $D = f^{p-1}$), 
the sequence of the successive maxima of:
$$C_{\varepsilon}(K) = \frac{H}{p^{N-1} \cdot (\sqrt{D}\,)^{\varepsilon}}, 
\ \ \hbox{$K \in {\mathcal K}_p^s$}, \ H := \order \Cl_K,\ N := \omega(f), $$ 

is infinite. 
\end{theorem}

\begin{proof} By assumption, for all $\eta \in ]0, 1[$, there exist an infinite 
sequence $(f_i)_{i \geq 1}$ and at least a degree $p$ cyclic fields $K_i$,
of conductor $f_i$, such~that:
$$C_{\varepsilon}(K_i) = \frac{\order \Cl_{K_i}}{p^{N_i-1}\! \cdot \!
(\sqrt{\vert D_i \vert}\,)^\varepsilon}\geq  \frac{1}{p^{N_i-1}}
(\sqrt{\vert D_i \vert}\,)^{1-\eta - \varepsilon},  N_i := \omega(f_i), 
\vert D_i \vert = f_i^{p-1}. $$ 

The claim consists in proving that the right member is unbounded, in 
other words that the sequence (which depends on the choice of $\eta$):
$$A_i^{(\eta)} := (1-\eta - \varepsilon) \cdot 
{\rm log}(\sqrt{\vert D_i \vert}\,)-(N_i-1) \cdot  {\rm log}(p)$$

is unbounded (which amounts to prove that the $N_i$ are not too large).
This is only a sufficient condition.
To simplify, we delete the index $i$, we fix $\eta$ and such an indice $i$, hence
the corresponding field $K$ of huge class number, and the
integers $N$, $f$, $D$; we put $A_i^{(\eta)} =: A$.

\smallskip
From \cite[\S\,2.3]{Gr3}, we have obtained the existence, for all $\theta>0$,
of a constant $C_{p,\theta}>0$ such that, whatever the field $K \in {\mathcal K}_p^s$:
$$p^{(p-1)(N-1)} \leq C_{p,\theta}\cdot (\sqrt{D_N}\,)^\theta,\ \, D_N=f_N^{p-1}, $$

where $f_N$ is the product of the $N$ first consecutive prime numbers congruent to
$1$ modulo $p$, so that $\vert D \vert > D_N$ since for the field $K$,
$\omega(D)=\omega(D_N)=N$. 

\smallskip
Which yields:
$$(N-1) \cdot {\rm log}(p) \leq \frac{{\rm log}(C_{p,\theta})}{p-1}
+ \frac{\theta \cdot {\rm log}(\sqrt {D_N}\,)}{p-1}. $$

Then:
\begin{equation}
\begin{aligned}
A \geq & (1-\eta - \varepsilon) \cdot {\rm log}(\sqrt{\vert D \vert}\,)
- \frac{{\rm log}(C_{p,\theta})}{p-1} - \frac{\theta \cdot {\rm log}(\sqrt {D_N}\,)}{p-1} \\
\geq & (1-\eta - \varepsilon) \cdot {\rm log}(\sqrt{\vert D \vert}\,)
- \frac{{\rm log}(C_{p,\theta})}{p-1} - \frac{\theta \cdot {\rm log}(\sqrt {\vert D \vert}\,)}{p-1} \\
\geq & \Big [ 1-\eta - \varepsilon - \frac{\theta}{p-1} \Big] \cdot {\rm log}(\sqrt {\vert D \vert}\,)
- \frac{{\rm log}(C_{p,\theta})}{p-1}.
\end{aligned}
\end{equation}

The parameters $\eta > 0$ and $\theta > 0$ may be chosen arbitrarily close to $0$.
Then as soon as $0 < \eta < 1- \varepsilon - \frac{\theta}{p-1}$, the corresponding
sequence $(A_i^{(\eta)})_i$ is unbounded, which yields the claim.
\end{proof}

\begin{remarks}{\rm
(i) It is more difficult to prove the same result for the
successive maxima of the means $C_ \varepsilon({\mathcal F}_{p, f}^s)$
(except for $p=2$) or $C'_ \varepsilon({\mathcal F}'^s_{p, f})$. 
However, since all the experiments show that the successive maxima are 
obtained with very small values of $N$, there is no doubt for the 
infiniteness of the corresponding sequences of maxima (perhaps 
for $0 < \varepsilon \ll 1$ since an $\varepsilon$ close to $0$
allows less restrictive values of $\theta$ and $\eta$ and easier 
numerical data).

\smallskip
(ii) The case $\varepsilon = 0$ makes sense for the study of the successive 
maxima of the function $C_{0}(K) := \frac{\order \Cl_K}{g_{K/\Q}^{} 
\cdot (\sqrt{\vert D_K \vert}\,)^0}$ equal to $\frac{H}{p^{N-1}}$ 
in the degree $p$ cyclic case, but there is no $0$-conjecture since 
the class number is unbounded in any infinite family ${\mathcal K}_d^s$. 
At the opposite case, for any $\eta > 0$, we know that
$\order \Cl_K \ll_{\eta} \cdot 
(\sqrt{\vert D_K \vert}\,)^{1+ \eta}$,
which will give some information (see \S\,\ref{eps=1}).

\smallskip
(iii) The inequality \eqref{c}, with $\varepsilon = c+\eta$, which is proved in 
some cases, and the relations \eqref{0} with Proposition \ref{00}, show that the 
problem of the strong $\varepsilon$-conjectures is only concerned with 
$\varepsilon \in ]0,1[$ (cf. Theorem \ref{thm}).}
\end{remarks}

\section{Study of the quadratic fields}
We shall first illustrate the property under consideration by means of the familly 
${\mathcal K}_2 = {\mathcal K}_2^- \cup {\mathcal K}_2^+$ of quadratic fields.

\subsection{Ordered list of discriminants}
Let $K \in {\mathcal K}_2^s$ identified by means of its signature and 
the absolute value $D$ of its discriminant $D_K$ (which is also its conductor $f_K$). 
We shall use first the case of imaginary quadratic fields since the class numbers are 
more rapidely computed by PARI/GP \cite{P}, which allows more important bounds
for the discriminants; then we shall see that the results are similar in the real case.

\smallskip
Any list of quadratic fields $K$ will be ordered by increasing $D$ as the following 
program does in the imaginary case, where ${\sf [bD, BD]}$ is the interval 
for ${\sf D}$, ${\sf H}$ is the class number of $K$ and ${\sf h = \frac{H}{2^{N-1}}}$,
${\sf N}$ being the number of ramified primes and ${\sf 2^{N-1}}$
the genus number (for the real case, see Section \ref{real}):

\smallskip
\footnotesize
\begin{verbatim}
{bD=1;BD=10^4;for(D=bD,BD,e2=valuation(D,2);d=D/2^e2;if(e2==1||e2>3||
core(d)!=d||(e2==0&Mod(-d,4)!=1)||(e2==2&Mod(-d,4)!=-1),next);
H=qfbclassno(-D);N=omega(D);h=H/2^(N-1);
print("D_K=",-D," H=",H," h=",h," N=",N))}

D_K=-3      H=1     h=1     N=1
D_K=-4      H=1     h=1     N=1
D_K=-7      H=1     h=1     N=1
D_K=-8      H=1     h=1     N=1
D_K=-11     H=1     h=1     N=1
D_K=-15     H=2     h=1     N=2
D_K=-19     H=1     h=1     N=1
(...)
D_K=-10000000003     H=10538     h=5269     N=2
D_K=-10000000004     H=40944     h=20472    N=2
D_K=-10000000007     H=95488     h=23872    N=3
D_K=-10000000011     H=20264     h=5066     N=3
D_K=-10000000015     H=39520     h=2470     N=5
D_K=-10000000019     H=39809     h=39809    N=1
D_K=-10000000020     H=42368     h=1324     N=6
(...)
D_K=-100000000000000000003     H=1442333424     h=360583356     N=3
D_K=-100000000000000000004     H=7297756288     h=456109768     N=5
D_K=-100000000000000000007     H=8734748716     h=2183687179    N=3
D_K=-100000000000000000011     H=2377640288     h=148602518     N=5
D_K=-100000000000000000015     H=5235364148     h=1308841037    N=3
D_K=-100000000000000000019     H=3301267440     h=1650633720    N=2
D_K=-100000000000000000020     H=3240597760     h=50634340      N=7
(...)
\end{verbatim}
\normalsize

As can be seen, the numbers $H$ and $h$ grow globally, but with many local 
decreases, the case of $N = \omega(D)$ being completely chaotic compared 
to the growth of $D$; so we intend to look at the successive extrema
(especially the successive maxima) of some 
function of these parameters.

\subsection{Conjecture \ref{genconj} for quadratic fields}\label{conj}
Let $K \in {\mathcal K}_2^s$ be any quadratic field. Let 
$G = {\rm Gal}(K/\Q) =: \langle \sigma \rangle$ and let $\Cl_K$, 
of order $H$, be the class group of $K$ (in the restricted sense for 
$K \in {\mathcal K}_2^+$).
We consider the quotient $\Cl_K/\Cl_K^G \simeq \Cl_K^{1-\sigma}$ 
(of order $h := \frac{H}{2^{N-1}}$).
Since ${\mathcal F}_{2, f}^s$ is reduced to a single field identified 
by its conductor $f =D := \vert D_K \vert$, we put:
\begin{equation}\label{defC}
C_{\varepsilon}(D) = \frac{h}{(\sqrt D\,)^\varepsilon}
 = \frac{H}{2^{N-1} \cdot (\sqrt{D}\,)^\varepsilon}, \ \, \varepsilon \in ]0, 1[.
\end{equation}

We shall compute the {\it successive maxima} of the function 
$C_{\varepsilon}(D)$ when $D$ increases.

\smallskip
The following conjecture, suggested by some computations in \cite{Gr3}, 
is the Conjecture \ref{genconj} for quadratic fields since $r(G)=R(G)=1$,
which will be strengthened from numerical experiments given in the next sections.

\begin{conjecture} \label{mainconj}
Let $\varepsilon \in ]0, 1[$. For each quadratic field $K$ (of a selected 
signature $s = \pm1$) of discriminant $D_K$, we denote by 
$N := \omega(D_K)=\omega(f_K)$ 
the number of ramified primes, by $H$ the class number in the restricted 
sense, and we put $D := \vert D_K \vert$. \par
Then, the successive maxima of $C_{\varepsilon}(D) := \ds
\frac{H}{2^{N-1} \cdot (\sqrt{D}\,)^\varepsilon}$, as $D$ 
increases from its minimum, only occur for prime 
conductors $f_K =  \ell \equiv s \pmod 4$
(i.e., $N=1$, $H$ odd), except possibly for a finite number. 
These properties do not depend on the choice of $\varepsilon$.
\end{conjecture}

The sequence of successive maxima is infinite from Theorem \ref{thm}, 
and only depends of $\varepsilon$.

\smallskip
For the imaginary quadratic fields, the numerical experiments show 
that all the successive maxima are given by prime discriminants.

\medskip
Denote by $s\, \ell_i \equiv 1 \pmod 4$ this conjectural sequence 
of prime discriminants giving these successive maxima; then:
\begin{equation}\label{li}
C_{\varepsilon}(\ell_i) = \frac{H_i}{ (\sqrt{\ell_i}\,)^\varepsilon}
\hbox{\ \  for all $i$, where $H_i = \order \Cl_{\Q(\sqrt{s\,\ell_i}\,)}$}. 
\end{equation}

As opposed to the conjectural inequality \eqref{000}
(except a set of fields of zero density), the fields
$K_i := \Q(\sqrt{s\,\ell_i}\,)$ would perhaps show some examples for the 
minorations given by Proposition \ref{00}

\subsection{Some heuristics for imaginary quadratic fields}
Using the program \ref{P1}, we start by giving some insights about
the behaviour of the class numbers in various situations.

\subsubsection{Large genus numbers}
We shall give some calculations showing that, contrary to the common sense, 
a great lot of ramified primes in $K$ (whence a large genus number) do not give 
{\it systematically} a large class number $H$ (see a numerical example in the next
subsection). 

\smallskip
For instance, we may use the family of fields 
$K=\Q(\sqrt{-2 \cdot 3 \cdot 5 \cdot 7 \ldots q_n\,})$,
using the sequence of the $n$ first odd prime numbers
or any similar one as $\Q(\sqrt{-q_{n_0}^{} \ldots q_{n_0+n}^{}\,})$. 
The following results are very convincing ($D= \vert D_K \vert$):

\footnotesize
\smallskip
\begin{verbatim}
D=[2,3] H=1
D=[2,3;3] H=2
D=[2,3;3;5] H=4
D=[2,3;3;5;7] H=8
D=[2,3;3;5;7;11] H=32
D=[2,3;3;5;7;11;13] H=128
D=[2,3;3;5;7;11;13;17] H=448
D=[2,3;3;5;7;11;13;17;19] H=2048
D=[2,3;3;5;7;11;13;17;19;23] H=10240
D=[2,3;3;5;7;11;13;17;19;23;29] H=44544
D=[2,3;3;5;7;11;13;17;19;23;29;31] H=264192
D=[2,3;3;5;7;11;13;17;19;23;29;31;37] H=1908736
D=[2,3;3;5;7;11;13;17;19;23;29;31;37;41] H=11059200
D=[2,3;3;5;7;11;13;17;19;23;29;31;37;41;43] H=76988416
D=[2,3;3;5;7;11;13;17;19;23;29;31;37;41;43;47] H=436731904
D=[2,3;3;5;7;11;13;17;19;23;29;31;37;41;43;47;53] H=3745513472
(...)
D=[3]    H=1
D=[3;5]    H=2
D=[2,2;3;5;7] H=8
D=[3;5;7;11]    H=8
D=[3;5;7;11;13]    H=96
D=[3;5;7;11;13;17]    H=256
D=[2,2;3;5;7;11;13;17;19] H=1536
D=[3;5;7;11;13;17;19;23]    H=2688
D=[3;5;7;11;13;17;19;23;29]    H=39424
D=[2,2;3;5;7;11;13;17;19;23;29;31]  H=193536
D=[2,2;3;5;7;11;13;17;19;23;29;31;37]  H=1210368
D=[2,2;3;5;7;11;13;17;19;23;29;31;37;41] H=7471104
D=[3;5;7;11;13;17;19;23;29;31;37;41;43]    H=57925632
D=[2,2;3;5;7;11;13;17;19;23;29;31;37;41;43;47]  H=364101632
D=[2,2;3;5;7;11;13;17;19;23;29;31;37;41;43;47;53] H=2738487296
D=[3;5;7;11;13;17;19;23;29;31;37;41;43;47;53;59]    H=19836665856
(...)
\end{verbatim}

\normalsize

\subsubsection{Successive maxima of $\frac{H}{2^{N-1} \cdot 
(\sqrt{\vert D_K \vert}\,)^{0.05}}$ (see Table I, Appendix \ref{Table1})}

As expected, the computation of the successive maxima
of $\frac{H}{2^{N-1} \cdot (\sqrt{\vert D_K \vert}\,)^{0.05}}$ in an interval 
containing such a huge discriminant gives much larger non-genus class 
numbers as the following computation shows with the discriminant $D_0$:
$$D_0=-892371480=-2^3 \cdot 3 \cdot 5 \cdot 7 \cdot 11 
\cdot 13 \cdot 17 \cdot 19 \cdot 23, $$

for which $\frac{H}{2^{N-1} \cdot (\sqrt{\vert D_K \vert}\,)^{0.05}}
=23.89441208$ while the order of magnitude of the successive maxima
for such order of magnitude of discriminants is $39326.\cdots$; it is clear 
that $H=10240$ for $D_K=-892371480$, is far to give a local maximum
(see line ${}^{***}$):

\footnotesize
\begin{verbatim}
eps=0.05
D_K=-891819431   H=64723 h=64723   N=1   C=38663.54897528134436
D_K=-891972311   H=65807 h=65807   N=1   C=39310.92910296295340
D_K=-892371479   H=65833 h=65833   N=1   C=39326.02076918814615
D_K=-892371480   H=10240 h=40      N=9   C=23.89441208396179319 ***
D_K=-892909079   H=66789 h=66789   N=1   C=39896.49651358522089
D_K=-898496759   H=68083 h=68083   N=1   C=40663.12685277796768
D_K=-900119351   H=68387 h=68387   N=1   C=40842.85100395781405
\end{verbatim}

\normalsize
Similarly, for the last discriminant $-961380175077106319535$ above, with
$H=19836665856$, $\frac{H}{2^{N-1} \cdot (\sqrt{\vert D_K \vert}\,)^{0.05}}
=180903.24774117$ for a local maximum $>23715549673$ at this place:

\smallskip
\footnotesize
\begin{verbatim}
D_K=-961380175077106313759 H=73527737200 N=5 C=21972467046.501154082
D_K=-961380175077106315439 H=74348146642 N=2 C=22217631934.140092600
D_K=-961380175077106319279 H=76863968144 N=2 C=22969440804.542458758
D_K=-961380175077106319519 H=79360715400 N=3 C=23715549673.045788420
\end{verbatim}
\normalsize

\subsubsection{Successive maxima of $\frac{H}{(\sqrt{D}\,)^\varepsilon}$
 (see Table II, Appendix \ref{Table2})}
In terms of successive maxima of $\frac{H}{(\sqrt{D}\,)^\varepsilon}$, we 
observe that they occur for fields $K$ such $N \in \{1, 2, 3\}$ for 
$D \leq153116519$ and that for very large discriminants this set is 
$\{1, 2, 3, 4\}$ (but mainly $1, 2$). We give below some excerpts 
of the whole table for $\varepsilon = 0.05$:

\footnotesize
\smallskip
\begin{verbatim}
{eps=0.05;bD=1;BD=10^9;Cm=0;for(D=bD,BD,e2=valuation(D,2);d=D/2^e2;
if(e2==1||e2>3||core(d)!=d||(e2==0&Mod(-d,4)!=1)||(e2==2&Mod(-d,4)!=-1),
next);H=qfbclassno(-D);C=H/(D^(eps/2));if(C>Cm,Cm=C;
print("D_K=",-D," H=",H," N=",omega(D)," C=",precision(C,1))))}

eps=0.05
D_K=-3            H=1      N=1    C=0.972908434869468710
D_K=-15           H=2      N=2    C=1.869079241783001633
D_K=-23           H=3      N=1    C=2.773818617890694607
D_K=-39           H=4      N=2    C=3.649920257029810563
D_K=-47           H=5      N=1    C=4.541167885124564221
D_K=-71           H=7      N=1    C=6.292403751297605636
D_K=-95           H=8      N=2    C=7.139156409132720178
(...)
D_K=-1036031      H=1908   N=2    C=1349.565761315970158
D_K=-1062311      H=1927   N=1    C=1362.151530472142189
D_K=-1086791      H=1965   N=1    C=1388.221940001361239
D_K=-1101239      H=1970   N=2    C=1391.294878906636963
D_K=-1127759      H=1982   N=2    C=1398.937276305459193
D_K=-1138199      H=1984   N=3    C=1400.026360256248040
D_K=-1139519      H=2027   N=1    C=1430.328227182419060
(...)
D_K=-100682039    H=22178  N=2    C=13990.99428148920046
D_K=-102628679    H=22235  N=1    C=14020.23894496263475
D_K=-104716751    H=22315  N=1    C=14063.59940393199907
D_K=-105734879    H=22530  N=2    C=14195.66476857999654
D_K=-106105271    H=23077  N=1    C=14539.04651797868699
D_K=-110487599    H=23275  N=1    C=14648.96196034653848
D_K=-110869391    H=23356  N=3    C=14698.67458157540117
(...)
D_K=-10000000000000391   H=206713870   N=2    C=32761940.53502909525
D_K=-10000000000013399   H=213240384   N=2    C=33796323.29593711028
D_K=-10000000000024271   H=221061984   N=4    C=35035963.35535050720
D_K=-10000000000026959   H=227089538   N=2    C=35991266.28552908299
D_K=-10000000000070519   H=231231712   N=4    C=36647756.62297995465
D_K=-10000000000123559   H=235020092   N=2    C=37248174.39021544804
D_K=-10000000000141511   H=245552727   N=1    C=38917484.54122868952
D_K=-10000000001180279   H=251465706   N=2    C=39854628.55744758011
(...)
\end{verbatim}

\normalsize
\subsubsection{Decomposition of the prime nuumbers}
We see that (except for few cases at the bigining of the algorithm) 
the discriminants are such $D_K \equiv 1 \pmod {8}$, which corresponds
to the splitting of $2$ in~$K$; similarly, the small unramified prime numbers
$3, 5, 7, 11,\ldots$ splits in $K$. For $D_K=-1432580159$, 
all primes $p$, $2 \leq p \leq 67$, split. For $D_K=-2017348511$
all primes $p$, $2 \leq p \leq 107$, split, except $p=79$ (Table I, Appendix \ref{Table1}).

\smallskip
This is not too surprising since the class of a prime ideal of small norm $p$ 
is of hight order since we must have a solution to $x^2+D\cdot y^2 = p^n$, 
which yields approximatively $n > \frac{D}{{\rm log}(p)}$. In the
references \cite{Da, DaKM,Du,GS,Lam,L0,L1,L2,MW,R}, using complex
$L$-functions, it is the process to get large class numbers
$H=\ds \frac{L(1,\chi)}{\pi} \sqrt{D}$, where $\chi$ is the 
Dirichlet character modulo $D$. This is also the principle of the paper
of Ellenberg--Venkatesh for the $p$-rank $\varepsilon$-conjecture
(see \cite[\S\,1.1]{EV}).

\section{Programs and tables for imaginary quadratic fields}

The two following programs compute the successive maxima of 
$\frac{H}{2^{N-1} \cdot (\sqrt{D}\,)^\varepsilon}$. 

Program I works
with a given $\varepsilon$ and the computations of $N, H$ as 
the discriminant $D$ increases arbitrarily. Program II computes, 
once for all, a list of triples $(D, N, H)$; then, for comparison 
of the results, we can use this list for a great lot of values of 
$\varepsilon$ for faster computations of the successive maxima.

\subsection{Successive maxima of $\frac{H}{2^{N-1} \cdot (\sqrt{D}\,)^\varepsilon}$}

\subsubsection{Program I -- Successive maxima of 
$\frac{H}{2^{N-1} \cdot (\sqrt{D}\,)^\varepsilon}$ with fixed $\varepsilon $}\label{P1}
One must choose $\varepsilon$ in ${\sf eps}$ and the interval 
${\sf [bD, BD]}$ for ${\sf D}$. The successive maxima of 
$C_\varepsilon(D) = \frac{H}{2^{N-1} \cdot (\sqrt{D}\,)^\varepsilon}$ 
are given in ${\sf C}$ with the corresponding values of $H$ (in ${\sf H}$),
$h$ (in ${\sf h}$) and $N$ (in ${\sf N}$). We write some excerpts of the 
table with various $\varepsilon$, $0<\varepsilon \ll 1$, before giving a more 
complete table in \S\,\ref{table1} and Appendix \ref{Table1} 
(all the missing results below are with $N=1$).

\footnotesize
\medskip
\begin{verbatim}
{eps=0.05;bD=1;BD=5*10^7;Cm=0;for(D=bD,BD,e2=valuation(D,2);d=D/2^e2;
if(e2==1||e2>3||core(d)!=d||(e2==0&Mod(-d,4)!=1)||(e2==2&Mod(-d,4)!=-1),
next);H=qfbclassno(-D);N=omega(D);h=H/2^(N-1);C=h/(D^(eps/2));if(C>Cm,Cm=C;
print("D_K=",-D," H=",H," h=",h," N=",N," C=",precision(C,1))))}
eps=0.05
D_K=-3         H=1       h=1         N=1    C=0.972908434869468710
D_K=-23        H=3       h=3         N=1    C=2.773818617890694607
D_K=-47        H=5       h=5         N=1    C=4.541167885124564221
D_K=-71        H=7       h=7         N=1    C=6.292403751297605636
D_K=-167       H=11      h=11        N=1    C=9.678872599268429560
D_K=-191       H=13      h=13        N=1    C=11.40033250135200530
(...)
D_K=-1006799   H=1853    h=1853      N=1    C=1311.601334671650505
D_K=-1032071   H=1857    h=1857      N=1    C=1313.618222552253122
D_K=-1062311   H=1927    h=1927      N=1    C=1362.151530472142189
D_K=-1086791   H=1965    h=1965      N=1    C=1388.221940001361239
D_K=-1139519   H=2027    h=2027      N=1    C=1430.328227182419060
D_K=-1190591   H=2051    h=2051      N=1    C=1445.678077927788222
(...)
D_K=-10057031  H=6425    h=6425      N=1    C=4293.499209586390154
D_K=-10289639  H=6563    h=6563      N=1    C=4383.211236938274762
D_K=-10616759  H=6583    h=6583      N=1    C=4393.130011813243966
D_K=-10865279  H=6725    h=6725      N=1    C=4485.297628751901855
D_K=-10984991  H=6767    h=6767      N=1    C=4512.073690323125874
D_K=-11101151  H=6777    h=6777      N=1    C=4517.553301414761614
(...)
D_K=-30706079  H=12017   h=12017     N=1    C=7809.360152349024611
D_K=-32936879  H=12151   h=12151     N=1    C=7882.608544620325792
D_K=-34103471  H=12285   h=12285     N=1    C=7962.605439538764212
D_K=-34867271  H=12531   h=12531     N=1    C=8117.555770269597220
D_K=-35377319  H=12549   h=12549     N=1    C=8126.265293030949160
D_K=-36098039  H=12673   h=12673     N=1    C=8202.426446333352781
(...)
D_K=-68514599  H=17485   h=17485     N=1    C=11137.07280268520213
D_K=-68621159  H=17763   h=17763     N=1    C=11313.70536648692712
D_K=-69671159  H=18455   h=18455     N=1    C=11749.99613226375525
D_K=-75038519  H=18611   h=18611     N=1    C=11827.35419243076771
D_K=-75865919  H=18653   h=18653     N=1    C=11850.79599797629443
D_K=-76708271  H=19191   h=19191     N=1    C=12189.23782258446658
(...)
\end{verbatim}
\normalsize

Taking intervals with huge bounds, still with ${\sf Cm=0}$, it is normal that the first 
maximum is not necessarily large enough to initiate the process which quickly 
stabilizes with prime discriminants (i.e., $N=1$):

\footnotesize
\begin{verbatim}
D_K=-100000000003    H=31057     h=31057     N=1    C=16487.67818446943949
D_K=-100000000004    H=150192    h=37548     N=3    C=19933.64911196521700
D_K=-100000000007    H=294590    h=147295    N=2    C=78196.62421286658525
D_K=-100000000019    H=160731    h=160731    N=1    C=85329.58760528567330
D_K=-100000000091    H=225785    h=225785    N=1    C=119865.7442379667594
D_K=-100000000103    H=350895    h=350895    N=1    C=186284.6970536358269
D_K=-100000000319    H=405637    h=405637    N=1    C=215346.3732873515489
D_K=-100000001111    H=520065    h=520065    N=1    C=276094.4184123037331
(...)
D_K=-100000000000000000391 H=12650126133 h=12650126133 N=1 C=4000321126.86
D_K=-100000000000000000559 H=13816100521 h=13816100521 N=1 C=4369034602.82
D_K=-100000000000000004159 H=15675687007 h=15675687007 N=1 C=4957087483.00
D_K=-100000000000000004519 H=17188889951 h=17188889951 N=1 C=5435604269.51
D_K=-100000000000000005071 H=18354412823 h=18354412823 N=1 C=5804174963.56
D_K=-100000000000000010039 H=21539854655 h=21539854655 N=1 C=6811500117.88
D_K=-100000000000000375799 H=22001805625 h=22001805625 N=1 C=6957581841.13
D_K=-100000000000000940831 H=22578025429 h=22578025429 N=1 C=7139798542.48
(...)
\end{verbatim}
\normalsize

\smallskip
There are huge intervals between two successive maxima; this probably means
that the density of these successive maxima is zero. Perhaps this comes easily from 
analytical computations as those given in \cite{Da,DaKM,Du,Lam} and others.

\subsubsection{Distribution of the values of $N$ for the successive maxima of 
$\frac{H}{2^{N-1}(\sqrt D\,)^\varepsilon}$}\label{table1}

We give the statistical distribution of the values of $N$ for the successive maxima of 
$C_\varepsilon(D) = \frac{H}{2^{N-1}(\sqrt D\,)^\varepsilon}$ with $\varepsilon =0.02$
(where $N_i$ denotes the number of $N=i$ at each new maximum; $N_D$ is the total 
number of consider discriminants); see Appendix \ref{Table1}:

\footnotesize 
\smallskip
\begin{verbatim}
{eps=0.02;bD=1;BD=10^12;ND=0;N1=0;N2=0;N3=0;N4=0;N5=0;N6=0;
Cm=0;for(D=bD,BD,e2=valuation(D,2);d=D/2^e2;
if(e2==1||e2>3||core(d)!=d||(e2==0&Mod(-d,4)!=1)||(e2==2&Mod(-d,4)!=-1),
next);ND=ND+1;H=qfbclassno(-D);N=omega(D);h=H/2^(N-1);C=h/D^(eps/2);
if(C>Cm,Cm=C;if(N==1,N1=N1+1);if(N==2,N2=N2+1);if(N>=3,N3=N3+1);print();
print("D_K=",-D," H=",H," h=",h," N=",N," C=",precision(C,1));
print("ND=",ND," N1=",N1," N2=",N2," N3=",N3)))}
\end{verbatim}
\normalsize

\subsubsection{Computations with $\varepsilon \approx 1$}\label{eps=1}
The case $\varepsilon = 1$ or $\varepsilon = 1\pm \eta$ 
(for any $\eta>0$ close to $0$ as we recalled in \S\,\ref{strong}) has 
some interest because of some known upper and lower bounds for $H$ 
with classical inequalities in ${\mathcal K}_2^-$:

$\ds \order \Cl_K \leq  {\mathscr C}_{\eta}\cdot (\sqrt {D}\,)^{1+ \eta}$,
for all $D$ (inequality \eqref{0}\,(ii)), for some constant ${\mathscr C}_{\eta}$,
and the inverse inequality
$\order \Cl_K \geq  (\sqrt {D}\,)^{1-\eta}$,
for infinitely many $D$ (Proposition \ref{00}), 
giving (in some sense with $\varepsilon = 1+\eta$):
$C_{1+\eta}(D) \leq \frac{{\mathscr C}_{\eta}}{2^{N-1}}$,
for all $D$, and proving that $\ds \liminf_D (C_{1+\eta}(D))=0$ 
and that $\ds \limsup_D(C_{1+\eta}(D))<\infty$.

A priori, the second inequality is more interesting 
when $N=1$ (then $D_K= s\,\ell \equiv 1 \pmod 4$), 
yielding $C_{1+\eta}(\ell) \gg_{\eta} 0$.
This gives some information regarding the conjectural sequence 
$\ell_i$ defined by \eqref{li}.

\footnotesize
\begin{verbatim}
{eps=1;bD=1;BD=10^8;Cm=0;for(D=bD,BD,e2=valuation(D,2);d=D/2^e2;
if(e2==1||e2>3||core(d)!=d||(e2==0&Mod(-d,4)!=1)||(e2==2&Mod(-d,4)!=-1),
next);H=qfbclassno(-D);N=omega(D);h=H/2^(N-1);C=h/(D^(eps/2));if(C>Cm,Cm=C;
print("D_K=",-D," H=",H," h=",h," N=",N," C=",precision(C,1))))}

eps=1
D_K=-3          H=1        h=1        N=1    C=0.577350269189625764
D_K=-23         H=3        h=3        N=1    C=0.625543242171224288
D_K=-47         H=5        h=5        N=1    C=0.729324957489472779
D_K=-71         H=7        h=7        N=1    C=0.830747160735697329
D_K=-167        H=11       h=11       N=1    C=0.851205555787550539
D_K=-191        H=13       h=13       N=1    C=0.940646986880148131
D_K=-239        H=15       h=15       N=1    C=0.970269341029726350
D_K=-311        H=19       h=19       N=1    C=1.077391156535111118
D_K=-479        H=25       h=25       N=1    C=1.142279155962384611
D_K=-719        H=31       h=31       N=1    C=1.156104917973712355
D_K=-1151       H=41       h=41       N=1    C=1.208498720419083129
D_K=-1319       H=45       h=45       N=1    C=1.239053663199498814
D_K=-1511       H=49       h=49       N=1    C=1.260560945801805027
D_K=-1559       H=51       h=51       N=1    C=1.291656751946608433
D_K=-2351       H=63       h=63       N=1    C=1.299314349827718283
D_K=-2999       H=73       h=73       N=1    C=1.333013744070415380
D_K=-3671       H=81       h=81       N=1    C=1.336881232618676007
D_K=-5711       H=109      h=109      N=1    C=1.442349199551804946
D_K=-6551       H=117      h=117      N=1    C=1.445546485075073189
D_K=-9239       H=139      h=139      N=1    C=1.446113287777516307
D_K=-10391      H=153      h=153      N=1    C=1.500938018707737875
D_K=-14951      H=185      h=185      N=1    C=1.512991916312454831
D_K=-15791      H=195      h=195      N=1    C=1.551778729394620610
D_K=-18191      H=213      h=213      N=1    C=1.579251567061564281
D_K=-31391      H=289      h=289      N=1    C=1.631155133694937910
D_K=-38639      H=325      h=325      N=1    C=1.653371403795025953
D_K=-63839      H=423      h=423      N=1    C=1.674161419982429269
D_K=-88919      H=505      h=505      N=1    C=1.693534674808400201
D_K=-95471      H=533      h=533      N=1    C=1.725009428818167194
D_K=-147671     H=667      h=667      N=1    C=1.735714238040899734
D_K=-191231     H=761      h=761      N=1    C=1.740225421963015830
D_K=-250799     H=875      h=875      N=1    C=1.747210185407398438
D_K=-284231     H=943      h=943      N=1    C=1.768788765798006676
D_K=-289511     H=963      h=963      N=1    C=1.789755692477908166
D_K=-312311     H=1001     h=1001     N=1    C=1.791184972640468990
D_K=-366791     H=1121     h=1121     N=1    C=1.850956776350963554
D_K=-514751     H=1347     h=1347     N=1    C=1.877452665000664534
D_K=-628319     H=1495     h=1495     N=1    C=1.886040855881855214
D_K=-701399     H=1581     h=1581     N=1    C=1.887770950411945618
D_K=-819719     H=1725     h=1725     N=1    C=1.905270304565445333
D_K=-1412759    H=2311     h=2311     N=1    C=1.944311765270786403
D_K=-2470079    H=3079     h=3079     N=1    C=1.959089455884703822
D_K=-3068831    H=3485     h=3485     N=1    C=1.989373351795705735
D_K=-3190151    H=3593     h=3593     N=1    C=2.011646181366134258
D_K=-4305479    H=4227     h=4227     N=1    C=2.037142972901196004
D_K=-10289639   H=6563     h=6563     N=1    C=2.045984480430942894
D_K=-11935871   H=7165     h=7165     N=1    C=2.073906326984584401
D_K=-12537719   H=7457     h=7457     N=1    C=2.105983077270287202
D_K=-16368959   H=8565     h=8565     N=1    C=2.116980398599532524
D_K=-29036999   H=11421    h=11421    N=1    C=2.119474846049997785
D_K=-29858111   H=11885    h=11885    N=1    C=2.175043874892598601
D_K=-48796439   H=15195    h=15195    N=1    C=2.175237288881570956
D_K=-58057991   H=16631    h=16631    N=1    C=2.182666393650392413
D_K=-69671159   H=18455    h=18455    N=1    C=2.210993836560839743
D_K=-82636319   H=20299    h=20299    N=1    C=2.233002255956684257
D_K=-106105271  H=23077    h=23077    N=1    C=2.240324202281881086
(...)
\end{verbatim}

\normalsize
For $\varepsilon=1.1$ we have found the sequence of prime discriminants:

\footnotesize
\begin{verbatim}
D_K=-3        H=1     h=1     N=1    C=0.5464913722529576510
D_K=-47       H=5     h=5     N=1    C=0.6016115735959143380
D_K=-71       H=7     h=7     N=1    C=0.6712834625402534319
D_K=-191      H=13    h=13    N=1    C=0.7233941632706684037
D_K=-239      H=15    h=15    N=1    C=0.7378573598045812696
D_K=-311      H=19    h=19    N=1    C=0.8086030720739909916
D_K=-479      H=25    h=25    N=1    C=0.8389874728227036802
D_K=-1151     H=41    h=41    N=1    C=0.8495568162708291688
D_K=-1319     H=45    h=45    N=1    C=0.8651230358188882253
D_K=-1511     H=49    h=49    N=1    C=0.8741795137245637057
D_K=-1559     H=51    h=51    N=1    C=0.8943444427816406492
D_K=-5711     H=109   h=109   N=1    C=0.9359115563608913749
D_K=-10391    H=153   h=153   N=1    C=0.9452134463623177874
D_K=-15791    H=195   h=195   N=1    C=0.9569942159441953143
D_K=-18191    H=213   h=213   N=1    C=0.9670712843980043869
D_K=-31391    H=289   h=289   N=1    C=0.9719748099741753952
D_K=-38639    H=325   h=325   N=1    C=0.9750325810206494251
D_K=-366791   H=1121  h=1121  N=1    C=0.9753833631207420025
\end{verbatim}

\normalsize
For $\varepsilon=1.2$ one only finds a sequence of 7 solutions in the selected
interval:

\footnotesize
\begin{verbatim}
D_K=-3     H=1   h=1    N=1    C=0.5172818579717865581
D_K=-71    H=7   h=7    N=1    C=0.5424291630211148068
D_K=-191   H=13  h=13   N=1    C=0.5563182817283039269
D_K=-239   H=15  h=15   N=1    C=0.5611158267043577307
D_K=-311   H=19  h=19   N=1    C=0.6068723733265467012
D_K=-479   H=25  h=25   N=1    C=0.6162241303969013829
D_K=-1559  H=51  h=51   N=1    C=0.6192449976582213274
\end{verbatim}

\normalsize
\smallskip
For $\varepsilon=1.25$, the successive maxima obtained is:

\footnotesize
\smallskip
\begin{verbatim}
D_K=-3     H=1   h=1    N=1    C=0.5032678828257016665
D_K=-311   H=19  h=19   N=1    C=0.5257488819694962654
D_K=-479   H=25  h=25   N=1    C=0.5281172008961688794
\end{verbatim}

\normalsize
But for $\varepsilon=1.3$ one gets an empty sequence in the interval $[3, 10^8]$
(i.e., the initialization ${\sf D_K=-3,\  H=1,\, h=1,\, N=1,\, C=0.4896335682000245806}$).
After some adjustments giving the limit value ${\sf eps \approx 1.2690005098}$,
we have the same property (still in the interval $[3, 10^8]$).

\subsubsection{Program II -- Successive maxima of 
$\frac{H}{2^{N-1} \cdot (\sqrt{D}\,)^\varepsilon}$ with variable $\varepsilon$}
Now we intend to vary $\varepsilon \in ]0,1[$; the program computes first the list of 
all discriminants in the selected interval, with the values of ${\sf N}$ and ${\sf H}$
and builts the corresponding lists ${\sf LD, LN, LH}$:

\footnotesize
\smallskip
\begin{verbatim}
{LD=List;LN=List;LH=List;NL=0;bD=1;BD=10^6;
for(D=bD,BD,e2=valuation(D,2);d=D/2^e2;if(e2==1||e2>3||core(d)!=d
||(e2==0&Mod(-d,4)!=1)||(e2==2&Mod(-d,4)!=-1),next);NL=NL+1;
listput(LD,D);N=omega(D);listput(LN,N);H=qfbclassno(-D);listput(LH,H))}
\end{verbatim}

\normalsize
\smallskip
Once we have built the lists ${\sf LD, LN, LH}$ of the ${\sf NL}$ discriminants 
${\sf D \in [bD, BD]}$, the numbers ${\sf N}$ of ramified primes and the class 
numbers ${\sf H}$, respectively, one uses the following instruction with adequate 
choices of $\varepsilon$, which is very fast (the program only writes the solutions
with $N>1$, if any):

\footnotesize
\smallskip
\begin{verbatim}
{for(k=0,100,eps=k*0.0005;print(eps);Cm=0;for(j=1,NL,D=LD[j];N=LN[j];H=LH[j];
C=H/(2^(N-1)*(D^(eps/2)));if(C>Cm,Cm=C;if(N>1,h=H/2^(N-1);
print("D_K=",D," H=",H," h=",h," N=",N," C=",precision(C,1))))))}

{for(k=0,100,eps=k*0.005;print(eps);Cm=0;for(j=1,NL,D=LD[j];N=LN[j];H=LH[j];
C=H/(2^(N-1)*D^(eps/2));if(C>Cm,Cm=C;if(N>1,h=H/2^(N-1);
print("D_K=",D," H=",H," h=",h," N=",N," C=",precision(C,1))))))}
\end{verbatim}

\normalsize
\smallskip
and so on at will. 
To obtain the complete table for each $\varepsilon$, one drops the test on $N$:

\footnotesize
\smallskip
\begin{verbatim}
{for(k=0,100,eps=k*0.005;print(eps);Cm=0;for(j=1,NL,D=LD[j];N=LN[j];H=LH[j];
C=H/(2^(N-1)*D^(eps/2));if(C>Cm,Cm=C;h=H/2^(N-1);
print("D_K=",D," H=",H," h=",h," N=",N," C=",precision(C,1)))))}
\end{verbatim}
\normalsize

\begin{remark}{\rm When $\varepsilon$ varies, the conjecture is
fulfilled but the lists of successive maxima may be different as illustrated by 
the following excerpt:

\footnotesize
\smallskip
\begin{verbatim}
eps=0.2
D_K=-61151   H=395    h=395    N=1   C=131.2069195353264980
D_K=-61871   H=399    h=399    N=1   C=132.3805503137288383
D_K=-63839   H=423    h=423    N=1   C=139.9045245218128567
D_K=-80471   H=455    h=455    N=1   C=147.0440503223027966
D_K=-87071   H=469    h=469    N=1   C=150.3784103142821144

eps=0.05
D_K=-61151   H=395    h=395    N=1   C=299.8727528609150616
D_K=-61871   H=399    h=399    N=1   C=302.8208103554958095
D_K=-63839   H=423    h=423    N=1   C=320.7843818658118495

D_K=-75479   H=429    h=429    N=1   C=323.9751045905628560  

D_K=-80471   H=455    h=455    N=1   C=343.0602594282148802
D_K=-87071   H=469    h=469    N=1   C=352.9197837212165526
\end{verbatim}
\normalsize

\smallskip
where for $\varepsilon=0.05$ we obtain a supplementary
maximum for the discriminant $D_K=-75479$.}
\end{remark}

\subsection{Successive maxima of $\frac{H}{(\sqrt D\,)^\varepsilon}$}

When we suppress the genus factor $2^{N-1}$, we obtain many successive 
maxima with composite discriminants. Before giving a more complete table
in \S\,\ref{table2} and Appendix \ref{Table2}, we give some examples 
with various $\varepsilon$

\subsubsection{First examples with various $\varepsilon$}
${}$
\footnotesize
\begin{verbatim}
{eps=0.1;bD=1;BD=10^8;Cm=0;for(D=bD,BD,e2=valuation(D,2);d=D/2^e2;
if(e2==1||e2>3||core(d)!=d||(e2==0&Mod(-d,4)!=1)||(e2==2&Mod(-d,4)!=-1),
next);H=qfbclassno(-D);N=omega(D);h=H/2^(N-1);C=H/D^(eps/2);
if(C>Cm,Cm=C; \\ extrema test
if(N>1,print("D_K=",-D," H=",H," h=",h," N=",N," C=",precision(C,1)))))}

eps=0.1
D_K=-15      H=2     h=1    N=2    C=1.746728606032060139
D_K=-39      H=4     h=2    N=2    C=3.330479470669139602
D_K=-95      H=8     h=4    N=2    C=6.370944279257599436
D_K=-119     H=10    h=5    N=2    C=7.874493918573016203
D_K=-215     H=14    h=7    N=2    C=10.70301447267741998
D_K=-551     H=26    h=13   N=2    C=18.96338088337081395
D_K=-671     H=30    h=15   N=2    C=21.66631878049052680
D_K=-791     H=32    h=16   N=2    C=22.92140074278638768
D_K=-959     H=36    h=18   N=2    C=25.53945176699087213
D_K=-1679    H=52    h=26   N=2    C=35.87160500976424895
D_K=-1991    H=56    h=28   N=2    C=38.30314759825849873
D_K=-2159    H=60    h=30   N=2    C=40.87319792142757059
D_K=-2519    H=64    h=32   N=2    C=43.26319266328724219
D_K=-2831    H=68    h=34   N=2    C=45.69954953825021836
(...)
D_K=-4199    H=88    h=22   N=3    C=57.98629737295230403
D_K=-4991    H=92    h=23   N=3    C=60.10055034028705556
(...)
D_K=-201431  H=808   h=404  N=2    C=438.7377920576307146
D_K=-230879  H=824   h=206  N=3    C=444.3835663948799569
D_K=-239279  H=880   h=220  N=3    C=473.7371498813250101
D_K=-250631  H=892   h=223  N=3    C=479.0855974383289828
D_K=-259871  H=894   h=447  N=2    C=479.2913905823451286
D_K=-277031  H=918   h=459  N=2    C=490.5872564296189357
D_K=-277679  H=930   h=465  N=2    C=496.9421066431691019
D_K=-301391  H=974   h=487  N=2    C=518.3253419873442213
D_K=-307271  H=992   h=496  N=2    C=527.3944964507170297
D_K=-321551  H=1008  h=504  N=2    C=534.6850471107224307
D_K=-328319  H=1024  h=512  N=2    C=542.6067047266572173
D_K=-331679  H=1048  h=524  N=2    C=555.0414075965477244
D_K=-346631  H=1054  h=527  N=2    C=556.9897983860860616
(...)
D_K=-10000335791 H=224228 h=56057  N=3 C=70907.00047065937018
D_K=-10000430831 H=225362 h=112681 N=2 C=71265.56829114792618
D_K=-10000494479 H=230920 h=115460 N=2 C=73023.13519081427577
D_K=-10000950671 H=232228 h=58057  N=3 C=73436.59259212560310
D_K=-10001585879 H=234304 h=117152 N=2 C=74092.84302164243799
D_K=-10005286079 H=234340 h=58585  N=3 C=74102.85661222034478
D_K=-10005638639 H=236064 h=118032 N=2 C=74647.88735802392332
D_K=-10005946271 H=239138 h=119569 N=2 C=75619.82786456386412
D_K=-10014676559 H=239668 h=119834 N=2 C=75784.11884852224786
D_K=-10017718031 H=240986 h=120493 N=2 C=76199.71955113559173
D_K=-10042787999 H=244463 h=244463 N=1 C=77289.48656509497127
D_K=-10085398391 H=245575 h=245575 N=1 C=77624.62225923972404
(...)
\end{verbatim}
\normalsize 

\smallskip
Taking $\varepsilon = 0.20$ then $\varepsilon = 0.005$, we obtain, locally,
the following variation of the results giving the successive maxima
of $\frac{H}{ D^{\varepsilon/2}}$:

\footnotesize
\begin{verbatim}
eps=0.20                              eps =0.005
D_K=-34151 H=292 h=73  N=3 C=102.811    D_K=-34151 H=292 h=73  N=3 C=284.478
D_K=-34271 H=300 h=150 N=2 C=105.59     D_K=-34271 H=300 h=150 N=2 C=292.269
D_K=-45239 H=352 h=176 N=2 C=120.501    D_K=-45239 H=352 h=176 N=2 C=342.691
D_K=-56759 H=384 h=192 N=2 C=128.507    D_K=-48551 H=354 h=177 N=2 C=344.578
                                        D_K=-55319 H=376 h=94  N=3 C=365.873
                                        D_K=-56759 H=384 h=192 N=2 C=373.633
D_K=-60359 H=398 h=199 N=2 C=132.375    D_K=-60359 H=398 h=199 N=2 C=387.196
D_K=-65159 H=424 h=212 N=2 C=139.948    D_K=-65159 H=424 h=212 N=2 C=412.411
\end{verbatim}
\normalsize

\subsubsection{Successive maxima of $\frac{H}{\sqrt D}$}
The computation of the successive maxima of $C=\frac{H}{\sqrt D}$, 
which corresponds to take the value $\varepsilon = 1$,
gives the following (two days of computer):

\footnotesize
\begin{verbatim}
{eps=1;Cm=0;bD=2;BD=10^15;for(D=bD,BD,e2=valuation(D,2);d=D/2^e2;
if(e2==1||e2>3||core(d)!=d||(e2==0&Mod(-d,4)!=1)||(e2==2&Mod(-d,4)!=-1),
next);H=qfbclassno(-D);C=H/D^(eps/2);if(C>Cm,Cm=C; \\ maxima test
print("D_K=",-D," H=",H," N=",omega(D)," C=",precision(C,1))))}

eps=1
D_K=-3          H=1       N=1    C=0.5773502691896257646
D_K=-23         H=3       N=1    C=0.6255432421712242881
D_K=-39         H=4       N=2    C=0.6405126152203485340
D_K=-47         H=5       N=1    C=0.7293249574894727793
D_K=-71         H=7       N=1    C=0.8307471607356973298
D_K=-119        H=10      N=2    C=0.9166984970282112951
D_K=-191        H=13      N=1    C=0.9406469868801481313
D_K=-215        H=14      N=2    C=0.9547920752586628931
D_K=-239        H=15      N=1    C=0.9702693410297263506
D_K=-311        H=19      N=1    C=1.0773911565351111182
D_K=-479        H=25      N=1    C=1.1422791559623846113
D_K=-671        H=30      N=2    C=1.1581371547232743925
D_K=-959        H=36      N=2    C=1.1625006306971278460
D_K=-1151       H=41      N=1    C=1.2084987204190831295
D_K=-1319       H=45      N=1    C=1.2390536631994988148
D_K=-1511       H=49      N=1    C=1.2605609458018050271
D_K=-1559       H=51      N=1    C=1.2916567519466084330
D_K=-2351       H=63      N=1    C=1.2993143498277182834
D_K=-2999       H=73      N=1    C=1.3330137440704153804
D_K=-3071       H=76      N=2    C=1.3714301220316550780
D_K=-5711       H=109     N=1    C=1.4423491995518049464
D_K=-6551       H=117     N=1    C=1.4455464850750731891
D_K=-8399       H=134     N=2    C=1.4621468997251127841
D_K=-10391      H=153     N=1    C=1.5009380187077378750
D_K=-13439      H=174     N=2    C=1.5009484314636654405
D_K=-13991      H=178     N=2    C=1.5048583556514348538
D_K=-14951      H=185     N=1    C=1.5129919163124548319
D_K=-15791      H=195     N=1    C=1.5517787293946206107
D_K=-18191      H=213     N=1    C=1.5792515670615642815
D_K=-31391      H=289     N=1    C=1.6311551336949379106
D_K=-38639      H=325     N=1    C=1.6533714037950259530
D_K=-45239      H=352     N=2    C=1.6549549083889909590
D_K=-63839      H=423     N=1    C=1.6741614199824292696
D_K=-88919      H=505     N=1    C=1.6935346748084002013
D_K=-95471      H=533     N=1    C=1.7250094288181671943  
D_K=-118271     H=606     N=2    C=1.7621119228622864136
D_K=-201431     H=808     N=2    C=1.8003137828933471731
D_K=-331679     H=1048    N=2    C=1.8197104745663294562
D_K=-366791     H=1121    N=1    C=1.8509567763509635548
D_K=-514751     H=1347    N=1    C=1.8774526650006645342
D_K=-628319     H=1495    N=1    C=1.8860408558818552140
D_K=-701399     H=1581    N=1    C=1.8877709504119456186
D_K=-819719     H=1725    N=1    C=1.9052703045654453332
D_K=-890951     H=1818    N=2    C=1.9260473919033585268
D_K=-1238639    H=2150    N=2    C=1.9318174505729727404
D_K=-1339439    H=2262    N=2    C=1.9544795494902961276
D_K=-2155919    H=2968    N=2    C=2.0213786607241667950
D_K=-4305479    H=4227    N=1    C=2.0371429729011960035
D_K=-6077111    H=5092    N=2    C=2.0655694919245702532
D_K=-11915279   H=7206    N=2    C=2.0875753006426724498
D_K=-12537719   H=7457    N=1    C=2.1059830772702872018
D_K=-16368959   H=8565    N=1    C=2.1169803985995325238
D_K=-20357039   H=9632    N=2    C=2.1348097104711738560
D_K=-29858111   H=11885   N=1    C=2.1750438748925986011
D_K=-48796439   H=15195   N=1    C=2.1752372888815709557
D_K=-54694631   H=16208   N=2    C=2.1915787444351673315
D_K=-63434159   H=17546   N=2    C=2.2030103210049586306
D_K=-69671159   H=18455   N=1    C=2.2109938365608397427
D_K=-82636319   H=20299   N=1    C=2.2330022559566842564
D_K=-106105271  H=23077   N=1    C=2.2403242022818810857
D_K=-131486759  H=25817   N=1    C=2.2514627416787898558
D_K=-173540351  H=29674   N=2    C=2.2525573483717439197
D_K=-197317559  H=32062   N=2    C=2.2824840043420033386
D_K=-242230271  H=35666   N=2    C=2.2916073035266200847
D_K=-258363551  H=36865   N=1    C=2.2934993203392119001
D_K=-354544511  H=43212   N=3    C=2.2949275753614584471
D_K=-469058399  H=49913   N=1    C=2.3046247166517230619
D_K=-499769591  H=51733   N=1    C=2.3141033452776863919
D_K=-507138551  H=52152   N=2    C=2.3158352241187728477
D_K=-625378991  H=58028   N=3    C=2.3204165730228952904
D_K=-693126671  H=61506   N=2    C=2.3362062349259721311
D_K=-826771391  H=67332   N=2    C=2.3416867288431861497
D_K=-967603199  H=73018   N=2    C=2.3473685708435125091
D_K=-1001354639 H=74768   N=3    C=2.3627719510460623112
D_K=-1017054191 H=75498   N=2    C=2.3673550657766015599
(...)
\end{verbatim}
\normalsize

which is infinite.

\subsubsection{Distribution of the values of $N$ for the successive maxima of 
$\frac{H}{(\sqrt D\,)^\varepsilon}$}\label{table2}

We give now the statistical distribution of the values of $N$
for the successive maxima of $C=\frac{H}{(\sqrt D\,)^\varepsilon}$
with $\varepsilon =0.02$; see Appendix \ref{Table2} for the whole table:

\footnotesize
\begin{verbatim}
{eps=0.02;bD=1;BD=10^12;ND=0;N1=0;N2=0;N3=0;N4=0;N5=0;N6=0;
Cm=0;for(D=bD,BD,e2=valuation(D,2);d=D/2^e2;
if(e2==1||e2>3||core(d)!=d||(e2==0&Mod(-d,4)!=1)||(e2==2&Mod(-d,4)!=-1),
next);ND=ND+1;H=qfbclassno(-D);N=omega(D);h=H/2^(N-1);C=H/D^(eps/2);
if(C>Cm,Cm=C;if(N==1,N1=N1+1);if(N==2,N2=N2+1);if(N==3,N3=N3+1);
if(N==4,N4=N4+1);if(N==5,N5=N5+1);if(N>=6,N6=N6+1);print();
print("D_K=",-D," H=",H," h=",h," N=",N," C=",precision(C,1));
print("ND=",ND," N1=",N1," N2=",N2," N3=",N3,
" N4=",N4," N5=",N5," N6=",N6)))}
(...)
D_K=-4079 H=85 h=85 N=1 C=78.21920339356804770
ND=1241      N1=21  N2=17  N3=0  N4=0  N5=0  N6=0

D_K=-4199 H=88 h=22 N=3 C=80.95640481885140620
ND=1278      N1=21  N2=17  N3=1  N4=0  N5=0  N6=0
(...)
D_K=-2819879 H=3220 h=1610 N=2 C=2775.578833830112931
ND=857149      N1=101  N2=100  N3=16  N4=0  N5=0  N6=0

D_K=-2843231 H=3290 h=1645 N=2 C=2835.683632670836421
ND=864246      N1=101  N2=101  N3=16  N4=0  N5=0  N6=0

D_K=-2919479 H=3305 h=3305 N=1 C=2847.858521758496226
ND=887419      N1=102  N2=101  N3=16  N4=0  N5=0  N6=0

D_K=-2954591 H=3464 h=866 N=3 C=2984.509115202034352
ND=898089      N1=102  N2=101  N3=17  N4=0  N5=0  N6=0

D_K=-3068831 H=3485 h=3485 N=1 C=3001.463405484378354
ND=932809      N1=103  N2=101  N3=17  N4=0  N5=0  N6=0

D_K=-3190151 H=3593 h=3593 N=1 C=3093.279082652959006
ND=969674      N1=104  N2=101  N3=17  N4=0  N5=0  N6=0

D_K=-3312839 H=3632 h=1816 N=2 C=3125.675125147948027
ND=1006982      N1=104  N2=102  N3=17  N4=0  N5=0  N6=0

D_K=-3406751 H=3636 h=1818 N=2 C=3128.242922219554128
ND=1035530      N1=104  N2=103  N3=17  N4=0  N5=0  N6=0

D_K=-3511199 H=3640 h=1820 N=2 C=3130.738753931376549
ND=1067278      N1=104  N2=104  N3=17  N4=0  N5=0  N6=0

D_K=-3524351 H=3714 h=1857 N=2 C=3194.266213425503417
ND=1071274      N1=104  N2=105  N3=17  N4=0  N5=0  N6=0
(...)
D_K=-73396319 H=18648 h=9324 N=2 C=15558.779203424189371
ND=22309821      N1=168  N2=191  N3=34  N4=0  N5=0  N6=0

D_K=-74175191 H=18672 h=2334 N=4 C=15577.158967342971961
ND=22546587      N1=168  N2=191  N3=34  N4=1  N5=0  N6=0
(...)
\end{verbatim}
\normalsize

The first crossing of the ``curves'' $N=1$ and $N=2$'' occurs 
approximatively for $D = 2843231$. The first $N=3$ is obtained for 
$D_K=-4199 = 13 \cdot 17 \cdot 19$.

\smallskip
The first $N=4$ is obtained for $D_K=-74175191=
-41 \!\cdot\! 71  \!\cdot\!  83  \!\cdot\!  307$ (for $N_D=22546587$).

\subsection{Successive  minima of  
$\frac{H}{2^{N-1}\cdot (\sqrt{\vert D_K \vert}\,)^{\varepsilon}}$}

Without any comments, we give the program computing the {\it successive minima}
of the function 
$$c_\varepsilon(K) := \frac{H}{2^{N-1}\cdot (\sqrt{\vert D_K \vert}\,)^{\varepsilon}}$$ 

up to $5 \cdot 10^7$, for imaginary quadratic fields. 

\smallskip
From \eqref{0}\,(ii), the sequence 
$\frac{H}{(\sqrt{\vert D_K \vert}\,)^{\varepsilon}}$ shall be finite for 
$\varepsilon=1+ \eta$, $\eta >0$, and infinite for $\varepsilon < 1$ and
a fortiori for $\frac{H}{2^{N-1}\cdot (\sqrt{\vert D_K \vert}\,)^{\varepsilon}}$
since $N$ is unbounded. We test the case $\varepsilon=1$:

\footnotesize
\begin{verbatim}
{eps=1;bD=1;BD=10^8;cM=1.0;for(D=bD,BD,e=valuation(D,2);d=D/2^e;
if(e==1||e>3||core(d)!=d||(e==0 & Mod(d,4)!=-1)||(e==2 & Mod(d,4)!=1),
next);N=omega(D);H=qfbclassno(-D);h=H/2^(N-1);c=h/(sqrt(D)^eps);
if(c<cM,cM=c; \\ minima test
print("D_K=",-D," H=",H," h=",h," N=",N," c=",precision(c,1))))}
eps=1
D_K=-3          H=1    h=1    N=1    C=0.5773502691896257646
D_K=-4          H=1    h=1    N=1    C=0.5000000000000000000
D_K=-7          H=1    h=1    N=1    C=0.3779644730092272272
D_K=-8          H=1    h=1    N=1    C=0.3535533905932737622
D_K=-11         H=1    h=1    N=1    C=0.3015113445777636227
D_K=-15         H=2    h=1    N=2    C=0.2581988897471611257
D_K=-19         H=1    h=1    N=1    C=0.22941573387056176592
D_K=-20         H=2    h=1    N=2    C=0.22360679774997896964
D_K=-24         H=2    h=1    N=2    C=0.20412414523193150818
D_K=-35         H=2    h=1    N=2    C=0.16903085094570331551
D_K=-40         H=2    h=1    N=2    C=0.15811388300841896661
D_K=-43         H=1    h=1    N=1    C=0.15249857033260466633
D_K=-51         H=2    h=1    N=2    C=0.14002800840280098035
D_K=-52         H=2    h=1    N=2    C=0.13867504905630728051
D_K=-67         H=1    h=1    N=1    C=0.12216944435630522344
D_K=-84         H=4    h=1    N=3    C=0.10910894511799619063
D_K=-88         H=2    h=1    N=2    C=0.10660035817780521715
D_K=-91         H=2    h=1    N=2    C=0.10482848367219182958
D_K=-115        H=2    h=1    N=2    C=0.09325048082403137657
D_K=-120        H=4    h=1    N=3    C=0.09128709291752768558
D_K=-123        H=2    h=1    N=2    C=0.09016696346674322897
D_K=-132        H=4    h=1    N=3    C=0.08703882797784891909
D_K=-148        H=2    h=1    N=2    C=0.08219949365267864445
D_K=-163        H=1    h=1    N=1    C=0.07832604499879573413
D_K=-168        H=4    h=1    N=3    C=0.07715167498104595513
D_K=-187        H=2    h=1    N=2    C=0.07312724241271306502
D_K=-195        H=4    h=1    N=3    C=0.07161148740394328805
D_K=-228        H=4    h=1    N=3    C=0.06622661785325219033
D_K=-232        H=2    h=1    N=2    C=0.06565321642986127833
D_K=-235        H=2    h=1    N=2    C=0.06523280730534421982
D_K=-267        H=2    h=1    N=2    C=0.06119900613621045639
D_K=-280        H=4    h=1    N=3    C=0.05976143046671968200
D_K=-312        H=4    h=1    N=3    C=0.05661385170722978753
D_K=-340        H=4    h=1    N=3    C=0.05423261445466404300
D_K=-372        H=4    h=1    N=3    C=0.05184758473652126342
D_K=-403        H=2    h=1    N=2    C=0.04981354813867178958
D_K=-408        H=4    h=1    N=3    C=0.04950737714883371546
D_K=-420        H=8    h=1    N=4    C=0.04879500364742665897
D_K=-427        H=2    h=1    N=2    C=0.04839339184958272744
D_K=-435        H=4    h=1    N=3    C=0.04794633014853841427
D_K=-483        H=4    h=1    N=3    C=0.04550157551932900739
D_K=-520        H=4    h=1    N=3    C=0.04385290096535146074
D_K=-532        H=4    h=1    N=3    C=0.04335549847620599959
D_K=-555        H=4    h=1    N=3    C=0.04244763599780088816
D_K=-595        H=4    h=1    N=3    C=0.04099600308453938812
D_K=-627        H=4    h=1    N=3    C=0.03993615319154358665
D_K=-660        H=8    h=1    N=4    C=0.03892494720807614855
D_K=-708        H=4    h=1    N=3    C=0.03758230140014144273
D_K=-715        H=4    h=1    N=3    C=0.03739787960033828712
D_K=-760        H=4    h=1    N=3    C=0.03627381250550058360
D_K=-795        H=4    h=1    N=3    C=0.03546634510659543259
D_K=-840        H=8    h=1    N=4    C=0.03450327796711771089
D_K=-1012       H=4    h=1    N=3    C=0.03143473067309657382
D_K=-1092       H=8    h=1    N=4    C=0.030261376633440120992
D_K=-1155       H=8    h=1    N=4    C=0.029424494316824984322
D_K=-1320       H=8    h=1    N=4    C=0.027524094128159015140
D_K=-1380       H=8    h=1    N=4    C=0.026919095102908275320
D_K=-1428       H=8    h=1    N=4    C=0.026462806201248157256
D_K=-1435       H=4    h=1    N=3    C=0.026398183867422732984
D_K=-1540       H=8    h=1    N=4    C=0.025482359571881277232
D_K=-1848       H=8    h=1    N=4    C=0.023262105259961771688
D_K=-1995       H=8    h=1    N=4    C=0.022388683141982250148
D_K=-3003       H=8    h=1    N=4    C=0.018248296715045297629
D_K=-3315       H=8    h=1    N=4    C=0.017368336857296871107
D_K=-5460       H=16   h=1    N=5    C=0.013533299049019169228
D_K=-31395      H=32   h=2    N=5    C=0.011287551685861596711
D_K=-33915      H=32   h=2    N=5    C=0.010860106519162152420
D_K=-40755      H=32   h=2    N=5    C=0.009906940323254157974
D_K=-94395      H=48   h=3    N=5    C=0.009764426918800617707
D_K=-106260     H=96   h=3    N=6    C=0.009203146787097731965
D_K=-145860     H=96   h=3    N=6    C=0.007855125898384663556
D_K=-264180     H=128  h=4    N=6    C=0.007782336824316934453
D_K=-298452     H=128  h=4    N=6    C=0.007321882321555281918
D_K=-304980     H=128  h=4    N=6    C=0.007243097165890304815
D_K=-323323     H=64   h=4    N=5    C=0.007034636991172515786
D_K=-435435     H=128  h=4    N=6    C=0.006061754228857454068
D_K=-915915     H=160  h=5    N=6    C=0.005224472151550085263
D_K=-1141140    H=320  h=5    N=7    C=0.004680589348015776317
D_K=-1381380    H=320  h=5    N=7    C=0.004254156107243915719
D_K=-3598980    H=512  h=8    N=7    C=0.004216967659631715196
D_K=-4555915    H=256  h=8    N=6    C=0.003748022424187745904
D_K=-6561555    H=512  h=8    N=7    C=0.003123105105098544600
D_K=-8843835    H=576  h=9    N=7    C=0.003026371188649244116
D_K=-16041795   H=768  h=12   N=7    C=0.002996089378541778906
D_K=-17852835   H=768  h=12   N=7    C=0.002840060884694504730
D_K=-19399380   H=1536 h=12   N=8    C=0.002724503191342580514
D_K=-23483460   H=1536 h=12   N=8    C=0.002476282585134087042
D_K=-33093060   H=1664 h=13   N=8    C=0.002259825416804010321
D_K=-40562340   H=1536 h=12   N=8    C=0.001884168528832614627
(...)
\end{verbatim}
\normalsize

\section{Programs and tables for real quadratic fields}\label{real}

\subsection{Successive maxima of $\frac{H}{2^{N-1}\cdot 
(\sqrt{\vert D_K \vert}\,)^{\varepsilon}}$ with various $\varepsilon$}
We have analogous properties for real quadratic fields and the reader
may use the following program to test various situations; here $H$ is 
the class number in the {\it restricted sense} to get $g_{K/\Q}^{}=2^{N-1}$:

\footnotesize
\smallskip
\begin{verbatim}
{LD=List;LN=List;LH=List;NL=0;bD=2;BD=10^5;
for(D=bD,BD,e2=valuation(D,2);d=D/2^e2;if(e2==1||e2>3||
core(d)!=d||(e2==0 & Mod(d,4)!=1)||(e2==2 & Mod(d,4)!=-1),next);
NL=NL+1;listput(LD,D);N=omega(D);listput(LN,N);P=x^2-D;
K=bnfinit(P,1);H=bnfnarrow(K)[1];listput(LH,H))}
\end{verbatim}
\normalsize

\smallskip
As for the imaginary case, one then uses any interval of values of $\varepsilon$:

\footnotesize
\smallskip
\begin{verbatim}
{for(k=0,100,eps=k*0.005;print(eps);Cm=0;for(j=1,NL,D=LD[j];N=LN[j];H=LH[j];
C=H/(2^(N-1)*(D^(eps/2)));if(C>Cm,Cm=C;if(N>1,h=H/2^(N-1);
print("D_K=",D," H=",H," h=",h," N=",N," C=",precision(C,1))))))}

{for(k=0,50,eps=k*0.01;print(eps);Cm=0;for(j=1,NL,D=LD[j];N=LN[j];H=LH[j];
C=H/(2^(N-1)*(D^(eps/2)));if(C>Cm,Cm=C;h=H/2^(N-1);
print("D_K=",D," H=",H," h=",h," N=",N," C=",precision(C,1)))))}
\end{verbatim}
\normalsize

When $\varepsilon < 0.42$, one finds only the exceptional counterexample
of composite discriminant:
$${\sf D_K=136 \ \  \ H=4 \ \ \  h=2  \ \ \ N=2 \ \  \ C=0.730557694812702258}$$

and no example with $N>1$ for larger $\varepsilon$ (in the selected interval):

\footnotesize
\smallskip
\begin{verbatim}
eps=0.41
D_K=5        H=1     h=1     N=1   C=0.718970628883420458
D_K=136      H=4     h=2     N=2   C=0.730557694812702258
D_K=229      H=3     h=3     N=1   C=0.984815739249975242
D_K=401      H=5     h=5     N=1   C=1.463273164377664277
D_K=577      H=7     h=7     N=1   C=1.901327800418162251
D_K=1129     H=9     h=9     N=1   C=2.130297468803580754
D_K=1297     H=11    h=11    N=1   C=2.530696066538532698
D_K=7057     H=21    h=21    N=1   C=3.413910809035837466
D_K=8761     H=27    h=27    N=1   C=4.198945711923644210
D_K=14401    H=43    h=43    N=1   C=6.039457956003442814
D_K=41617    H=57    h=57    N=1   C=6.440570637045096139
D_K=57601    H=63    h=63    N=1   C=6.659665172069346510
D_K=90001    H=87    h=87    N=1   C=8.392639118660242231
D_K=176401   H=153   h=153   N=1   C=12.85756466018982938
D_K=650281   H=207   h=207   N=1   C=13.31326654631031378
D_K=921601   H=235   h=235   N=1   C=14.07137574672269499
(...)
eps=0.42
D_K=5        H=1     h=1     N=1   C=0.713208152908235133
D_K=229      H=3     h=3     N=1   C=0.958419857815418096
D_K=401      H=5     h=5     N=1   C=1.420069782057472125
D_K=577      H=7     h=7     N=1   C=1.841836686780795434
D_K=1129     H=9     h=9     N=1   C=2.056727626379498179
D_K=1297     H=11    h=11    N=1   C=2.441604350895807169
D_K=7057     H=21    h=21    N=1   C=3.265946529044727296
D_K=8761     H=27    h=27    N=1   C=4.012614990621811735
D_K=14401    H=43    h=43    N=1   C=5.757129506348588758
D_K=41617    H=57    h=57    N=1   C=6.107001038822202258
D_K=57601    H=63    h=63    N=1   C=6.304494128338600865
D_K=90001    H=87    h=87    N=1   C=7.927336842220217845
D_K=176401   H=153   h=153   N=1   C=12.10392466127156707
D_K=650281   H=207   h=207   N=1   C=12.45142685877934641
D_K=921601   H=235   h=235   N=1   C=13.13753377133558054
(...)
\end{verbatim}
\normalsize

\smallskip
If $\varepsilon \to 0$, we find analogous sub-lists:

\footnotesize
\smallskip
\begin{verbatim}
eps=0.0000001
D_K=5        H=1     h=1     N=1   C=0.999999919528107616
D_K=136      H=4     h=2     N=2   C=1.999999508734571761
D_K=229      H=3     h=3     N=1   C=2.999999184941810187
D_K=401      H=5     h=5     N=1   C=4.999998501509867721
D_K=577      H=7     h=7     N=1   C=6.999997774755560416
D_K=1129     H=9     h=9     N=1   C=8.999996836911151974
D_K=1297     H=11    h=11    N=1   C=10.99999605770565506
D_K=7057     H=21    h=21    N=1   C=20.99999069513798489
D_K=8761     H=27    h=27    N=1   C=26.99998774461458228
D_K=14401    H=43    h=43    N=1   C=42.99997941364113356
D_K=32401    H=45    h=45    N=1   C=44.99997663163079519
D_K=41617    H=57    h=57    N=1   C=56.99996968665561316
D_K=57601    H=63    h=63    N=1   C=62.99996547192955773
D_K=90001    H=87    h=87    N=1   C=86.99995037705828937
D_K=176401   H=153   h=153   N=1   C=152.9999075840874610
D_K=650281   H=207   h=207   N=1   C=206.9998614636418471
D_K=921601   H=235   h=235   N=1   C=234.9998386271104724
(...)
\end{verbatim}
\normalsize

\smallskip
For the real quadratic fields the conjecture may have only
the exception $D_K=136$ (with $N=2$), but $\varepsilon = 0.0000001$
and we will now go to the limit.
For a more complete table, see Table III, Appendix \ref{Table3}.

\subsection{Successive maxima of $H$ and $\frac{H}{2^{N-1}}$}
In this subsection we put $\varepsilon = 0$ to study the successive
maxima of $H$ (with the the corresponding value of~$h$), then the 
successive maxima of $h$:

\footnotesize
\smallskip
\begin{verbatim}
Successive maxima of H
{Hm=0;bD=2;BD=10^9;for(D=bD,BD,e2=valuation(D,2);d=D/2^e2;
if(e2==1||e2>3||core(d)!=d||(e2==0&Mod(d,4)!=1)||(e2==2&Mod(d,4)!=-1),
next);P=x^2-D;K=bnfinit(P,1);H=bnfnarrow(K)[1];if(H>Hm,Hm=H;N=omega(D);
print("D_K=",D," H=",H," h=",H/2^(N-1)," N=",N)))}
D_K=5          H=1       h=1      N=1
D_K=12         H=2       h=1      N=2
D_K=60         H=4       h=1      N=3
D_K=316        H=6       h=3      N=2
D_K=505        H=8       h=4      N=2
D_K=817        H=10      h=5      N=2
D_K=940        H=12      h=3      N=3
D_K=1596       H=16      h=2      N=4
D_K=3129       H=20      h=5      N=3
D_K=4009       H=22      h=11     N=2
D_K=4345       H=24      h=6      N=3
D_K=6097       H=28      h=7      N=3
D_K=11289      H=36      h=9      N=3
D_K=13321      H=52      h=13     N=3
D_K=17889      H=60      h=15     N=3
D_K=34689      H=64      h=16     N=3
D_K=44065      H=68      h=17     N=3
D_K=49321      H=92      h=23     N=3
D_K=54769      H=100     h=25     N=3
D_K=104017     H=104     h=26     N=3
D_K=108889     H=140     h=35     N=3
D_K=176401     H=153     h=153    N=1
D_K=181689     H=160     h=40     N=3
D_K=247921     H=240     h=60     N=3
D_K=318049     H=244     h=61     N=3
D_K=365521     H=252     h=63     N=3
D_K=381049     H=262     h=131    N=2
D_K=685561     H=268     h=67     N=3
D_K=697921     H=376     h=94     N=3
D_K=767449     H=392     h=196    N=2
D_K=1299505    H=448     h=112    N=3
D_K=1483321    H=456     h=114    N=3
D_K=1686481    H=474     h=237    N=2
D_K=1713961    H=478     h=239    N=2
D_K=1758289    H=492     h=123    N=3
D_K=2191249    H=504     h=126    N=3
D_K=2259961    H=596     h=149    N=3
D_K=2419321    H=600     h=150    N=3
D_K=2500681    H=612     h=306    N=2
D_K=2782641    H=616     h=154    N=3
D_K=2784721    H=632     h=158    N=3
D_K=3755521    H=752     h=47     N=5
D_K=3992041    H=780     h=195    N=3
D_K=4232449    H=880     h=220    N=3
D_K=5247841    H=984     h=246    N=3
D_K=5674321    H=1128    h=564    N=2
D_K=7064521    H=1180    h=295    N=3
(...)
\end{verbatim}
\normalsize

We note the exceptional case of $D_K= 3755521=
7 \cdot 11 \cdot 17 \cdot 19 \cdot 151$
for which $H=752$, $h=47$, $N=5$.

\footnotesize
\smallskip
\begin{verbatim}
Successive maxima of h
{hm=0;bD=2;BD=10^9;for(D=bD,BD,e2=valuation(D,2);d=D/2^e2;
if(e2==1||e2>3||core(d)!=d||(e2==0&Mod(d,4)!=1)||(e2==2&Mod(d,4)!=-1),
next);P=x^2-D;K=bnfinit(P,1);H=bnfnarrow(K)[1];N=omega(D);h=H/2^(N-1);
if(h>hm,hm=h;print("D_K=",D," H=",H," h=",h," N=",N)))}
D_K=5          H=1       h=1      N=1
D_K=136        H=4       h=2      N=2
D_K=229        H=3       h=3      N=1
D_K=401        H=5       h=5      N=1
D_K=577        H=7       h=7      N=1
D_K=1129       H=9       h=9      N=1
D_K=1297       H=11      h=11     N=1
D_K=7057       H=21      h=21     N=1
D_K=8761       H=27      h=27     N=1
D_K=14401      H=43      h=43     N=1
D_K=32401      H=45      h=45     N=1
D_K=41617      H=57      h=57     N=1
D_K=57601      H=63      h=63     N=1
D_K=90001      H=87      h=87     N=1
D_K=176401     H=153     h=153    N=1
D_K=650281     H=207     h=207    N=1
D_K=921601     H=235     h=235    N=1
D_K=1299601    H=357     h=357    N=1
D_K=2944657    H=377     h=377    N=1
D_K=3686401    H=455     h=455    N=1
D_K=3920401    H=575     h=575    N=1
D_K=7290001    H=655     h=655    N=1
(...)
\end{verbatim}
\normalsize

\section{Cyclic cubic fields -- Successive maxima of $\frac{H}{3^{N-1}
 \cdot  (\sqrt{D}\,)^\varepsilon}$}
The program computing the list of cyclic cubic fields with increasing
conductors $f$ (in ${\sf f}$) is for instance given in \cite[\S\,6.1]{Gr2}.
We then obtain the following program for the computation of the successive
maxima of $\frac{H}{3^{N-1} \cdot f^\varepsilon}$ (using $\sqrt D = f$), 
where $H$ is the class number in the ordinary sense and $N = \omega(f)$; 
this is equivalent to compute the successive maxima of the numbers
$C_\varepsilon(D) =  \sup_{K \in {\mathcal F}_{3, f}}^{}
\big(\frac{H}{3^{N-1} \cdot f^\varepsilon} \big)$ over increasing conductors.

\smallskip
When $N > 1$, we do not consider in this program the mean of the class 
numbers $h$ and $H$ for the $2^{N-1}$ fields of same conductor $f$ 
and compute the maxima inside this set of fields ${\mathcal F}_{3, f}$; 
which shall give some defects of the method, but a forthcomming program, 
for degree $p$ cyclic fields, $p\geq 3$ (see \S\,\ref{means}), shall compute 
the means $C_\varepsilon({\mathcal F}_{3, f})$ of the 
$\frac{H}{2^{N-1} \cdot (\sqrt{D}\,)^\varepsilon}$, $K \in {\mathcal F}_{3, f}$.

\footnotesize
\begin{verbatim}
{eps=0.1;bf=7;Bf=10^8;Cm=0; \\ successive conductors:
for(f=bf,Bf,phi=eulerphi(f);if(Mod(phi,3)!=0,next);
e3=valuation(f,3);ff=f/3^e3;if(e3==1||e3>2||core(ff)!=ff,next);
\\ test of q|ff congruent to 1 (mod 3):
F=factor(ff);Div=component(F,1);S=0;dF=matsize(F)[1];
for(j=1,dF,q=Div[j];if(Mod(q,3)!=1,S=1;break));
\\ list of degree 3 polynomials of conductor f:
if(S==0,N=omega(f); \\ f is the conductor to be put as (a^2+27*b^2)/4
for(b=1,sqrt(4*f/27),if(e3==2 & Mod(b,3)==0,next);
A=4*f-27*b^2;if(issquare(A,&a)==1, \\ expression of a defining polynomial:
if(e3==0,if(Mod(a,3)==1,a=-a);P=x^3+x^2+(1-f)/3*x+(f*(a-3)+1)/27);
if(e3==2,if(Mod(a,9)==3,a=-a);P=x^3-f/3*x-f*a/27);
K=bnfinit(P,1);H=K.no;h=H/3^(N-1);C=h/f^eps; if(C>Cm,Cm=C;
print("f_K=",f," H=",H," h=",h," N=",N," C=",precision(C,1)))))))}
eps=0.1
f=7        H=1      h=1     N=1    C=0.823171253993044237
f=163      H=4      h=4     N=1    C=2.403483975167433545
f=313      H=7      h=7     N=1    C=3.940429903816599872
f=1063     H=13     h=13    N=1    C=6.475749278819724722
f=1489     H=19     h=19    N=1    C=9.150906778915694477
f=2763     H=63     h=21    N=2    C=9.507825584310405539
f=4219     H=28     h=28    N=1    C=12.15170110906513844
f=4867     H=228    h=76    N=2    C=32.51527696336010219
f=11491    H=109    h=109   N=1    C=42.79477185917101752
f=15013    H=127    h=127   N=1    C=48.54639357681429304
f=23173    H=349    h=349   N=1    C=127.7400849965886920
f=74419    H=688    h=688   N=1    C=224.0887506125577062
f=165889   H=2352   h=784   N=2    C=235.6862811297153681
f=191413   H=961    h=961   N=1    C=284.7909895844608044
f=241117   H=1216   h=1216  N=1    C=352.1362960501321708
f=294313   H=2107   h=2107  N=1    C=598.1134808281728111
f=348457   H=3391   h=3391  N=1    C=946.4831145084927291
f=888313   H=3484   h=3484  N=1    C=3436.604704785608734
f=1066063  H=4087   h=4087  N=1    C=4030.666398029350791
f=1074607  H=4681   h=4681  N=1    C=4616.442084136778108
f=1339813  H=4819   h=4819  N=1    C=4751.490684118511988
f=1543813  H=6223   h=6223  N=1    C=6134.952524706830718
f=1632013  H=6277   h=6277  N=1    C=6187.844694562183552
f=1653919  H=15307  h=15307 N=1    C=15089.38598754532585
f=3046471  H=43216  h=43216 N=1    C=42575.59929298335085
f=6575641  H=46108  h=46108 N=1    C=45389.80807553668707
(...)
\end{verbatim}
\normalsize

For the conductor $f=165889 = 19 \cdot 8731$, we have the following 
information for the two cubic fields of conductor $f$, which perhaps explains 
the exception and suggests the method that we will perform in the next section:

\footnotesize
\begin{verbatim}
x^3+x^2-55296*x+3809303   H=3      h=1    N=2   C=0.300620256543004  
x^3+x^2-55296*x-1996812   H=2352   h=784  N=2   C=235.6862811297153
\end{verbatim}
\normalsize

Indeed, the program considers the large local maximum $H=2352$
while the other field has class number $3$.

\smallskip
The same program for $p=3$, computing the mean values over the
sets of fields $K$ having the same conductor $f$, would be the following
(the new list obtained does not contain data with $N=2$):

\footnotesize
\begin{verbatim}
{eps=0.1;bf=7;Bf=10^8;Cm=0; \\ successive conductors:
for(f=bf,Bf,phi=eulerphi(f);if(Mod(phi,3)!=0,next);
e3=valuation(f,3);ff=f/3^e3;if(e3==1||e3>2||core(ff)!=ff,next);
\\ test of q|ff congruent to 1 (mod p):
F=factor(ff);Div=component(F,1);S=0;dF=matsize(F)[1];
for(j=1,dF,q=Div[j];if(Mod(q,3)!=1,S=1;break));
\\ list of degree 3 polynomials of conductor f:
if(S==0,N=omega(f);HH=1;hh=1;CC=1.0;
for(b=1,sqrt(4*f/27),if(e3==2 & Mod(b,3)==0,next);
A=4*f-27*b^2;if(issquare(A,&a)==1, 
if(e3==0,if(Mod(a,3)==1,a=-a);P=x^3+x^2+(1-f)/3*x+(f*(a-3)+1)/27);
if(e3==2,if(Mod(a,9)==3,a=-a);P=x^3-f/3*x-f*a/27);
K=bnfinit(P,1);D=K.disc;H=K.no;h=H/3^(N-1);C=h/f^eps;
HH=HH*H;hh=hh*h;CC=CC*C)));nK=2^(N-1);
C=CC^(1/nK); \\ mean value of CC and extrema test:
if(C>Cm,Cm=C;H=HH^(1/nK);h=hh^(1/nK); \\ mean values of H,h
print("f=",factor(f)," P=",P," H=",precision(H,1),
" h=",precision(h,1)," N=",N," C=",precision(C,1))))}
\end{verbatim}
\normalsize

One verifies that this program gives the same results and is faster 
than the one we will give in the next subection for any $p \geq 3$.

\section{General approach of the degree $p$ cyclic case, $p>2$}

For these degree $p$ cyclic fields $K$ (thus real), except the cubic 
case, a definig polynomial of $K$ corresponding to increasing conductors
${\sf f}$ is only possible, via PARI/GP, by using ${\sf polsubcyclo(f,p)}$,
and selecting the polynomials ${\sf P}$ giving the conductor ${\sf f}$.

\begin{remark}{\rm 
We know that $f = (p^2)^\delta \cdot q_1 \ldots q_n$, for primes 
$q_i \equiv 1 \pmod p$ and $\delta \in \{0,1\}$.
If $N = n + \delta$ is the number of ramified primes, 
there are $\frac{p^N-1}{p-1}$ fields whose conductor $f'$ divides $f$ 
and $(p-1)^{N-1}$ fields $K$ of exact conductor $f$, a priori indistinguishable.
So there are various methods to built a ``mean value'' (in the 
multiplicative sense) for the computations of successive maxima of 
$C_\varepsilon$.

\smallskip
Let $K(f)$ be the compositum of the degree $p$ cyclic fields of 
conductors $f' \mid f$; then $K(f)$ is the subfield of $\Q(\mu_f)$
fixed by $\Gamma^p$, where $\Gamma = {\rm Gal}(\Q(\mu_f)/\Q)$.
Considering that the set of rational characters $\chi$ of order $p$ of
$\Gamma$ defines the set of degree $p$ cyclic subfields 
$K_\chi = K^{{\rm Ker}(\chi)}$) 
of $K(f)$, we get that, rouhgly speaking, the class 
number of $K(f)$ is equal to the product of the class numbers of the 
$K_\chi$ (this is true for the non-$p$-parts), which gives a serious reason to 
take into account, for any family of numbers $(X_K)_{K \in {\mathcal F}'_{p, f}}$,
the standard mean value $\Big[\prod_{K \in {\mathcal F}'_{p, f}}
X_K \Big]^{\frac{1}{\order {\mathcal F}'{}_{\!\!\! p, f} }}$,
or the analogous one replacing ${\mathcal F}'_{p, f}$ by 
${\mathcal F}_{p, f}$ (cf. Conjecture \ref{mainconj}).}
\end{remark}

\smallskip
In the programs, ${\sf p}$ is the prime considered, ${\sf factor(f)}$
gives the set of ramified primes, 
${\sf P}$ is a defining polynomial of the field,
${\sf H}$ its class number in the ordinary 
sense, ${\sf h} := \frac{H}{p^{N-1}}$ the non-genus part, 
and ${\sf C}=C_\varepsilon(D)$. 
The programs compute the two previous multiplicative mean values 
of ${\sf H}$, ${\sf h}$ and ${\sf C}$, taking into account the two following
remarks:

\smallskip
(i) We consider, first, only the fields of exact conductor $f$; so in that case
the number $N$ (then the genus number) is the same for all these fields
and makes sense in the results; but if a local maximum is reached by a 
composite conductor, the mean values of ${\sf H}$ and ${\sf h}$ can be 
non-integers. Otherwise, if the maximum is reached by a prime 
conductor, one may write the corresponding polynomial which is unique.

\smallskip
(ii) Then we consider all the fields whose conductor $f'$ divides $f$;
in that case $N$ and $P$ depend on each field and the program gives 
only the mean values of ${\sf H}$, ${\sf h}$ and ${\sf C}$.

\subsection{Fields of conductor $f$: mean value 
of \footnotesize $\frac{H}{p^{N-1} \cdot (\sqrt{D}\,)^\varepsilon}$}\label{means}

Thus $\sqrt D=f^{\frac{p-1}{2}}$; since PARI/GP
gives the $\frac{p^N-1}{p-1}$ polynomials (defining all the fields of degree $p$
{\it contained in $\Q(\mu_f)$}) in the vector ${\sf V=Mat(polsubcyclo(f,p))}$,
we must test, by means of the discriminants, using the function 
${\sf nfdisc}$, if the field defined by $P$ is of exact conductor $f$. 

\smallskip
The results show that one always get $N=1$ and 
for instance, for $p=3$, the conductor $165889=19 \cdot 8731$ 
does not appear as a maximum. Since all results are with prime
conductors, we write the corresponding (unique) polynomial; this makes 
sense only in that case:

\footnotesize
\begin{verbatim}
{eps=0.01;p=3;Cm=0; \\ successive conductors:
for(f=7,10^6,phi=eulerphi(f);if(Mod(phi,p)!=0,next);
ep=valuation(f,p);ff=f/p^ep;if(ep==1||ep>2||core(ff)!=ff,next);
\\ test of q|ff congruent to 1 (mod p):
F=factor(ff);Div=component(F,1);S=0;dF=matsize(F)[1];
for(j=1,dF,q=Div[j];if(Mod(q,p)!=1,S=1;break));
\\ list of degree p polynomials in the fth cyclotomic field:
if(S==0,V=Mat(polsubcyclo(f,p));dV=matsize(V)[2];HH=1;hh=1;CC=1.0;
for(k=1,dV,P=component(V,k)[1];D=nfdisc(P);N=omega(f);
\\ test conductor(K)=f:
if(omega(D)==N,K=bnfinit(P,1);H=K.no;h=H/p^(N-1);C=h/D^(eps/2);
HH=HH*H;hh=hh*h;CC=CC*C));nK=(p-1)^(N-1);
C=CC^(1/nK); \\ mean value of CC and extrema test:
if(C>Cm,Cm=C;H=HH^(1/nK);h=hh^(1/nK);\\ mean values of H,h
print("f=",factor(f)," P=",P," H=",precision(H,1),
" h=",precision(h,1)," N=",N," C=",precision(C,1)))))}

p=3
f=[7,1]      P=x^3+x^2-2*x-1             H=1    h=1    N=1 C=0.9807290047229
f=[163,1]    P=x^3+x^2-54*x-169          H=4    h=4    N=1 C=3.8013522515881
f=[313,1]    P=x^3+x^2-104*x+371         H=7    h=7    N=1 C=6.6091041630333
f=[1063,1]   P=x^3+x^2-354*x+2441        H=13   h=13   N=1 C=12.124895929808
f=[1489,1]   P=x^3+x^2-496*x+4081        H=19   h=19   N=1 C=17.661380777478
f=[4219,1]   P=x^3+x^2-1406*x-625        H=28   h=28   N=1 C=25.757632116142
f=[5119,1]   P=x^3+x^2-1706*x+26543      H=31   h=31   N=1 C=28.462290210252
f=[7351,1]   P=x^3+x^2-2450*x-1089       H=49   h=49   N=1 C=44.826271355692
f=[9511,1]   P=x^3+x^2-3170*x+65168      H=73   h=73   N=1 C=66.610178628175
f=[11491,1]  P=x^3+x^2-3830*x-89800      H=109  h=109  N=1 C=99.271119537133
f=[15013,1]  P=x^3+x^2-5004*x+134561     H=127  h=127  N=1 C=115.35569659944
f=[23173,1]  P=x^3+x^2-7724*x-258336     H=349  h=349  N=1 C=315.62805883284
f=[74419,1]  P=x^3+x^2-24806*x-11025     H=688  h=688  N=1 C=614.99502127609
f=[191413,1] P=x^3+x^2-63804*x+6181931   H=961  h=961  N=1 C=850.94927649437
f=[241117,1] P=x^3+x^2-80372*x-35721     H=1216 h=1216 N=1 C=1074.2646788161
f=[294313,1] P=x^3+x^2-98104*x+11794321  H=2107 h=2107 N=1 C=1857.7036788965
f=[348457,1] P=x^3+x^2-116152*x-15190144 H=3391 h=3391 N=1 C=2984.7385706495
(...)
\end{verbatim}
\normalsize

Same program with $\varepsilon=0.4$:

\footnotesize
\begin{verbatim}
f=[7,1]      P=x^3+x^2-2*x-1             H=1    h=1    N=1 C=0.4591565499594
f=[163,1]    P=x^3+x^2-54*x-169          H=4    h=4    N=1 C=0.5214167154550
f=[313,1]    P=x^3+x^2-104*x+371         H=7    h=7    N=1 C=0.7028785742753
f=[1063,1]   P=x^3+x^2-354*x+2441        H=13   h=13   N=1 C=0.8004423281900
f=[1489,1]   P=x^3+x^2-496*x+4081        H=19   h=19   N=1 C=1.0223408675783
f=[7351,1]   P=x^3+x^2-2450*x-1089       H=49   h=49   N=1 C=1.3920585783307
f=[9511,1]   P=x^3+x^2-3170*x+65168      H=73   h=73   N=1 C=1.8708216514284
f=[11491,1]  P=x^3+x^2-3830*x-89800      H=109  h=109  N=1 C=2.5899022210956
f=[15013,1]  P=x^3+x^2-5004*x+134561     H=127  h=127  N=1 C=2.7115444434617
f=[23173,1]  P=x^3+x^2-7724*x-258336     H=349  h=349  N=1 C=6.2637226912587
f=[74419,1]  P=x^3+x^2-24806*x-11025     H=688  h=688  N=1 C=7.7431006805904
f=[241117,1] P=x^3+x^2-80372*x-35721     H=1216 h=1216 N=1 C=8.5515028017316
f=[294313,1] P=x^3+x^2-98104*x+11794321  H=2107 h=2107 N=1 C=13.681721863115
f=[348457,1] P=x^3+x^2-116152*x-15190144 H=3391 h=3391 N=1 C=20.581077709540
(...)
\end{verbatim}
\normalsize

\subsection{Fields of conductors $f' \!\mid \!f$: mean value
of \footnotesize $\frac{H}{p^{N-1} \cdot (\sqrt{D}\,)^\varepsilon}$}

The program is the same as the previous one without testing the conductors.

\smallskip
The mean value $C_\varepsilon$ (in ${\sf C}$) is computed 
over all the $C_\varepsilon(D) = \frac{H}{p^{N-1} \cdot 
(\sqrt{D}\,)^\varepsilon}$.

\smallskip
If ${\sf H}$ and ${\sf h}$ are the corresponding mean values, we will see that 
${\sf H=h}$ for the successive maxima of these mean values in the selected intervals,
which indicates that the maxima are reached for prime conductors
(i.e., ${\sf dV}=1$):

\footnotesize
\smallskip
\begin{verbatim}
{eps=0.02;p=3;Cm=0; \\ successive conductors:
for(f=7,10^6,phi=eulerphi(f);if(Mod(phi,p)!=0,next);
ep=valuation(f,p);ff=f/p^ep;if(ep==1||ep>2||core(ff)!=ff,next);
 \\ test of q|ff congruent to 1 (mod p):
F=factor(ff);Div=component(F,1);S=0;dF=matsize(F)[1];
for(j=1,dF,q=Div[j];if(Mod(q,p)!=1,S=1;break));
\\ list of degree p polynomials in the fth cyclotomic field:
if(S==0,V=Mat(polsubcyclo(f,p));dV=matsize(V)[2];HH=1;hh=1;CC=1;
for(k=1,dV,P=component(V,k)[1];K=bnfinit(P,1);
D=K.disc;N=omega(D);H=K.no;h=H/p^(N-1);C=h/D^(eps/2);
HH=HH*H;hh=hh*h;CC=CC*C);
C=CC^(1/dV); \\ mean value of CC and extrema test:
if(C>Cm,Cm=C;H=HH^(1/dV);h=hh^(1/dV); \\ mean values of H,h
print("f=",factor(f)," H=",precision(H,1),
" h=",precision(h,1)," C=",precision(C,1)))))}
p=3
f=[7,1]       H=1      h=1       C=0.961829380704799441
f=[163,1]     H=4      h=4       C=3.612569735163622387
f=[313,1]     H=7      h=7       C=6.240036833974918259
f=[1063,1]    H=13     h=13      C=11.30870010066782027
f=[1489,1]    H=19     h=19      C=16.41707215616319439*
f=[4219,1]    H=28     h=28      C=23.69484329394764450***
f=[5119,1]    H=31     h=31      C=26.13232141976207791***
f=[7351,1]    H=49     h=49      C=41.00805313579896863*
f=[9511,1]    H=73     h=73      C=60.77966982023849296
f=[11491,1]   H=109    h=109     C=90.41059792803446483
f=[15013,1]   H=127    h=127     C=104.7790294326226146
f=[23173,1]   H=349    h=349     C=285.4471963397972172
f=[74419,1]   H=688    h=688     C=549.7367386546242885
f=[191413,1]  H=961    h=961     C=753.5012186954210082
f=[241117,1]  H=1216   h=1216    C=949.0498356513206700
f=[294313,1]  H=2107   h=2107    C=1637.903634829544508
f=[348457,1]  H=3391   h=3391    C=2627.149612245101090
(...)
\end{verbatim}
\normalsize

Same program wit  h $\varepsilon=0.4$: 

\footnotesize
\begin{verbatim}
f=[7,1]       H=1      h=1       C=0.459156549959434092
f=[163,1]     H=4      h=4       C=0.521416715455083946
f=[313,1]     H=7      h=7       C=0.702878574275380209
f=[1063,1]    H=13     h=13      C=0.800442328190010161
f=[1489,1]    H=19     h=19      C=1.022340867578304918*
f=[7351,1]    H=49     h=49      C=1.392058578330795015*
f=[9511,1]    H=73     h=73      C=1.870821651428443619
f=[11491,1]   H=109    h=109     C=2.589902221095671686
f=[15013,1]   H=127    h=127     C=2.711544443461731055
f=[23173,1]   H=349    h=349     C=6.263722691258787681
f=[74419,1]   H=688    h=688     C=7.743100680590427658
f=[241117,1]  H=1216   h=1216    C=8.551502801731692495
f=[294313,1]  H=2107   h=2107    C=13.68172186311554962
f=[348457,1]  H=3391   h=3391    C=20.58107770954063231
(...)
\end{verbatim}
\normalsize

We remark an example of different lists according to $\varepsilon$
(cf. $1489$ and $7351$).

\subsection{Successive maxima of the mean value of 
{\footnotesize $\frac{H}{(\sqrt{D}\,)^\varepsilon}$}}

In that case, the parameters $N$, $P$ do not make sense for non-prime $f$.

\subsubsection{$G \simeq \Z/p\Z$. Fields of conductor $f$ and mean value 
of \footnotesize $\frac{H}{(\sqrt{D}\,)^\varepsilon}$.}
${}$
\footnotesize 
\begin{verbatim}
{eps=0.02;p=3;Cm=0; \\ successive conductors:
for(f=1,10^6,phi=eulerphi(f);if(Mod(phi,p)!=0,next);
ep=valuation(f,p);ff=f/p^ep;if(ep==1||ep>2||core(ff)!=ff,next);
\\ test of q|ff congruent to 1 (mod p):
F=factor(ff);Div=component(F,1);S=0;dF=matsize(F)[1];
for(j=1,dF,q=Div[j];if(Mod(q,p)!=1,S=1;break));
\\ list of degree p polynomials in the fth cyclotomic field:
if(S==0,V=Mat(polsubcyclo(f,p));dV=matsize(V)[2];HH=1;hh=1;CC=1;
for(k=1,dV,P=component(V,k)[1];D=nfdisc(P);N=omega(f);
\\ test of ramification:
if(omega(D)==N,K=bnfinit(P,1);H=K.no;C=H/D^(eps/2);
HH=HH*H;CC=CC*C));nK=(p-1)^(N-1);
C=CC^(1/nK); \\ mean value of CC and extrema test:
if(C>Cm,Cm=C;H=HH^(1/nK); \\ mean of H
print("f=",factor(f)," H=",precision(H,1)," C=",precision(C,1)))))}
p=3
f=[7,1]          H=1      C=0.961829380704799441
f=[3,2; 7,1]     H=3.0    C=2.761432575434841409
f=[163,1]        H=4      C=3.612569735163622387
f=[313,1]        H=7      C=6.240036833974918259
f=[3,2; 73,1]    H=9.0    C=7.904805747328513653
f=[1063,1]       H=13     C=11.30870010066782027
f=[7,1; 181,1]   H=18.0   C=15.60331847344150489
f=[1489,1]       H=19     C=16.41707215616319439
f=[3,2; 307,1]   H=23.811761799581315314    
                          C=20.32188136664178736
f=[4219,1]       H=28     C=23.69484329394764450
f=[5119,1]       H=31     C=26.13232141976207791
f=[7351,1]       H=49     C=41.00805313579896863
f=[9511,1]       H=73     C=60.77966982023849296
f=[11491,1]      H=109    C=90.41059792803446483
f=[15013,1]      H=127    C=104.7790294326226146
f=[23173,1]      H=349    C=285.4471963397972172
f=[74419,1]      H=688    C=549.7367386546242885
f=[191413,1]     H=961    C=753.5012186954210082
f=[241117,1]     H=1216   C=949.0498356513206700
f=[294313,1]     H=2107   C=1637.903634829544508
f=[348457,1]     H=3391   C=2627.149612245101090
(...)
\end{verbatim}

\normalsize
\smallskip
Same program with $\varepsilon=0.2$:

\footnotesize
\smallskip
\begin{verbatim}
f=[7,1]          H=1      C=0.677610913400480949
f=[3,1; 7,1]     H=3      C=1.309945251235621361
f=[163,1]        H=4      C=1.444183804721662077
f=[313,1]        H=7      C=2.218141118127442645
f=[3,1; 73,1]    H=9      C=2.458835278103877640
f=[1063,1]       H=13     C=3.225794517087244988
f=[7,1; 181,1]   H=18     C=4.312381578109574420
f=[1489,1]       H=19     C=4.407320782968695309
f=[3,1; 307,1]   H=23.811761799581315314   
                          C=4.881073557098234152
f=[4219,1]       H=28     C=5.273708565859104133
f=[5119,1]       H=31     C=5.617261834937099453
f=[7351,1]       H=49     C=8.258987246521752304
f=[9511,1]       H=73     C=11.68631595303996494
f=[11491,1]      H=109    C=16.80176604108711584
f=[15013,1]      H=127    C=18.55710495523587440
f=[23173,1]      H=349    C=46.75509832359800418
f=[74419,1]      H=688    C=72.98803510333878161
f=[191413,1]     H=961    C=84.39740660613575690
f=[241117,1]     H=1216   C=101.9736603584756001
f=[294313,1]     H=2107   C=169.7863008772629543
f=[348457,1]     H=3391   C=264.1787927011785785
(...)
\end{verbatim}
\normalsize
 
As expected, there are some composite conductors giving successive maxima.

\subsubsection{$G \simeq \Z/p\Z$. Fields of conductors $f' \mid f$ and mean 
value of \footnotesize $\frac{H}{(\sqrt{D}\,)^\varepsilon}$.}
${}$
\footnotesize
\begin{verbatim}
{eps=0.02;p=3;Cm=0;
for(f=7,10^6,phi=eulerphi(f);if(Mod(phi,p)!=0,next);
ep=valuation(f,p);ff=f/p^ep;if(ep==1||ep>2||core(ff)!=ff,next);
F=factor(ff);Div=component(F,1);S=0;dF=matsize(F)[1];
for(j=1,dF,q=Div[j];if(Mod(q,p)!=1,S=1;break));
if(S==0,V=Mat(polsubcyclo(f,p));dV=matsize(V)[2];HH=1;CC=1;
for(k=1,dV,P=component(V,k)[1];K=bnfinit(P,1);
D=K.disc;H=K.no;C=H/D^(eps/2);HH=HH*H;CC=CC*C);
C=CC^(1/dV); \\ mean value of CC and extrema test:
if(C>Cm,Cm=C;H=HH^(1/dV);
print("f=",factor(f)," H=",precision(H,1)," C=",precision(C,1)))))}
p=3
f=[7,1]       H=1      C=0.961829380704799441
f=[3,2; 7,1]  H=1.7320508075688772936    
                       C=1.627685591700590660
f=[163,1]     H=4      C=3.612569735163622387
f=[313,1]     H=7      C=6.240036833974918259
f=[1063,1]    H=13     C=11.30870010066782027
f=[1489,1]    H=19     C=16.41707215616319439
f=[4219,1]    H=28     C=23.69484329394764450
f=[5119,1]    H=31     C=26.13232141976207791
f=[7351,1]    H=49     C=41.00805313579896863
f=[9511,1]    H=73     C=60.77966982023849296
f=[11491,1]   H=109    C=90.41059792803446483
f=[15013,1]   H=127    C=104.7790294326226146
f=[23173,1]   H=349    C=285.4471963397972172
f=[74419,1]   H=688    C=549.7367386546242885
f=[191413,1]  H=961    C=753.5012186954210082
f=[241117,1]  H=1216   C=949.0498356513206700
f=[294313,1]  H=2107   C=1637.903634829544508
f=[348457,1]  H=3391   C=2627.149612245101090
(...)
\end{verbatim}

\normalsize
Same program with $\varepsilon=0.2$:

\footnotesize
\begin{verbatim}
f=[7,1]       H=1      C=0.677610913400480949
f=[3,1; 7,1]  H=1.7320508075688772936    
                       C=0.930378241131994665
f=[163,1]     H=4      C=1.444183804721662077
f=[313,1]     H=7      C=2.218141118127442645
f=[1063,1]    H=13     C=3.225794517087244988
f=[1489,1]    H=19     C=4.407320782968695309
f=[4219,1]    H=28     C=5.273708565859104133
f=[5119,1]    H=31     C=5.617261834937099453
f=[7351,1]    H=49     C=8.258987246521752304
f=[9511,1]    H=73     C=11.68631595303996494
f=[11491,1]   H=109    C=16.80176604108711584
f=[15013,1]   H=127    C=18.55710495523587440
f=[23173,1]   H=349    C=46.75509832359800418
f=[74419,1]   H=688    C=72.98803510333878161
f=[191413,1]  H=961    C=84.39740660613575690
f=[241117,1]  H=1216   C=101.9736603584756001
f=[294313,1]  H=2107   C=169.7863008772629543
f=[348457,1]  H=3391   C=264.1787927011785785
\end{verbatim}
\normalsize

\subsection{Analogous tables for $p=5$}
The computations need more time but we obtain similar results. 
The polynomial ${\sf P}$ makes sense because $N=1$ in all the 
successive maxima given by the program.

\subsubsection{$G \simeq \Z/p\Z$. Fields of conductor $f$ and mean value of \footnotesize 
$\frac{H}{p^{N-1} \cdot (\sqrt{D}\,)^\varepsilon}$.}
${}$
\footnotesize
\smallskip
\begin{verbatim}
p=5  eps=0.01
f=[11,1]      P=x^5+x^4-4*x^3-3*x^2+3*x+1 
                        H=1        h=1        N=1  C=0.953173909654138686
f=[191,1]     P=x^5+x^4-76*x^3-359*x^2-437*x-155 
                        H=11       h=11       N=1  C=9.903119457621479703
f=[941,1]     P=x^5+x^4-376*x^3+3877*x^2-13445*x+15271 
                        H=16       h=16       N=1  C=13.95237660139343057
f=[3931,1]    P=x^5+x^4-1572*x^3+36637*x^2-309637*x+899921 
                        H=256      h=256      N=1  C=216.9451294678943811
f=[5051,1]    P=x^5+x^4-2020*x^3-55965*x^2-558681*x-1920311 
                        H=1451     h=1451     N=1  C=1223.488425196746388
f=[80251,1]   P=x^5+x^4-32100*x^3-3617715*x^2-151594781*x-2241837391 
                        H=37631    h=37631    N=1  C=30023.18416657978783
f=[154291,1]  P=x^5+x^4-61716*x^3+9584557*x^2-555677185*x+11397220121 
                        H=108691   h=108691   N=1  C=85590.73829767367053
f=[349211,1]  P=x^5+x^4-139684*x^3-32923613*x^2-2897797577*x-90346698359 
                        H=186091   h= 186091  N=1  C=144166.2531422279504
f=[555671,1]  P=x^5+x^4-222268*x^3+65858127*x^2-7299765465*x+286894118701 
                        H=721151   h= 721151  N=1  C=553515.5874018525546
f=[641491,1]  P=x^5+x^4-256596*x^3-82033869*x^2-9804476597*x-415497361495 
                        H=1566401  h= 1566401 N=1  C=1198834.245498212652
f=[1464901,1] P=x^5+x^4-585960*x^3+282608701*x^2-51039799381*x+3272126756879 
                        H=4628591  h= 4628591 N=1  C=3484437.195260249301
p=5  eps=0.400
f=[11,1]      P=x^5+x^4-4*x^3-3*x^2+3*x+1 
                        H=1        h=1        N=1  C=0.146854024200198000
f=[191,1]     P=x^5+x^4-76*x^3-359*x^2-437*x-155 
                        H=11       h=11       N=1  C=0.164651613663662086
f=[3931,1]    P=x^5+x^4-1572*x^3+36637*x^2-309637*x+899921 
                        H=256      h=256      N=1  C=0.340908264873088935
f=[5051,1]    P=x^5+x^4-2020*x^3-55965*x^2-558681*x-1920311 
                        H=1451     h=1451     N=1  C=1.581122495380098063
f=[80251,1]   P=x^5+x^4-32100*x^3-3617715*x^2-151594781*x-2241837391 
                        H=37631    h=37631    N=1  C=4.487303013765233752
f=[154291,1]  P=x^5+x^4-61716*x^3+9584557*x^2-555677185*x+11397220121 
                        H=108691   h=108691   N=1  C=7.682827980583928749
f=[349211,1]  P=x^5+x^4-139684*x^3-32923613*x^2-2897797577*x-90346698359 
                        H=186091   h= 186091  N=1  C=144166.2531422279504
f=[555671,1]  P=x^5+x^4-222268*x^3+65858127*x^2-7299765465*x+286894118701 
                        H=721151   h= 721151  N=1  C=553515.5874018525546
f=[641491,1]  P=x^5+x^4-256596*x^3-82033869*x^2-9804476597*x-415497361495 
                        H=1566401  h= 1566401 N=1  C=1198834.2454982126520
f=[1464901,1] P=x^5+x^4-585960*x^3+282608701*x^2-51039799381*x+3272126756879 
                        H=4628591  h= 4628591 N=1  C=3484437.195260249301
\end{verbatim}
\normalsize 

\subsubsection{$G \simeq \Z/p\Z$. Fields of conductors $f' \mid f$ and 
mean value of \footnotesize $\frac{H}{p^{N-1} \cdot (\sqrt{D}\,)^\varepsilon}$.}
${}$
\footnotesize
\begin{verbatim}
p=5  eps=0.02
f=[11,1]      H=1 h=1            C=0.908540502045356139
f=[191,1]     H=11 h=11          C=8.915615908356468213
f=[941,1]     H=16 h=16          C=12.16680080169443102
f=[3931,1]    H=256 h=256        C=183.8483953118806834
f=[5051,1]    H=1451 h=1451      C=1031.649846030609568
f=[80251,1]   H=37631 h=37631    C=23953.43167867894208
f=[154291,1]  H=108691 h=108691  C=67400.00995796213464
(...)
p=5  eps=0.40
f=[11,1]      H=1 h=1            C=0.146854024200198000
f=[191,1]     H=11 h=11          C=0.164651613663662086
f=[3931,1]    H=256 h=256        C=0.340908264873088935
f=[5051,1]    H=1451 h=1451      C=1.581122495380098063
f=[80251,1]   H=37631 h=37631    C=4.487303013765233752
f=[154291,1]  H=108691 h=108691  C=7.682827980583928749
(...)
\end{verbatim}
\normalsize

\subsubsection{$G \simeq \Z/p\Z$. Fields of conductors $f' \mid f$ and 
mean value of \footnotesize $\frac{H}{(\sqrt{D}\,)^\varepsilon}$.}
${}$
\footnotesize
\begin{verbatim}
p=5  eps=0.020
f=[11,1]      H=1       C=0.908540502045356139
f=[191,1]     H=11      C=8.915615908356468213
f=[941,1]     H=16      C=12.16680080169443102
f=[3931,1]    H=256     C=183.8483953118806834
f=[5051,1]    H=1451    C=1031.649846030609568
f=[80251,1]   H=37631   C=23953.43167867894208
f=[154291,1]  H=108691  C=67400.00995796213464
(...)
p=5  eps=0.200
f=[11,1]      H=1       C=0.383215375735627288
f=[191,1]     H=11      C=1.345796325712134995
f=[3931,1]    H=256     C=9.341976011931885039
f=[5051,1]    H=1451    C=47.89789912716968310
f=[80251,1]   H=37631   C=410.9278521967080188
f=[154291,1]  H=108691  C=913.8130312255608229
(...)
\end{verbatim}
\normalsize

\section{Other real or imaginary abelian fields}
We will try to illustrate the Conjecture \ref{genconj} about the 
nature of $G$ and especially about the influence of the 
number of generators of $G$. As we have explained
in the introduction, we have chosen to classify by
increasing conductors instead of discriminants.
The programs eliminate the integers $f$ which are not conductors 
for trivial reasons, then test that the current polynomial $P$ of degree 
$d=\order G$ (given by ${\sf polsubcyclo(f,d)}$) 
defines a field $K$ of exact conductor $f$; for this, we consider
the base field $\Q$ defined via ${\sf Q=bnfinit(y-1,1)}$ and the programs
compute the conductor of the (abelian) extension defined by $P$
using ${\sf rnfconductor(Q,P)}$.

\smallskip
For this, we will consider for $p=2,3$, $G \simeq \Z/4\Z$,
$(\Z/2\Z)^2$, $(\Z/2\Z)^3$, $(\Z/3\Z)^2$, then $G \simeq \Z/6\Z$. 
In an heuristic point of view, we must 
observe that the successive maxima, for the non-genus parts, are obtained for 
almost all conductors $f$ fulfilling the inequalities $r(G) \leq \omega(f) \leq R(G)$. 
The number of fields $K$ fulfilling the conditions (Galois group, conductor, 
signature, etc.) is denoted ${\sf nK}$ in the programs,
for the computation of the means.

\subsection{Imaginary quartic cyclic fields of conductor $f$}
The program uses ${\sf G1=polgalois(polcyclo(5))}$ to have the group
$\Z/4\Z$ and test the Galois groups of order $4$ given by PARI/GP
via ${\sf polgalois(P)}$.

\subsubsection{$G \simeq \Z/4\Z$. Successive maxima of the mean of 
$\frac{H}{\raise1.8pt \hbox{\footnotesize $g_{K/\Q}^{}$} 
\cdot (\sqrt{D}\,)^\varepsilon}$.}
${}$
\footnotesize
\smallskip
\begin{verbatim}
{eps=0.01;Cm=0;G1=polgalois(polcyclo(5));Q=bnfinit(y-1,1);
\\ successive conductors:
for(f=1,10^8,phi=eulerphi(f);if(Mod(phi,4)!=0,next);
f2=valuation(f,2);ff=f/(2^f2);if(f2==1||f2>4||core(ff)!=ff,next);
\\ list of degree 4 polynomials in the fth cyclotomic field:
V=Mat(polsubcyclo(f,4));dV=matsize(V)[2];nK=0;HH=1;hh=1;CC=1.0;
for(k=1,dV,P=component(V,k)[1];D=nfdisc(P);
\\ test of ramification and Galois group:
if(omega(D)==omega(f) & polgalois(P)==G1,fK=rnfconductor(Q,P);
if(component(fK[1][1],1)[1]==f,K=bnfinit(P,1); \\ test fK=f
if(K.sign[2]!=0, \\ test K imaginary
\\ computation of the genus number:
Df=factor(f);df=matsize(Df)[1];Div=component(Df,1);gK=1;for(j=1,df,
q=Div[j];F=idealfactor(K,q);eq=component(F,1)[1][3];gK=gK*eq);gK=gK/4;
nK=nK+1;H=K.no;h=H/gK;C=h/D^(eps/2);HH=HH*H;hh=hh*h;CC=CC*C))));
\\ mean of CC and extrema test:
if(nK!=0,C=CC^(1/nK);if(C>Cm,Cm=C;H=HH^(1/nK);h=hh^(1/nK); \\ mean of H,h
print("f=",f,"=",Df," H=",precision(H,1),
" C=",precision(C,1)," omega(f)=",omega(f)," nK=",nK))))}

f=5=[5,1]                H=1        C=0.976147508055066522  omega(f)=1  nK=1
f=51=[3,1; 17,1]         H=10       C=4.739604006983370970  omega(f)=2  nK=1
f=109=[109,1]            H=17       C=15.84482780824856713  omega(f)=1  nK=1
f=123=[3,1; 41,1]        H=34       C=15.90325237458455595  omega(f)=2  nK=1
f=181=[181,1]            H=25       C=23.12463199169500313  omega(f)=1  nK=1
f=229=[229,1]            H=51       C=47.00809414205433725  omega(f)=1  nK=1
f=511=[7,1; 73,1]        H=226      C=103.9148512284847603  omega(f)=2  nK=1
f=723=[3,1; 241,1]       H=244      C=111.1369214637977656  omega(f)=2  nK=1
f=733=[733,1]            H=135      C=122.2805008461487052  omega(f)=1  nK=1
f=829=[829,1]            H=145      C=131.0960735251626679  omega(f)=1  nK=1
f=1093=[1093,1]          H=925      C=832.8416392829036284  omega(f)=1  nK=1
f=2029=[2029,1]          H=1183     C=1055.298969187637205  omega(f)=1  nK=1
f=3027=[3,1; 1009,1]     H=9100     C=4056.783364201576760  omega(f)=2  nK=1
f=4516=[2,2; 1129,1]     H=10404    C=4616.994304602346489  omega(f)=2  nK=1
f=7573=[7573,1]          H=7569     C=6619.870380619097221  omega(f)=1  nK=1
f=10376=[2,3; 1297,1]    H=21692    C=9539.917095634037343  omega(f)=2  nK=1
f=10957=[10957,1]        H=14119    C=12280.28831693695852  omega(f)=1  nK=1
f=13271=[23,1; 577,1]    H=52780    C=23249.01055510957600  omega(f)=2  nK=1
f=14267=[11,1; 1297,1]   H=107470   C=47113.91854137649729  omega(f)=2  nK=1
f=31333=[31333,1]        H=69331    C=59359.11736777416773  omega(f)=1  nK=1
f=40967=[71,1; 577,1]    H=201600   C=87807.23696987217193  omega(f)=2  nK=1
f=43387=[43,1; 1009,1]   H=205254   C=89098.12848060980410  omega(f)=2  nK=1
f=48851=[11,1; 4441,1]   H=208090   C=89556.08009624690970  omega(f)=2  nK=1
f=54987=[3,1; 18329,1]   H=258990   C=110543.8804978991659  omega(f)=2  nK=1
f=59053=[59053,1]        H=257125   C=218059.8330367149523  omega(f)=1  nK=1
f=61327=[7,1; 8761,1]    H=1721952  C=736888.4677682894188  omega(f)=2  nK=1
f=150031=[7,1; 21433,1]  H=3365766  C=1421139.460131018697  omega(f)=2  nK=1
f=197341=[197341,1]      H=1927305  C=1605174.264796961130  omega(f)=1  nK=1
f=245029=[245029,1]      H=2147103  C=1782438.755437137458  omega(f)=1  nK=1
f=267707=[11,1; 24337,1] H=4458500  C=1870470.412521304054  omega(f)=2  nK=1
f=319611=[3,1; 106537,1] H=4964300  C=2063689.257767956675  omega(f)=2  nK=1
f=356411=[11,1; 32401,1] H=12376170 C=5169921.722394828539  omega(f)=2  nK=1
(...)
\end{verbatim}
\normalsize

\smallskip
To get the table for real fields, one has to replace the test ${\sf if(K.sign[2]!=0}$
by ${\sf if(K.sign[2]==0}$, which gives the short excerpt:

\footnotesize
\smallskip
\begin{verbatim}
f=15=[3,1; 5,1]         H=1     C=0.482741062227355455  omega(f)=2  nK=1
f=16=[2,4]              H=1     C=0.962594443101751380  omega(f)=1  nK=1
f=212=[2,2; 53,1]       H=5     C=2.323032419483186330  omega(f)=2  nK=1
f=257=[257,1]           H=3     C=2.760401519960496457  omega(f)=1  nK=1
f=401=[401,1]           H=5     C=4.570069858094963666  omega(f)=1  nK=1
f=577=[577,1]           H=7     C=6.363270756545588592  omega(f)=1  nK=1
f=916=[2,2; 229,1]      H=15    C=6.817782326293629939  omega(f)=2  nK=1
f=1129=[1129,1]         H=9     C=8.099386061052115192  omega(f)=1  nK=1
f=1297=[1297,1]         H=11    C=9.878672447143713655  omega(f)=1  nK=1
f=1697=[1697,1]         H=17    C=15.20560462285605392  omega(f)=1  nK=1
f=2005=[5,1; 401,1]     H=80    C=35.97684919154074598  omega(f)=2  nK=1
f=4616=[2,3; 577,1]     H=112   C=49.85853230401655734  omega(f)=2  nK=1
f=7753=[7753,1]         H=75    C=65.57211873846320180  omega(f)=1  nK=1
f=8116=[2,2; 2029,1]    H=175   C=76.98005594903680110  omega(f)=2  nK=1
f=10376=[2,3; 1297,1]   H=286   C=125.7798400032885248  omega(f)=2  nK=1
f=24337=[24337,1]       H=185   C=158.9928970375704956  omega(f)=1  nK=1
f=25096=[2,3; 3137,1]   H=1152  C=499.9699077725989672  omega(f)=2  nK=1
f=54772=[2,2; 13693,1]  H=1755  C=750.2034311196041105  omega(f)=2  nK=1
f=56456=[2,3; 7057,1]   H=4074  C=1746.750303432222082  omega(f)=2  nK=1
f=115208=[2,3; 14401,1] H=4472  C=1896.989860579581230  omega(f)=2  nK=1
(...)
\end{verbatim}
\normalsize

\subsubsection{$G \simeq \Z/4\Z$. Successive maxima of the mean 
of $\frac{H}{(\sqrt{D}\,)^\varepsilon}$.}
Considering the same program for the successive maxima of the whole
class numbers $H$, the number $\omega(f)$ of primes dividing 
$f$ tends to be larger in a statistical point of view (for instance 
some $\omega(f) = 3$):   

\footnotesize
\smallskip
\begin{verbatim}
{eps=0.1;Cm=0;G1=polgalois(polcyclo(5));Q=bnfinit(y-1,1);
\\ successive conductors:
for(f=1,10^6,phi=eulerphi(f);if(Mod(phi,4)!=0,next);
f2=valuation(f,2);ff=f/(2^f2);if(f2==1||f2>4||core(ff)!=ff,next);
\\ list of degree 4 polynomials in the fth cyclotomic field:
V=Mat(polsubcyclo(f,4));dV=matsize(V)[2];nK=0;HH=1;CC=1.0;
for(k=1,dV,P=component(V,k)[1];D=nfdisc(P);
\\ test of ramification and Galois group:
if(omega(D)==omega(f) & polgalois(P)==G1,fK=rnfconductor(Q,P);
if(component(fK[1][1],1)[1]==f,K=bnfinit(P,1); \\ test fK=f
if(K.sign[2]!=0, \\ test K imaginary
nK=nK+1;H=K.no;C=H/abs(D)^(eps/2);HH=HH*H;CC=CC*C))));
\\ mean of CC and extrema test:
if(nK!=0,C=CC^(1/nK);if(C>Cm,Cm=C;H=HH^(1/nK); \\ mean of H
print("f=",f,"=",factor(f)," H=",precision(H,1),
" C=",precision(C,1)," omega(f)=",omega(f)," nK=",nK))))}

f=5=[5,1]                 H=1        C=0.785515030231764334  omega(f)=1  nK=1
f=40=[2,3; 5,1]           H=2        C=1.276072931359182746  omega(f)=2  nK=1
f=51=[3,1; 17,1]          H=10       C=5.857611124383549536  omega(f)=2  nK=1
f=109=[109,1]             H=17       C=8.410754307693637111  omega(f)=1  nK=1
f=123=[3,1; 41,1]         H=34       C=17.45216109112128972  omega(f)=2  nK=1
f=187=[11,1; 17,1]        H=34       C=17.48929528433495368  omega(f)=2  nK=1
f=229=[229,1]             H=51       C=22.57328962224578600  omega(f)=1  nK=1
f=291=[3,1; 97,1]         H=68       C=30.67474966040826406  omega(f)=2  nK=1
f=411=[3,1; 137,1]        H=74       C=31.69651460480555020  omega(f)=2  nK=1
f=435=[3,1; 5,1; 29,1]    H=115.37764081484765738 
                                     C=49.00093067514680456  omega(f)=3  nK=2
f=511=[7,1; 73,1]         H=226      C=97.74627279797258465  omega(f)=2  nK=1
f=1015=[5,1; 7,1; 29,1]   H=326.3372488699382291 
                                     C=127.3361291154895528  omega(f)=3  nK=2
f=1093=[1093,1]           H=925      C=323.8535846044137371  omega(f)=1  nK=1
f=1604=[2,2; 401,1]       H=1060     C=375.5157960003103599  omega(f)=2  nK=1
f=1799=[7,1; 257,1]       H=1398     C=500.6190813114782884  omega(f)=2  nK=1
f=3027=[3,1; 1009,1]      H=9100     C=2888.987028762488449  omega(f)=2  nK=1
f=4516=[2,2; 1129,1]      H=10404    C=3155.660350206402243  omega(f)=2  nK=1
f=8655=[3,1; 5,1; 577,1]  H=15064    C=4427.461660644001969  omega(f)=3  nK=1
f=10376=[2,3; 1297,1]     H=21692    C=6012.424012890724163  omega(f)=2  nK=1
f=12419=[11,1; 1129,1]    H=24858    C=6814.324702892086861  omega(f)=2  nK=1
f=13271=[23,1; 577,1]     H=52780    C=14863.47066425293806  omega(f)=2  nK=1
f=14267=[11,1; 1297,1]    H=107470   C=28854.06354818241425  omega(f)=2  nK=1
f=40967=[71,1; 577,1]     H=201600   C=50721.05997730998500  omega(f)=2  nK=1
f=54987=[3,1; 18329,1]    H=258990   C=53223.08625892298533  omega(f)=2  nK=1
f=61327=[7,1; 8761,1]     H=1721952  C=363185.2356760567566  omega(f)=2  nK=1
f=150031=[7,1; 21433,1]   H=3365766  C=620742.7707399591228  omega(f)=2  nK=1
f=267707=[11,1; 24337,1]  H=4458500  C=771097.7682083982741  omega(f)=2  nK=1
f=319611=[3,1; 106537,1]  H=4964300  C=783467.2078602062662  omega(f)=2  nK=1
f=331223=[23,1; 14401,1]  H=4510700  C=783993.1400401131476  omega(f)=2  nK=1
f=356411=[11,1; 32401,1]  H=12376170 C=2050516.345595339414  omega(f)=2  nK=1
(...)
\end{verbatim}
\normalsize

\subsection{Imaginary biquadratic fields of conductor $f$}
For non-cyclic fields, except for $K=\Q(\sqrt 2, \sqrt{-1})$,
the conductor is necessarily such that $\omega(f) \geq 2$, and 
the ``exceptional'' cases should give $\omega(f) =3, 4, \ldots$

\smallskip
We put ${\sf G2=polgalois(polcyclo(12))}$ to have the group $(\Z/2\Z)^2$:

\subsubsection{$G \simeq \Z/2\Z \times \Z/2\Z$. Successive maxima of the mean of $\frac{H}
{\raise1.6pt \hbox{\footnotesize $g_{K/\Q}^{}$} \cdot (\sqrt{D}\,)^\varepsilon}$.}
${}$
\footnotesize
\begin{verbatim}
{eps=0.01;Cm=0;G2=polgalois(polcyclo(12));Q=bnfinit(y-1,1);
\\ successive conductors:
for(f=1,10^6,phi=eulerphi(f);if(Mod(phi,4)!=0,next);
f2=valuation(f,2);ff=f/(2^f2);if(f2==1||f2>4||core(ff)!=ff,next);
\\ list of degree 4 polynomials in the fth cyclotomic field:
V=Mat(polsubcyclo(f,4));dV=matsize(V)[2];nK=0;HH=1;hh=1;CC=1.0;
for(k=1,dV,P=component(V,k)[1];D=nfdisc(P);
\\ test of ramification and Galois group:
if(omega(D)==omega(f) & polgalois(P)==G2,fK=rnfconductor(Q,P);
if(component(fK[1][1],1)[1]==f,K=bnfinit(P,1); \\ test fK=f
if(K.sign[2]!=0, \\ test K imaginary
\\ computation of the genus number:
Df=factor(f);df=matsize(Df)[1];Div=component(Df,1);gK=1;for(j=1,df,
q=Div[j];F=idealfactor(K,q);eq=component(F,1)[1][3];gK=gK*eq);gK=gK/4;
nK=nK+1;H=K.no;h=H/gK;C=h/abs(D)^(eps/2);HH=HH*H;hh=hh*h;CC=CC*C))));
\\ mean of CC and extrema test:
if(nK!=0,C=CC^(1/nK);if(C>Cm,Cm=C;H=HH^(1/nK);h=hh^(1/nK); \\ mean of H,h
print("f=",f,"=",Df," H=",precision(H,1),
" C=",precision(C,1)," omega(f)=",omega(f)," nK=",nK))))}

f=8=[2,3]                H=1     C=0.972654947412285518  omega(f)=1  nK=1
f=12=[2,2; 3,1]          H=1     C=0.975457130078795936  omega(f)=2  nK=1
f=39=[3,1; 13,1]         H=2     C=1.928054694106823799  omega(f)=2  nK=1
f=69=[3,1; 23,1]         H=3     C=2.875628398112372534  omega(f)=2  nK=1
f=95=[5,1; 19,1]         H=4     C=3.821930235323273249  omega(f)=2  nK=1
f=119=[7,1; 17,1]        H=5     C=4.766663944449055461  omega(f)=2  nK=1
f=155=[5,1; 31,1]        H=6     C=5.704898650119474394  omega(f)=2  nK=1
f=213=[3,1; 71,1]        H=7     C=6.634592355562253982  omega(f)=2  nK=1
f=295=[5,1; 59,1]        H=12    C=11.33660528868662035  omega(f)=2  nK=1
f=316=[2,2; 79,1]        H=15    C=14.16101517671179429  omega(f)=2  nK=1
f=391=[17,1; 23,1]       H=21    C=19.78324489680824008  omega(f)=2  nK=1
f=527=[17,1; 31,1]       H=27    C=25.35979029775623778  omega(f)=2  nK=1
f=655=[5,1; 131,1]       H=30    C=28.11634357939838355  omega(f)=2  nK=1
f=695=[5,1; 139,1]       H=36    C=33.71961852330659701  omega(f)=2  nK=1
f=755=[5,1; 151,1]       H=42    C=39.30699295336691860  omega(f)=2  nK=1
f=1055=[5,1; 211,1]      H=54    C=50.36875720346152355  omega(f)=2  nK=1
f=1195=[5,1; 239,1]      H=60    C=55.89559343232499964  omega(f)=2  nK=1
f=1207=[17,1; 71,1]      H=63    C=58.68450919759415447  omega(f)=2  nK=1
f=1343=[17,1; 79,1]      H=85    C=79.09302129521269893  omega(f)=2  nK=1
f=1751=[17,1; 103,1]     H=120   C=111.3649136315469349  omega(f)=2  nK=1
f=2155=[5,1; 431,1]      H=126   C=116.6906535951326993  omega(f)=2  nK=1
f=2159=[17,1; 127,1]     H=150   C=138.9148686560359562  omega(f)=2  nK=1
f=2567=[17,1; 151,1]     H=154   C=142.3726150732244520  omega(f)=2  nK=1
f=3239=[41,1; 79,1]      H=175   C=161.4113014454594570  omega(f)=2  nK=1
f=3247=[17,1; 191,1]     H=208   C=191.8441285757209858  omega(f)=2  nK=1
f=3791=[17,1; 223,1]     H=238   C=219.1741945424862956  omega(f)=2  nK=1
f=4471=[17,1; 263,1]     H=286   C=262.9432257122401507  omega(f)=2  nK=1
f=5095=[5,1; 1019,1]     H=312   C=286.4726411861354221  omega(f)=2  nK=1
f=5287=[17,1; 311,1]     H=323   C=296.4629523337229678  omega(f)=2  nK=1
f=6071=[13,1; 467,1]     H=336   C=307.9687675618586813  omega(f)=2  nK=1
f=6155=[5,1; 1231,1]     H=378   C=346.4172575739003240  omega(f)=2  nK=1
f=6239=[17,1; 367,1]     H=405   C=371.1110393200093025  omega(f)=2  nK=1
f=6995=[5,1; 1399,1]     H=432   C=395.3992758795738976  omega(f)=2  nK=1
f=7295=[5,1; 1459,1]     H=440   C=402.5524030522415731  omega(f)=2  nK=1
f=7463=[17,1; 439,1]     H=495   C=452.7683540573952683  omega(f)=2  nK=1
f=7895=[5,1; 1579,1]     H=504   C=460.7411639385067668  omega(f)=2  nK=1
f=8551=[17,1; 503,1]     H=525   C=479.5557849789485247  omega(f)=2  nK=1
f=9655=[5,1; 1931,1]     H=546   C=498.1327779099005440  omega(f)=2  nK=1
f=9895=[5,1; 1979,1]     H=644   C=587.3969800430742345  omega(f)=2  nK=1
f=10183=[17,1; 599,1]    H=800   C=729.4763720843714494  omega(f)=2  nK=1
f=12359=[17,1; 727,1]    H=884   C=804.5118246192972826  omega(f)=2  nK=1
f=12431=[31,1; 401,1]    H=1155  C=1051.082786075149247  omega(f)=2  nK=1
f=13271=[23,1; 577,1]    H=1386  C=1260.474876339380179  omega(f)=2  nK=1
f=17663=[17,1; 1039,1]   H=1426  C=1293.149930556092194  omega(f)=2  nK=1
f=22559=[17,1; 1327,1]   H=1560  C=1411.209202757567127  omega(f)=2  nK=1
f=24463=[17,1; 1439,1]   H=1833  C=1656.827780210867060  omega(f)=2  nK=1
f=24599=[17,1; 1447,1]   H=1978  C=1787.792507027724221  omega(f)=2  nK=1
f=27119=[47,1; 577,1]    H=3045  C=2749.505294205692715  omega(f)=2  nK=1
f=29831=[23,1; 1297,1]   H=3432  C=3095.997498110343389  omega(f)=2  nK=1
f=40967=[71,1; 577,1]    H=4116  C=3701.272193141451084  omega(f)=2  nK=1
f=47423=[47,1; 1009,1]   H=4235  C=3802.712781426917957  omega(f)=2  nK=1
f=53063=[47,1; 1129,1]   H=5760  C=5166.239854556097149  omega(f)=2  nK=1
f=60959=[47,1; 1297,1]   H=9130  C=8177.497107751319985  omega(f)=2  nK=1
f=80159=[71,1; 1129,1]   H=10710 C=9566.431099491350923  omega(f)=2  nK=1
f=96359=[167,1; 577,1]   H=14784 C=13181.14176742287956  omega(f)=2  nK=1
f=137903=[239,1; 577,1]  H=22050 C=19589.02663140455669  omega(f)=2  nK=1
f=151751=[263,1; 577,1]  H=26754 C=23745.28612625392000  omega(f)=2  nK=1
f=207143=[359,1; 577,1]  H=34713 C=30713.50960296288906  omega(f)=2  nK=1
f=216599=[167,1; 1297,1] H=37268 C=32959.41709369729731  omega(f)=2  nK=1
f=248687=[431,1; 577,1]  H=43071 C=38038.94157342156453  omega(f)=2  nK=1
f=313799=[311,1; 1009,1] H=58121 C=51211.38983446847804  omega(f)=2  nK=1
f=351119=[311,1; 1129,1] H=83106 C=73143.85340462042777  omega(f)=2  nK=1
(...)
\end{verbatim}
\normalsize

\subsubsection{$G \simeq \Z/2\Z \times \Z/2\Z$. Successive maxima of the 
mean of $\frac{H}{(\sqrt{D}\,)^\varepsilon}$.}
Same computations for the successive maxima of the whole class number; we
observe simlar results as for the cyclic case with few $\omega(D)=3$:

\footnotesize
\begin{verbatim}
{eps=0.01;Cm=0;G2=polgalois(polcyclo(12));Q=bnfinit(y-1,1);
\\ successive conductors:
for(f=1,10^6,phi=eulerphi(f);if(Mod(phi,4)!=0,next);
f2=valuation(f,2);ff=f/(2^f2);if(f2==1||f2>3||core(ff)!=ff,next);
\\ list of degree 4 polynomials in the fth cyclotomic field:
V=Mat(polsubcyclo(f,4));dV=matsize(V)[2];nK=0;HH=1;hh=1;CC=1.0;
for(k=1,dV,P=component(V,k)[1];D=nfdisc(P); 
\\ test of ramification and Galois group:
if(omega(D)==omega(f) & polgalois(P)==G2,fK=rnfconductor(Q,P);
if(component(fK[1][1],1)[1]==f,K=bnfinit(P,1); \\ test fK=f
if(K.sign[2]!=0, \\ test K imaginary
nK=nK+1;H=K.no;C=H/D^(eps/2);HH=HH*H;CC=CC*C))));
\\ mean of CC and extrema test:
if(nK!=0,C=CC^(1/nK);if(C>Cm,Cm=C;H=HH^(1/nK); \\ mean of H
print("f=",f,"=",factor(f)," H=",precision(H,1),
" C=",precision(C,1)," omega(f)=",omega(f)," nK=",nK))))}

f=8=[2,3]               H=1      C=0.972654947412285518  omega(f)=1  nK=1
f=12=[2,2; 3,1]         H=1      C=0.975457130078795936  omega(f)=2  nK=1
f=24=[2,3; 3,1]         H=1.4142135623730950488 
                                 C=1.365236011752337206  omega(f)=2  nK=4
f=39=[3,1; 13,1]        H=2      C=1.928054694106823799  omega(f)=2  nK=1
f=56=[2,3; 7,1]         H=2.378414230005442134 
                                 C=2.276671989463074733  omega(f)=2  nK=4
f=69=[3,1; 23,1]        H=3      C=2.875628398112372534  omega(f)=2  nK=1
f=95=[5,1; 19,1]        H=4      C=3.821930235323273249  omega(f)=2  nK=1
f=119=[7,1; 17,1]       H=5      C=4.766663944449055461  omega(f)=2  nK=1
f=136=[2,3; 17,1]       H=5.656854249492380195 
                                 C=5.357742969863248032  omega(f)=2  nK=4
f=155=[5,1; 31,1]       H=6      C=5.704898650119474394  omega(f)=2  nK=1
f=213=[3,1; 71,1]       H=7      C=6.634592355562253982  omega(f)=2  nK=1
f=255=[3,1; 5,1; 17,1]  H=9.118028227819110568 
                                 C=8.626517954608540449  omega(f)=3  nK=4
f=295=[5,1; 59,1]       H=12     C=11.33660528868662035  omega(f)=2  nK=1
f=316=[2,2; 79,1]       H=15     C=14.16101517671179429  omega(f)=2  nK=1
f=391=[17,1; 23,1]      H=21     C=19.78324489680824008  omega(f)=2  nK=1
f=527=[17,1; 31,1]      H=27     C=25.35979029775623778  omega(f)=2  nK=1
f=655=[5,1; 131,1]      H=30     C=28.11634357939838355  omega(f)=2  nK=1
f=695=[5,1; 139,1]      H=36     C=33.71961852330659701  omega(f)=2  nK=1
f=755=[5,1; 151,1]      H=42     C=39.30699295336691860  omega(f)=2  nK=1
f=1055=[5,1; 211,1]     H=54     C=50.36875720346152355  omega(f)=2  nK=1
f=1195=[5,1; 239,1]     H=60     C=55.89559343232499964  omega(f)=2  nK=1
f=1207=[17,1; 71,1]     H=63     C=58.68450919759415447  omega(f)=2  nK=1
f=1343=[17,1; 79,1]     H=85     C=79.09302129521269893  omega(f)=2  nK=1
f=1751=[17,1; 103,1]    H=120    C=111.3649136315469349  omega(f)=2  nK=1
f=2155=[5,1; 431,1]     H=126    C=116.6906535951326993  omega(f)=2  nK=1
f=2159=[17,1; 127,1]    H=150    C=138.9148686560359562  omega(f)=2  nK=1
f=2567=[17,1; 151,1]    H=154    C=142.3726150732244520  omega(f)=2  nK=1
f=3239=[41,1; 79,1]     H=175    C=161.4113014454594570  omega(f)=2  nK=1
f=3247=[17,1; 191,1]    H=208    C=191.8441285757209858  omega(f)=2  nK=1
f=3791=[17,1; 223,1]    H=238    C=219.1741945424862956  omega(f)=2  nK=1
f=4471=[17,1; 263,1]    H=286    C=262.9432257122401507  omega(f)=2  nK=1
f=5095=[5,1; 1019,1]    H=312    C=286.4726411861354221  omega(f)=2  nK=1
f=5287=[17,1; 311,1]    H=323    C=296.4629523337229678  omega(f)=2  nK=1
f=6071=[13,1; 467,1]    H=336    C=307.9687675618586813  omega(f)=2  nK=1
f=6155=[5,1; 1231,1]    H=378    C=346.4172575739003240  omega(f)=2  nK=1
f=6239=[17,1; 367,1]    H=405    C=371.1110393200093025  omega(f)=2  nK=1
f=6995=[5,1; 1399,1]    H=432    C=395.3992758795738976  omega(f)=2  nK=1
f=7295=[5,1; 1459,1]    H=440    C=402.5524030522415731  omega(f)=2  nK=1
f=7463=[17,1; 439,1]    H=495    C=452.7683540573952683  omega(f)=2  nK=1
f=7895=[5,1; 1579,1]    H=504    C=460.7411639385067668  omega(f)=2  nK=1
f=8551=[17,1; 503,1]    H=525    C=479.5557849789485247  omega(f)=2  nK=1
f=9655=[5,1; 1931,1]    H=546    C=498.1327779099005440  omega(f)=2  nK=1
f=9895=[5,1; 1979,1]    H=644    C=587.3969800430742345  omega(f)=2  nK=1
f=10183=[17,1; 599,1]   H=800    C=729.4763720843714494  omega(f)=2  nK=1
f=12359=[17,1; 727,1]   H=884    C=804.5118246192972826  omega(f)=2  nK=1
f=12431=[31,1; 401,1]   H=1155   C=1051.082786075149247  omega(f)=2  nK=1
f=13271=[23,1; 577,1]   H=1386   C=1260.474876339380179  omega(f)=2  nK=1
f=17663=[17,1; 1039,1]  H=1426   C=1293.149930556092194  omega(f)=2  nK=1
f=22559=[17,1; 1327,1]  H=1560   C=1411.209202757567127  omega(f)=2  nK=1
f=24463=[17,1; 1439,1]  H=1833   C=1656.827780210867060  omega(f)=2  nK=1
f=24599=[17,1; 1447,1]  H=1978   C=1787.792507027724221  omega(f)=2  nK=1
f=27119=[47,1; 577,1]   H=3045   C=2749.505294205692715  omega(f)=2  nK=1
f=29831=[23,1; 1297,1]  H=3432   C=3095.997498110343389  omega(f)=2  nK=1
f=40967=[71,1; 577,1]   H=4116   C=3701.272193141451084  omega(f)=2  nK=1
f=47423=[47,1; 1009,1]  H=4235   C=3802.712781426917957  omega(f)=2  nK=1
f=53063=[47,1; 1129,1]  H=5760   C=5166.239854556097149  omega(f)=2  nK=1
f=60959=[47,1; 1297,1]  H=9130   C=8177.497107751319985  omega(f)=2  nK=1
f=80159=[71,1; 1129,1]  H=10710  C=9566.431099491350923  omega(f)=2  nK=1
f=96359=[167,1; 577,1]  H=14784  C=13181.14176742287956  omega(f)=2  nK=1
f=137903=[239,1; 577,1] H=22050  C=19589.02663140455669  omega(f)=2  nK=1
f=151751=[263,1; 577,1] H=26754  C=23745.28612625392000  omega(f)=2  nK=1
f=207143=[359,1; 577,1] H=34713  C=30713.50960296288906  omega(f)=2  nK=1
f=216599=[167,1; 1297,1]H=37268  C=32959.41709369729731  omega(f)=2  nK=1
f=248687=[431,1; 577,1] H=43071  C=38038.94157342156453  omega(f)=2  nK=1
f=313799=[311,1; 1009,1]H=58121  C=51211.38983446847804  omega(f)=2  nK=1
f=351119=[311,1; 1129,1]H=83106  C=73143.85340462042777  omega(f)=2  nK=1
(...)
\end{verbatim}
\normalsize

\subsection{Experiment for multi-quadratic fields of rank $3$}
We give the case of imaginary multi-quadratic fields.
For the $2$-rank $3$, we obtain the following results 
with many $\omega(f)=3=r(G)=R(G)$, which are,
as for the case of the $3$-rank $2$, very in
accordence with the Conjecture \ref{mainconj} (all the cases in the
selected interval are such that ${\sf h=H}$ since ${\sf nK=1}$);
note that the cases $\omega(f)=2$ correspond to conductors
$f = 2^3 \cdot \ell$ for which the wild ramification of $2$
``counts for two generator'' of ${\rm Gal}(\Q(\mu_f)/\Q)$:

\footnotesize
\smallskip
\begin{verbatim}
{eps=0.01;Cm=0;G3=polgalois(polcyclo(24));Q=bnfinit(y-1,1);
\\ successive conductors:
for(f=1,10^6,phi=eulerphi(f);if(Mod(phi,8)!=0,next);
f2=valuation(f,2);ff=f/(2^f2);if(f2==1||f2>4||core(ff)!=ff,next);
\\ list of degree 8 polynomials in the fth cyclotomic field:
V=Mat(polsubcyclo(f,8));dV=matsize(V)[2];nK=0;HH=1;hh=1;CC=1.0;
for(k=1,dV,P=component(V,k)[1];D=nfdisc(P);
\\ test of ramification and Galois group:
if(omega(D)==omega(f) & polgalois(P)==G3,fK=rnfconductor(Q,P);
if(component(fK[1][1],1)[1]==f,K=bnfinit(P,1); \\ test fK=f
if(K.sign[2]!=0, \\ test K imaginary
\\ computation of the genus number:
Df=factor(f);df=matsize(Df)[1];Div=component(Df,1);gK=1;for(j=1,df,
q=Div[j];F=idealfactor(K,q);eq=component(F,1)[1][3];gK=gK*eq);gK=gK/8;
nK=nK+1;H=K.no;h=H/gK;C=h/abs(D)^(eps/2);HH=HH*H;hh=hh*h;CC=CC*C))));
\\ mean of CC and extrema test:
if(nK!=0,C=CC^(1/nK);if(C>Cm,Cm=C;H=HH^(1/nK);h=hh^(1/nK); \\ mean of H,h
print("f=",f,"=",Df," H=",H,
" C=",precision(C,1)," omega(f)=",omega(f)," nK=",nK))))}

f=24=[2,3; 3,1]            H=1      C=0.925497340811340481  omega(f)=2  nK=1
f=56=[2,3; 7,1]            H=2      C=1.819892080922239686  omega(f)=2  nK=1
f=104=[2,3; 13,1]          H=3      C=2.696248944021590053  omega(f)=2  nK=1
f=136=[2,3; 17,1]          H=4      C=3.575762069753877514  omega(f)=2  nK=1
f=184=[2,3; 23,1]          H=6      C=5.331314392329002593  omega(f)=2  nK=1
f=248=[2,3; 31,1]          H=12     C=10.59916401240363877  omega(f)=2  nK=1
f=328=[2,3; 41,1]          H=16     C=14.05341612030092425  omega(f)=2  nK=1
f=376=[2,3; 47,1]          H=20     C=17.51885180427532467  omega(f)=2  nK=1
f=568=[2,3; 71,1]          H=42     C=36.48729968846260794  omega(f)=2  nK=1
f=632=[2,3; 79,1]          H=60     C=52.01352759088901895  omega(f)=2  nK=1
f=904=[2,3; 113,1]         H=64     C=55.08533633618848662  omega(f)=2  nK=1
f=1016=[2,3; 127,1]        H=120    C=103.0440149547490389  omega(f)=2  nK=1
f=1495=[5,1; 13,1; 23,1]   H=120    C=103.6784324181077920  omega(f)=3  nK=1
f=1595=[5,1; 11,1; 29,1]   H=160    C=138.0590144834850848  omega(f)=3  nK=1
f=1784=[2,3; 223,1]        H=336    C=285.2927796818069852  omega(f)=2  nK=1
f=2056=[2,3; 257,1]        H=384    C=325.1248480613886757  omega(f)=2  nK=1
f=2755=[5,1; 19,1; 29,1]   H=416    C=355.0511294873608592  omega(f)=3  nK=1
f=3055=[5,1; 13,1; 47,1]   H=450    C=383.2765860472549375  omega(f)=3  nK=1
f=3512=[2,3; 439,1]        H=750    C=628.2457702585119454  omega(f)=2  nK=1
f=3835=[5,1; 13,1; 59,1]   H=792    C=671.5059640041688235  omega(f)=3  nK=1
f=4355=[5,1; 13,1; 67,1]   H=990    C=837.2505305391947183  omega(f)=3  nK=1
f=4495=[5,1; 29,1; 31,1]   H=1344   C=1135.911966664581723  omega(f)=3  nK=1
f=4616=[2,3; 577,1]        H=1568   C=1306.291624871039468  omega(f)=2  nK=1
f=4879=[7,1; 17,1; 41,1]   H=2730   C=2303.541446342435365  omega(f)=3  nK=1
f=6088=[2,3; 761,1]        H=3600   C=2982.582105985213100  omega(f)=2  nK=1
f=8555=[5,1; 29,1; 59,1]   H=5376   C=4485.541300328763516  omega(f)=3  nK=1
f=10295=[5,1; 29,1; 71,1]  H=7168   C=5958.616976839243116  omega(f)=3  nK=1
f=11455=[5,1; 29,1; 79,1]  H=7200   C=5972.450989307247585  omega(f)=3  nK=1
f=11635=[5,1; 13,1; 179,1] H=7680   C=6368.628152049485991  omega(f)=3  nK=1
f=12536=[2,3; 1567,1]      H=9900   C=8084.467416708329516  omega(f)=2  nK=1
f=12935=[5,1; 13,1; 199,1] H=10368  C=8579.454193948520892  omega(f)=3  nK=1
f=14495=[5,1; 13,1; 223,1] H=12250  C=10113.73859193621600  omega(f)=3  nK=1
f=15215=[5,1; 17,1; 179,1] H=12480  C=10293.64404441184357  omega(f)=3  nK=1
f=15535=[5,1; 13,1; 239,1] H=14040  C=11575.52994096647780  omega(f)=3  nK=1
f=16031=[17,1; 23,1; 41,1] H=31248  C=25746.78542880426499  omega(f)=3  nK=1
f=21607=[17,1; 31,1; 41,1] H=55080  C=45113.03465787028027  omega(f)=3  nK=1
f=30095=[5,1; 13,1; 463,1] H=59598  C=48491.07352541684482  omega(f)=3  nK=1
f=30595=[5,1; 29,1; 211,1] H=70848  C=57625.48395127060845  omega(f)=3  nK=1
f=32759=[17,1; 41,1; 47,1] H=140400 C=114040.8629779094895  omega(f)=3  nK=1
f=47995=[5,1; 29,1; 331,1] H=155232 C=125128.8153786116152  omega(f)=3  nK=1
f=49487=[17,1; 41,1; 71,1] H=341334 C=274972.7946695323715  omega(f)=3  nK=1
(...)
\end{verbatim}
\normalsize

\subsection{Experiment for multi-cubic fields of rank $2$}
The program is easily deduced from the previous ones.
For the $3$-rank $2$, we obtain the following results, still in
accordence with the Conjecture \ref{mainconj} (all cases in the
selected interval are such that ${\sf nK=1}$, which means ${\sf h=H}$):

\footnotesize
\begin{verbatim}
{eps=0.01;Cm=0;G2=polgalois(polsubcyclo(63,9));Q=bnfinit(y-1,1);
\\ successive conductors:
for(f=1,10^6,phi=eulerphi(f);if(Mod(phi,9)!=0,next);
f3=valuation(f,3);ff=f/(3^f3);if(f3==1||f3>2||core(ff)!=ff,next);
\\ list of degree 9 polynomials in the fth cyclotomic field:
V=Mat(polsubcyclo(f,9));dV=matsize(V)[2];nK=0;HH=1;hh=1;CC=1.0;
for(k=1,dV,P=component(V,k)[1];D=nfdisc(P);
\\ test of ramification and Galois group:
if(omega(D)==omega(f) & polgalois(P)==G2,fK=rnfconductor(Q,P);
if(component(fK[1][1],1)[1]==f,K=bnfinit(P,1); \\ test fK=f
\\ computation of the genus number:
Df=factor(f);df=matsize(Df)[1];Div=component(Df,1);gK=1;for(j=1,df,
q=Div[j];F=idealfactor(K,q);eq=component(F,1)[1][3];gK=gK*eq);gK=gK/9;
nK=nK+1;H=K.no;h=H/gK;C=h/abs(D)^(eps/2);HH=HH*H;hh=hh*h;CC=CC*C)));
\\ mean of CC and extrema test:
if(nK!=0,C=CC^(1/nK);if(C>Cm,Cm=C;H=HH^(1/nK);h=hh^(1/nK); \\ mean of H,h
print("f=",f,"=",Df," H=",precision(H,1),
" C=",precision(C,1)," omega(f)=",omega(f)," nK=",nK))))}

f=63=[3,2; 7,1]        H=1     C=0.883120128476567309  omega(f)=2  nK=1
f=657=[3,2; 73,1]      H=3     C=2.469416204162289171  omega(f)=2  nK=1
f=679=[7,1; 97,1]      H=4     C=3.289303128491503489  omega(f)=2  nK=1
f=1261=[13,1; 97,1]    H=7     C=5.650366096932072126  omega(f)=2  nK=1
f=1267=[7,1; 181,1]    H=12    C=9.684962591947700189  omega(f)=2  nK=1
f=1687=[7,1; 241,1]    H=13    C=10.40231258099219766  omega(f)=2  nK=1
f=1899=[3,2; 211,1]    H=16    C=12.75746054059586620  omega(f)=2  nK=1
f=2119=[13,1; 163,1]   H=28    C=22.25225881928152337  omega(f)=2  nK=1
f=2817=[3,2; 313,1]    H=112   C=88.25197174283300061  omega(f)=2  nK=1
f=5947=[19,1; 313,1]   H=196   C=151.0174501642113834  omega(f)=2  nK=1
f=7441=[7,1; 1063,1]   H=325   C=248.7335806711150378  omega(f)=2  nK=1
f=12439=[7,1; 1777,1]  H=432   C=325.5668539139437488  omega(f)=2  nK=1
f=14527=[73,1; 199,1]  H=637   C=477.8308156077793569  omega(f)=2  nK=1
f=36667=[37,1; 991,1]  H=729   C=531.8623730500720912  omega(f)=2  nK=1
f=38503=[139,1; 277,1] H=1296  C=944.1481890324295829  omega(f)=2  nK=1
f=45229=[31,1; 1459,1] H=1417  C=1027.323779063204492  omega(f)=2  nK=1
f=51457=[7,1; 7351,1]  H=2548  C=1840.162134135815576  omega(f)=2  nK=1
f=53809=[7,1; 7687,1]  H=59584 C=42973.82464267023290  omega(f)=2  nK=1
\end{verbatim}
\normalsize

\subsection{Degree $6$ imaginary cyclic fields of conductor $f$}
We consider the more complex case of an abelian imaginary 
extension of degree $6$. 
In the following program, the odd prime number $p$ may be arbitrary,
giving the results for degree $2\,p$ abelian fields. Since $2$ does not 
ramify in the degree $p$ subfield of $K$, the $2$-part of $f_K$ if that
of the conductor of the quadratic field. Then the $p$-part may be
$p$ or $p^2$. The signature of $K$ is given by the sign of the 
discriminant $D$. Recall that $r(G)=1$ and $R(G)=2$:

\subsubsection{$G \simeq \Z/6\Z$. Successive maxima of the mean of 
$\frac{H}{\raise1.6pt \hbox{\footnotesize $g_{K/\Q}^{}$} \cdot (\sqrt{D}\,)^\varepsilon}$.}
${}$
\footnotesize
\begin{verbatim}
{eps=0.1;Cm=0;Q=bnfinit(y-1,1);p=3; \\ successive conductors:
for(f=1,10^6,phi=eulerphi(f);if(Mod(phi,2*p)!=0,next);
f2=valuation(f,2);fp=valuation(f,p);ff=f/(2^f2*p^fp);
if(f2==1||f2>3||fp>2||core(ff)!=ff,next);
\\ list of degree 2*p polynomials in the fth cyclotomic field:
V=Mat(polsubcyclo(f,2*p));dV=matsize(V)[2];nK=0;HH=1;hh=1;CC=1.0;
for(k=1,dV,P=component(V,k)[1];D=nfdisc(P); \\ discriminant
if(omega(D)==omega(f)&sign(D)==-1,fK=rnfconductor(Q,P);
if(component(fK[1][1],1)[1]==f,K=bnfinit(P,1); \\ test fK=f
\\ computation of the genus number:
Df=factor(f);df=matsize(Df)[1];Div=component(Df,1);gK=1;for(j=1,df,
q=Div[j];F=idealfactor(K,q);eq=component(F,1)[1][3];gK=gK*eq);gK=gK/(2*p);
nK=nK+1;H=K.no;h=H/gK;C=h/abs(D)^(eps/2);HH=HH*H;hh=hh*h;CC=CC*C))); 
\\ mean of CC and extrema test:
if(nK!=0,C=CC^(1/nK);if(C>Cm,Cm=C;H=HH^(1/nK);h=hh^(1/nK); \\ means of H,h
print("f=",f,"=",Df," H=",precision(H,1),
" C=",precision(C,1)," omega(f)=",omega(f)," nK=",nK))))}

f=7=[7,1]               H=1         C=0.614788152951264365  omega(f)=1  nK=1
f=31=[31,1]             H=9         C=3.814187916735195017  omega(f)=1  nK=1
f=127=[127,1]           H=65        C=19.36254438931795110  omega(f)=1  nK=1
f=215=[5,1; 43,1]       H=266       C=40.79801435406385917  omega(f)=2  nK=1
f=223=[223,1]           H=301       C=77.89153810527680943  omega(f)=1  nK=1
f=439=[439,1]           H=405       C=88.47873262719499548  omega(f)=1  nK=1
f=527=[17,1; 31,1]      H=1026      C=142.1378046706661584  omega(f)=2  nK=1
f=607=[607,1]           H=832       C=167.6201224272746063  omega(f)=1  nK=1
f=919=[919,1]           H=1273      C=231.2062265704086934  omega(f)=1  nK=1
f=1063=[1063,1]         H=3211      C=562.3503754553501155  omega(f)=1  nK=1
f=1399=[1399,1]         H=5184      C=847.6382852373517913  omega(f)=1  nK=1
f=1535=[5,1; 307,1]     H=13034     C=1222.975132851082342  omega(f)=2  nK=1
f=1831=[1831,1]         H=46417     C=7095.859076727868026  omega(f)=1  nK=1
f=4219=[4219,1]         H=136080    C=16884.64442622664510  omega(f)=1  nK=1
f=7295=[5,1; 1459,1]    H=378560    C=24057.21799929730949  omega(f)=2  nK=1
f=8287=[8287,1]         H=233415    C=24464.11035100730691  omega(f)=1  nK=1
f=9127=[9127,1]         H=355167    C=36337.12762910420034  omega(f)=1  nK=1
f=9511=[9511,1]         H=9535041   C=965530.5493702800350  omega(f)=1  nK=1
f=32359=[32359,1]       H=58428335  C=4356369.789422838646  omega(f)=1  nK=1
f=82351=[82351,1]       H=131613755 C=7769338.805955799822  omega(f)=1  nK=1
f=97039=[97039,1]       H=152805744 C=8657711.338391957094  omega(f)=1  nK=1
f=110359=[110359,1]     H=267362135 C=14668913.96712886411  omega(f)=1  nK=1
f=110735=[5,1; 22147,1] H=749451690 C=24129005.39563081245  omega(f)=2  nK=1
(...)
\end{verbatim}
\normalsize

For instance, the field of conductor $f=9511$ is such that
$\order \Cl_K = 9535041 = 69 \cdot 73 \cdot (3 \cdot 631)$,
where the classe number of the quadratic (resp. cubic) field
is $69$ (resp. $73$), the large number of ``relative'' classes being 
$1893 = 3 \cdot 631$.

\subsubsection{$G \simeq \Z/6\Z$. Successive maxima of the 
mean of $\frac{H}{(\sqrt{D}\,)^\varepsilon}$.}
Without the genus factor, we obtain supplementary
conductors, as $f=46631$ for which $\omega(f) = 3$
and $f=62909$ with the first $\omega(f) = 4$:

\footnotesize
\begin{verbatim}
{eps=0.1;Cm=0;Q=bnfinit(y-1,1);p=3;
for(f=1,10^6,phi=eulerphi(f);if(Mod(phi,2*p)!=0,next);
f2=valuation(f,2);fp=valuation(f,p);ff=f/(2^f2*p^fp);
if(f2==1||f2>3||fp>2||core(ff)!=ff,next);
V=Mat(polsubcyclo(f,2*p));dV=matsize(V)[2];nK=0;HH=1;hh=1;CC=1.0;
for(k=1,dV,P=component(V,k)[1];D=nfdisc(P);
if(omega(D)==omega(f)&sign(D)==-1,fK=rnfconductor(Q,P);
if(component(fK[1][1],1)[1]==f,K=bnfinit(P,1);
nK=nK+1;H=K.no;C=H/abs(D)^(eps/2);HH=HH*H;CC=CC*C)));
if(nK!=0,C=CC^(1/nK);if(C>Cm,Cm=C;H=HH^(1/nK);
print("f=",f,"=",factor(f)," H=",precision(H,1),
" C=",precision(C,1)," omega(f)=",omega(f)," nK=",nK))))}
f=7=[7,1]               H=1         C=0.614788152951264365  omega(f)=1  nK=1
f=31=[31,1]             H=9         C=3.814187916735195017  omega(f)=1  nK=1
f=117=[3,2; 13,1]       H=13.637704533357613719 
                                    C=4.854098763584681809  omega(f)=2  nK=5
f=127=[127,1]           H=65        C=19.36254438931795110  omega(f)=1  nK=1
f=215=[5,1; 43,1]       H=266       C=81.59602870812771834  omega(f)=2  nK=1
f=439=[439,1]           H=405       C=88.47873262719499548  omega(f)=1  nK=1
f=527=[17,1; 31,1]      H=1026      C=284.2756093413323168  omega(f)=2  nK=1
f=815=[5,1; 163,1]      H=1440      C=316.5702435490889628  omega(f)=2  nK=1
f=1023=[3,1;11,1;31,1]  H=2416      C=606.0094764864390741  omega(f)=3  nK=1
f=1399=[1399,1]         H=5184      C=847.6382852373517913  omega(f)=1  nK=1
f=1535=[5,1; 307,1]     H=13034     C=2445.950265702164685  omega(f)=2  nK=1
f=1831=[1831,1]         H=46417     C=7095.859076727868026  omega(f)=1  nK=1
f=4219=[4219,1]         H=136080    C=16884.64442622664510  omega(f)=1  nK=1
f=7295=[5,1; 1459,1]    H=378560    C=48114.43599859461900  omega(f)=2  nK=1
f=9511=[9511,1]         H=9535041   C=965530.5493702800350  omega(f)=1  nK=1
f=32359=[32359,1]       H=58428335  C=4356369.789422838646  omega(f)=1  nK=1
f=82351=[82351,1]       H=131613755 C=7769338.805955799822  omega(f)=1  nK=1
f=97039=[97039,1]       H=152805744 C=8657711.338391957094  omega(f)=1  nK=1
f=110359=[110359,1]     H=267362135 C=14668913.96712886411  omega(f)=1  nK=1
f=110735=[5,1; 22147,1] H=749451690 C=48258010.79126162491  omega(f)=2  nK=1
(...)
\end{verbatim}
\normalsize

\subsection{Degree $10$ imaginary cyclic fields of conductor $f$}
${}$
\subsubsection{$G \simeq \Z/10\Z$. Successive maxima of the mean of $\frac{H}
{\raise1.6pt \hbox{\footnotesize $g_{K/\Q}^{}$} \cdot (\sqrt{D}\,)^\varepsilon}$}
${}$
\footnotesize
\begin{verbatim}
f=11=[11,1]           H=1           C=0.339917318183292738  omega(f)=1  nK=1
f=31=[31,1]           H=3           C=0.639747345421760414  omega(f)=1  nK=1
f=71=[71,1]           H=7           C=1.028091534474586170  omega(f)=1  nK=1
f=77=[7,1; 11,1]      H=5           C=1.177981365155155362  omega(f)=2  nK=1
f=88=[2,3; 11,1]      H=10.488088481701515470 
                                    C=1.591537796213539734  omega(f)=2  nK=2
f=123=[3,1; 41,1]     H=22.627416997969520780 
                                    C=2.508456244067394871  omega(f)=2  nK=2
f=124=[2,2; 31,1]     H=41          C=7.340475295096265777  omega(f)=2  nK=1
f=151=[151,1]         H=1967        C=205.7144190804658850  omega(f)=1  nK=1
f=431=[431,1]         H=161931      C=10563.67974232756675  omega(f)=1  nK=1
f=1091=[1091,1]       H=262667      C=11281.95533783001965  omega(f)=1  nK=1
f=1255=[5,1; 251,1]   H=426432      C=11863.80817047845353  omega(f)=2  nK=1
f=2111=[2111,1]       H=411649      C=13137.33757891912240  omega(f)=1  nK=1
f=2351=[2351,1]       H=22273713    C=677218.6915129510617  omega(f)=1  nK=1
f=3931=[3931,1]       H=43298816    C=1044599.109023261171  omega(f)=1  nK=1
f=5051=[5051,1]       H=704444539   C=15181930.21787324823  omega(f)=1  nK=1
(...)
\end{verbatim}
\normalsize

\subsubsection{$G \simeq \Z/10\Z$. Successive maxima of the 
mean of  $\frac{H}{(\sqrt{D}\,)^\varepsilon}$}
${}$
\footnotesize
\begin{verbatim}
f=11=[11,1]             H=1          C=0.339917318183292738 omega(f)=1 nK=1
f=31=[31,1]             H=3          C=0.639747345421760414 omega(f)=1 nK=1
f=55=[5,1; 11,1]        H=4          C=0.909265644114650678 omega(f)=2 nK=1
f=71=[71,1]             H=           C=1.028091534474586170 omega(f)=1 nK=1
f=77=[7,1; 11,1]        H=5          C=1.177981365155155362 omega(f)=2 nK=1
f=88=[2,3; 11,1]        H=10.488088481701515470 
                                     C=2.250774336434575013 omega(f)=2 nK=2
f=123=[3,1; 41,1]       H=22.627416997969520780 
                                     C=3.547492840979584473 omega(f)=2 nK=2
f=124=[2,2; 31,1]       H=41         C=7.340475295096265777 omega(f)=2 nK=1
f=132=[2,2; 3,1; 11,1]  H=44         C=8.035828447368965852 omega(f)=3 nK=1
f=143=[11,1; 13,1]      H=50         C=8.950709414213818126 omega(f)=2 nK=1
f=151=[151,1]           H=1967       C=205.7144190804658850 omega(f)=1 nK=1
f=431=[431,1]           H=161931     C=10563.67974232756675 omega(f)=1 nK=1
f=1091=[1091,1]         H=262667     C=11281.95533783001965 omega(f)=1 nK=1
f=1255=[5,1; 251,1]     H=426432     C=23727.61634095690706 omega(f)=2 nK=1
f=2351=[2351,1]         H=22273713   C=677218.6915129510617 omega(f)=1 nK=1
f=3931=[3931,1]         H=43298816   C=1044599.109023261171 omega(f)=1 nK=1
f=5051=[5051,1]         H=704444539  C=15181930.21787324823 omega(f)=1 nK=1
f=8831=[8831,1]         H=9494084864 C=159128753.6488153656 omega(f)=1 nK=1
(...)
\end{verbatim}
\normalsize

\subsection{Some additional data}
We indicate the numerical results for some other abelian groups 
(limited by PARI/GP to groups $G$ of order less than $12$),
for the mean of $\frac{H} {\raise1.6pt \hbox{\footnotesize $g_{K/\Q}^{}$} 
\cdot (\sqrt{D}\,)^\varepsilon}$ with $\varepsilon=0.01$.
These results strengthen somewhat the conjectural inequality
$r(G) \leq \omega(f) \leq R(G)$ of Conjecture \ref{genconj}.
\subsubsection{$G=\Z/2\Z \times \Z/4\Z$.}
${}$
\footnotesize
\begin{verbatim}
f=15=[3,1; 5,1]      H=1       h=1      C=0.9321557326415818 omega(f)=2 nK=1
f=39=[3,1; 13,1]     H=2       h=2      C=1.8116290245333817 omega(f)=2 nK=1
f=51=[3,1; 17,1]     H=5       h=5      C=4.4927692286025652 omega(f)=2 nK=1
f=123=[3,1; 41,1]    H=17      h=17     C=14.877260946454665 omega(f)=2 nK=1
f=187=[11,1; 17,1]   H=17      h=17     C=14.883586644228593 omega(f)=2 nK=1
f=272=[2,4; 17,1] 
                H=43.886403 h=17.416330 C=14.929910722942526 omega(f)=2 nK=9
f=287=[7,1; 41,1]    H=35      h=35     C=30.114979221003303 omega(f)=2 nK=1
f=291=[3,1; 97,1]    H=68      h=68     C=57.991365573835759 omega(f)=2 nK=1
f=327=[3,1; 109,1]   H=102     h=102    C=86.683203264487910 omega(f)=2 nK=1
f=391=[17,1; 23,1]   H=105     h=105    C=90.581868448183013 omega(f)=2 nK=1
f=411=[3,1; 137,1]   H=111     h=111    C=93.686914633082859 omega(f)=2 nK=1
f=451=[11,1; 41,1]   H=111     h=111    C=94.648036619685123 omega(f)=2 nK=1
f=511=[7,1; 73,1]    H=791     h=791    C=668.92101009616677 omega(f)=2 nK=1
f=623=[7,1; 89,1]    H=1199    h=1199   C=1007.9419532931184 omega(f)=2 nK=1
f=791=[7,1; 113,1]   H=1280    h=1280   C=1068.3551688844457 omega(f)=2 nK=1
f=1371=[3,1; 457,1]  H=5340    h=5340   C=4347.1171682925243 omega(f)=2 nK=1
f=1679=[23,1; 73,1]  H=8268    h=8268   C=6827.5710056916921 omega(f)=2 nK=1
f=1799=[7,1; 257,1]  H=17475   h=17475  C=14230.404002409721 omega(f)=2 nK=1
f=3027=[3,1; 1009,1] H=27300   h=27300  C=21702.186282280434 omega(f)=2 nK=1
f=3431=[47,1; 73,1]  H=36465   h=36465  C=29684.830601395338 omega(f)=2 nK=1
f=3503=[31,1; 113,1] H=43836   h=43836  C=35514.960098930545 omega(f)=2 nK=1
f=4372=[2,2; 1093,1] H=60125   h=60125  C=47408.398500381492 omega(f)=2 nK=1
f=4616=[2,3; 577,1] 
         H=129838.811331 h=77202.619119 C=60564.691595960546 omega(f)=2 nK=4
f=5183=[71,1; 73,1]  H=89600   h=89600  C=72340.774959541852 omega(f)=2 nK=1
f=5311=[47,1; 113,1] H=112355  h=112355 C=90273.045356601456 omega(f)=2 nK=1
f=7031=[79,1; 89,1]  H=197100  h=197100 C=157852.83359289318 omega(f)=2 nK=1
f=7251=[3,1; 2417,1] H=267325  h=267325 C=207013.58713575608 omega(f)=2 nK=1
f=8116=[2,2; 2029,1] H=502775  h=502775 C=389147.29273104160 omega(f)=2 nK=1
f=10376=[2,3; 1297,1] 
       H=1048152.694144 h=623235.320742 C=477184.80861414767 omega(f)=2 nK=4
f=13271=[23,1;577,1] H=5225220 h=5225220C=4055409.8285235355 omega(f)=2 nK=1
\end{verbatim}
\normalsize
\subsubsection{$G=\Z/2\Z \times \Z/3\Z \times \Z/5Z$.}
${}$
\footnotesize
\begin{verbatim}
f=31=[31,1]       H=9 h=9            C=0.061911033427912298 omega(f)=1 nK=1
f=77=[7,1; 11,1]  H=35.777087 h=35.777087         
                                     C=0.188405690668671425 omega(f)=2 nK=2
f=93=[3,1; 31,1]  H=755 h=755        C=2.705201622444429614 omega(f)=2 nK=1
f=124=[2,2; 31,1] H=5084 h=5084      C=14.68095059019253155 omega(f)=2 nK=1
f=151=[151,1]     H=238007 h=238007  C=164.8440046936183581 omega(f)=1 nK=1
f=211=[211,1]     H=467523 h=467523  C=199.3395258269297336 omega(f)=1 nK=1
f=244=[2,2; 61,1] H=900747.730538 h=636924.828402 
                                     C=643.3512814804028363 omega(f)=2 nK=2
f=248=[2,3; 31,1] H=731438.520484 h=517205.137855 
                                     C=814.9949665105118220 omega(f)=2 nK=2
f=331=[331,1]     H=10805967 h=10805967 
                                     C=2398.352674944173664 omega(f)=1 nK=1
\end{verbatim}
\normalsize

\subsubsection{$G=\Z/9\Z$.}
${}$
\footnotesize
\begin{verbatim}
f=19=[19,1]      H=1     h=1     C=0.888893756893086074  omega(f)=1 nK=1
f=163=[163,1]    H=4     h=4     C=3.262665022855041198  omega(f)=1 nK=1
f=1063=[1063,1]  H=13    h=13    C=9.837438305141874497  omega(f)=1 nK=1
f=1141=[7,1; 163,1] 
        H=52.306787 h=17.435595  C=13.41517512486724241  omega(f)=2 nK=2
f=1153=[1153,1]  H=19    h=19    C=14.33112977180156483  omega(f)=1 nK=1
f=1459=[1459,1]  H=247   h=247   C=184.5587931331744690  omega(f)=1 nK=1
f=3547=[3547,1]  H=16777 h=16777 C=12098.17485260014863  omega(f)=1 nK=1
\end{verbatim}
\normalsize

\subsubsection{$G=\Z/8\Z$.}
${}$
\footnotesize
\begin{verbatim}
f=32=[2,5]         H=1        h=1        C=0.898132372883934 omega(f)=1 nK=1
f=68=[2,2;17,1]    H=4        h=2        C=1.761664016490415 omega(f)=2 nK=1
f=73=[73,1]        H=89       h=89       C=76.59028869662640 omega(f)=1 nK=1
f=89=[89,1]        H=113      h=113      C=96.57168147147798 omega(f)=1 nK=1
f=233=[233,1]      H=1433     h=1433     C=1184.100959851340 omega(f)=1 nK=1
f=339=[3,1;113,1]  H=4274     h=2137     C=1771.757792528851 omega(f)=2 nK=1
f=601=[601,1]      H=34153    h=34153    C=27300.35230848298 omega(f)=1 nK=1
f=1609=[1609,1]    H=189473   h=189473   C=146324.7766845321 omega(f)=1 nK=1
f=1731=[3,1;577,1] H=460656   h=230328   C=180369.5234991963 omega(f)=2 nK=1
f=2441=[2441,1]    H=387433   h=387433   C=294870.7790203549 omega(f)=1 nK=1
f=3449=[3449,1]    H=549225   h=549225   C=412981.9168610820 omega(f)=1 nK=1
f=4201=[4201,1]    H=647593   h=647593   C=483598.3404675706 omega(f)=1 nK=1
f=4297=[4297,1]    H=1747657  h=1747657  C=1304053.513066895 omega(f)=1 nK=1
f=4409=[4409,1]    H=8637921  h=8637921  C=6439576.759445647 omega(f)=1 nK=1
f=7753=[7753,1]    H=98996475 h=98996475 C=72358288.91356542 omega(f)=1 nK=1
\end{verbatim}
\normalsize

\smallskip
We have no justification for these results and hope that analytic tools may help. 
In this direction see \cite{Da,DaKM,Du,GS, Lam,L0,L1,L2,R} among a wide literature 
on the subject in the spirit of Brauer--Siegel theorems and 
Cohen--Lenstra--Martinet--Malle heuristics.

\newpage
\section{Appendix -- Numerical tables for quadratic fields}
\subsection{Table I: Imaginary case -- Successive maxima of 
$\frac{H}{2^{N-1}(\sqrt D\,)^\varepsilon}$} \label{Table1}
${}$
\footnotesize

\normalsize 

\newpage
\subsection{Table II: Imaginary case -- Successive maxima of 
$\frac{H}{(\sqrt D\,)^\varepsilon}$} \label{Table2}
${}$
\footnotesize
%
\normalsize

\medskip
These two tables on imaginary quadratic fields may suggest some regularities 
depending on the context of definition of the successive maxima; for the non-genus
part of the class number and $C_\varepsilon(D) = \frac{H}{2^{N-1} \cdot 
(\sqrt D)^\varepsilon}$, the
Conjecture \ref{genconj} is verified, at least for $D \in [3,\, 2 \times 10^9]$,
since we have not found any composite discriminant, and in the case using the 
whole class number $\frac{H}{(\sqrt D)^\varepsilon}$, probabilistic densities 
for $\omega(D)$ seem to exist in a different manner than that predicted by the
Conjecture on the non-genus part.

\newpage
\subsection{Table III: Real case -- Successive maxima of 
$\frac{H}{2^{N-1}(\sqrt D\,)^\varepsilon}$} \label{Table3}
The real quadratic fields give rise to the same results and
comments as for the imaginary case. 

\smallskip
The computations are
slower and $H$ must be computed in the restricted sense
(two days of computation to obtain the last discriminant 
$D_K=34574401$ of the following list).

\smallskip
Due to the particular case $D_K=136$, giving a local maximum,
we get $N_2=1$ at the end of the computations, which 
is an excellent verification of the conjecture:

\smallskip
\footnotesize
\begin{verbatim}
{eps=0.02;bD=2;BD=10^10;ND=0;N1=0;N2=0;N3=0;
Cm=0;for(D=bD,BD,e2=valuation(D,2);d=D/2^e2;
if(e2==1||e2>3||core(d)!=d||(e2==0&Mod(d,4)!=1)||(e2==2&Mod(d,4)!=-1),
next);ND=ND+1;P=x^2-D;K=bnfinit(P,1);H=bnfnarrow(K)[1];
N=omega(D);h=H/2^(N-1);C=h/D^(eps/2);if(C>Cm,Cm=C;
if(N==1,N1=N1+1);if(N==2,N2=N2+1);if(N>=3,N3=N3+1);print();
print("D_K=",D,"   H=",H,"   h=",h,"   N=",N,"   C=",precision(C,1));
print("ND=",ND,"   N1=",N1,"   N2=",N2,"   N3=",N3)))}

eps=0.02
D_K=5   H=1   h=1   N=1   C=0.9840344433634576028
ND=1     N1=1     N2=0     N3=0

D_K=136   H=4   h=2   N=2   C=1.9041212797772518802
ND=40     N1=1     N2=1     N3=0

D_K=229   H=3   h=3   N=1   C=2.841338001830550898
ND=69     N1=2     N2=1     N3=0

D_K=401   H=5   h=5   N=1   C=4.709107022186433819
ND=121     N1=3     N2=1     N3=0

D_K=577   H=7   h=7   N=1   C=6.568803671814335481
ND=177     N1=4     N2=1     N3=0

D_K=1129   H=9   h=9   N=1   C=8.389103838263256518
ND=342     N1=5     N2=1     N3=0

D_K=1297   H=11   h=11   N=1   C=10.239135384664958651
ND=397     N1=6     N2=1     N3=0

D_K=7057   H=21   h=21   N=1   C=19.219102083164795610
ND=2143     N1=7     N2=1     N3=0

D_K=8761   H=27   h=27   N=1   C=24.656886007086108242
ND=2662     N1=8     N2=1     N3=0

D_K=14401   H=43   h=43   N=1   C=39.07369921920168190
ND=4379     N1=9     N2=1     N3=0

D_K=32401   H=45   h=45   N=1   C=40.56083898547825383
ND=9848     N1=10     N2=1     N3=0

D_K=41617   H=57   h=57   N=1   C=51.24861676322140264
ND=12651     N1=11     N2=1     N3=0

D_K=57601   H=63   h=63   N=1   C=56.45939878831204140
ND=17509     N1=12     N2=1     N3=0

D_K=90001   H=87   h=87   N=1   C=77.62056135842406069
ND=27358     N1=13     N2=1     N3=0

D_K=176401   H=153   h=153   N=1   C=135.58961275598587437
ND=53621     N1=14     N2=1     N3=0

D_K=650281   H=207   h=207   N=1   C=181.06701198422100325
ND=197663     N1=15     N2=1     N3=0

D_K=921601   H=235   h=235   N=1   C=204.84361566477445036
ND=280137     N1=16     N2=1     N3=0

D_K=1299601   H=357   h=357   N=1   C=310.1202431051597330
ND=395027     N1=17     N2=1     N3=0

D_K=2944657   H=377   h=377   N=1   C=324.8261637468532506
ND=895068     N1=18     N2=1     N3=0

D_K=3686401   H=455   h=455   N=1   C=391.1518340310603556
ND=1120538     N1=19     N2=1     N3=0

D_K=3920401   H=575   h=575   N=1   C=494.0086344963980478
ND=1191666     N1=20     N2=1     N3=0

D_K=7290001   H=655   h=655   N=1   C=559.2603422239048551
ND=2215899     N1=21     N2=1     N3=0

D_K=10497601   H=1061   h=1061   N=1   C=902.6190252870294034
ND=3190902     N1=22     N2=1     N3=0

D_K=15210001   H=1245   h=1245   N=1   C=1055.2322581637658160
ND=4623288     N1=23     N2=1     N3=0

D_K=34574401   H=1539   h=1539   N=1   C=1293.7521162120423479
ND=10509357    N1=24     N2=1     N3=0
(...)
\end{verbatim}

\end{document}